\documentclass[a4paper,12pt]{article}
\usepackage{tikz}
\usetikzlibrary{arrows}
\usetikzlibrary{calc}
\usetikzlibrary{positioning}
\usetikzlibrary{cd}
\usepackage{graphicx,amsmath,amssymb,amsthm,amscd,mathrsfs,mathabx,mathtools,authblk}
\usepackage[%
 setpagesize=false,%
 bookmarks=true,%
 bookmarksdepth=tocdepth,%
 bookmarksnumbered=true,%
 colorlinks=true,%
 hidelinks=true,%
 pdftitle={},%
 pdfsubject={},%
 pdfauthor={},%
 pdfkeywords={}%
]{hyperref}
\usepackage[margin=2.5cm]{geometry}

\newcommand{\C}{\mathbb{C}}

\newtheorem{thm}{Theorem}[section]
\newtheorem{prop}[thm]{Proposition}
\newtheorem{lem}[thm]{Lemma}
\newtheorem{cor}[thm]{Corollary}

\theoremstyle{definition}
\newtheorem{df}[thm]{Definition}

\theoremstyle{remark}
\newtheorem{rem}[thm]{Remark}

\title{Middle convolution for Lie algebra representations 
}
\date{}

\author{Kazuki Hiroe\\
Department of Mathematics and Informatics, Chiba University\\
1-33, Yayoi-cho, Inage-ku, Chiba-shi, Chiba, 263-8522 JAPAN\\
email: {\tt kazuki@math.s.chiba-u.ac.jp}
}
\begin{document}
\maketitle
\begin{abstract}
This paper introduces a Lie algebra analogue of the middle convolution functor, 
which is defined on the category of modules over certain Lie algebras, 
including, as typical motivating examples, free Lie algebras, Drinfeld-Kohno Lie algebras, 
and the holonomy Lie algebras of complements of hyperplane arrangements. 
First, we demonstrate that the middle convolution for Lie algebra representations can be regarded as 
a natural generalization of the infinitesimal analogue of the Long-Moody functor for Drinfeld-Kohno Lie algebras. 
Second, we show that our middle convolution recovers the Dettweiler-Reiter additive middle convolution 
for Fuchsian systems on the punctured Riemann sphere as a special case. 
Furthermore, we show that when applied to the holonomy Lie algebra of the complement of a hyperplane arrangement, 
our functor is compatible with Haraoka's middle convolution for logarithmic connections on such complements. 
Finally, we establish a Riemann-Hilbert correspondence between the middle convolution 
for the holonomy Lie algebra and the middle convolution for local systems on complements of hyperplane arrangements.
\end{abstract}
\tableofcontents
\section{Introduction}
The middle convolution functor was originally introduced by Katz \cite{Katz} for  
local systems over the punctured Riemann sphere, and 
his theory of rigid local systems and the middle convolution has become a powerful tool in the study of ordinary differential equations.
Through the work of Dettweiler and Reiter \cite{DR00,DR07}, 
the middle convolution was successfully translated into an algebraic operation 
on Fuchsian systems on the punctured Riemann sphere. 
Subsequently, Crawley-Boevey \cite{CB03} established a connection to 
the representation theory of quivers, applying it to solve the additive 
Deligne-Simpson problem. 
Furthermore, Arinkin \cite{Ari10} extended the theory of rigid local 
systems to differential equations with irregular singularities,
and algebraic operations corresponding to middle convolution for such equations were also introduced by 
Kawakami \cite{Kaw10}, Takemura \cite{Tak11}, and Yamakawa \cite{Yam11}.
As a generalization to the case of several variables, Haraoka \cite{Har1} 
constructed the middle convolution functor for logarithmic connections 
on complements of hyperplane arrangements. 
Subsequently, a natural generalization of middle convolution 
for local systems on such complements was given 
by Adachi and the author \cite{AdaHir}. 
More recently, Adachi \cite{Ada25} introduced the middle convolution functor 
for irregular singular connections in this setting.

The main purpose of this paper is to introduce a Lie algebra analogue of 
the middle convolution functor, which is defined on the category of 
modules over certain Lie algebras. 
The Lie algebras appearing in this context include, 
as typical examples, free Lie algebras, Drinfeld-Kohno Lie algebras, 
and more generally, the holonomy Lie algebras of complements 
of hyperplane arrangements.

We then show that this middle convolution for Lie algebra representations 
is closely related to several known versions of middle convolution. 
These include Dettweiler and Reiter's middle convolution for Fuchsian systems 
on the punctured Riemann sphere, Haraoka's middle convolution for logarithmic 
connections on complements of hyperplane arrangements, 
as well as Katz's middle convolution for local systems on the punctured Riemann 
sphere and its generalization to complements of hyperplane arrangements 
by Adachi and the author.

Furthermore, we demonstrate that the middle convolution for Lie algebra 
representations can be regarded as a natural generalization of the Long-Moody 
functor for braid group representations, 
introduced by Long and Moody \cite{Lo94}. 
This functor serves as a higher-dimensional analogue of the classical 
Burau representation. 
In recent years, the categorical aspects of this construction 
have been extensively studied; 
in particular, Souli\'e \cite{Sou19} investigated it 
from the perspective of functor category theory.

Below, we provide a more detailed overview of the contents of this paper.
\subsection{Middle convolution for Lie algebra representations}
First, we introduce the definition of the middle convolution for Lie algebra representations.
The Drinfeld-Kohno Lie algebra $\mathfrak{P}_{n}$ is an infinitesimal analogue of the pure braid group $P_{n}$,
which is defined as the Lie algebra over a field $K$ generated by $A_{i,j}$ for $1\le i\neq j\le n$ 
with the infinitesimal braid relation:
\begin{align*}
A_{i,j}-A_{j,i}&=0, \\
[A_{i,k},A_{i,j}+A_{j,k}]&=0,\\
[A_{i,j},A_{k,l}]&=0,
\end{align*}
for mutually distinct $i,j,k,l$.
As an infinitesimal analogue of the decomposition of the pure braid group $P_{n+1}$,
\(
	P_{n+1}=F_{n}\rtimes P_{n},	
\)
the Drinfeld-Kohno Lie algebra $\mathfrak{P}_{n+1}$ also has the semidirect sum decomposition
\[
	\mathfrak{P}_{n+1}=\mathbb{L}(A_{1,n+1},\ldots,A_{n,n+1})\oplus \mathfrak{P}_{n}.
\]
Here we denote the free Lie algebra generated by $x_{1},\ldots,x_{n}$
by $\mathbb{L}(x_{1},\ldots,x_{n})$.
Then as a generalization of the Drinfeld-Kohno Lie algebra,
we 
consider the semidirect sum of the free Lie algebra $\mathbb{L}_{n}=\mathbb{L}(x_{1},\ldots,x_{n})$
and a Lie algebra $\mathfrak{g}$,
\begin{equation}\label{eq:generalizedDK}
\mathbb{L}_{n}\oplus_{\theta_{g}} \mathfrak{g},
\end{equation}
associated with the following homomorphism $\theta_{g}\colon \mathfrak{g}\rightarrow \mathrm{Der}(\mathbb{L}_{n})$.
The adjoint action of $\mathfrak{P}_{n}$ on $\mathbb{L}(A_{1,n+1},\ldots,A_{n,n+1})$ gives
the Lie algebra homomorphism $\theta\colon \mathfrak{P}_{n}\to \mathrm{Der}(\mathbb{L}_{n})$ to the derivation Lie algebra $\mathrm{Der}(\mathbb{L}_{n})$
of $\mathbb{L}_{n}$.
Given a Lie algebra $\mathfrak{g}$ with a Lie algebra homomorphism $g\colon \mathfrak{g}\to \mathfrak{P}_{n}$,
the composition map $\theta_{g}=\theta\circ g\colon \mathfrak{g}\to \mathrm{Der}(\mathbb{L}_{n})$ determines 
the semidirect sum Lie algebra $\mathbb{L}_{n}\oplus_{\theta_{g}} \mathfrak{g}$.

Notice that this Lie algebra recovers the Drinfeld-Kohno Lie algebra $\mathfrak{P}_{n+1}$ if we take $\mathfrak{g}=\mathfrak{P}_{n}$ and $g=\mathrm{id}$.

Then the middle convolution functor is defined as a functor on the category 
$U(\mathbb{L}_{n}\oplus_{\theta_{g}} \mathfrak{g})\text{-}\mathbf{mod}$
of finite dimensional modules over the universal enveloping algebra $U(\mathbb{L}_{n}\oplus_{\theta_{g}} \mathfrak{g})$,
in terms of Lie algebra homology.
\begin{df}[Middle convolution for $U(\mathbb{L}_{n}\oplus_{\theta_{g}} \mathfrak{g})$-modules]\label{df:liemc0}
	For $\lambda\in K\setminus\{0\}$,
	the \emph{middle convolution functor} is defined as follows: 
	\begin{align*}
	\mathfrak{mc}_{\lambda}\colon& U(\mathbb{L}_{n}\oplus_{\theta_{g}} \mathfrak{g})\text{-}\mathbf{mod}\longrightarrow 
	U(\mathbb{L}_{n}\oplus_{\theta_{g}} \mathfrak{g})\text{-}\mathbf{mod}\\
	& V\longmapsto 
	\mathrm{Coker}\left(
		\bigoplus_{i=0}^{n}H_{1}(L(x_{i}),V\otimes_{K}K_{\lambda})
		\to H_{1}(\mathbb{L}_{n+1},V\otimes_{K}K_{\lambda}) 
	\right).
	\end{align*}
	Here $K_{\lambda}$
is the one-dimensional module over the Lie algebra $\mathbb{L}(x_{n+1})$ defined by $x_{n+1}\cdot 1=\lambda\cdot 1$,
and we put $x_{0}=-\sum_{i=1}^{n+1} x_i$.
We regard $V\otimes_{K}K_{\lambda}$ as a module over $\mathbb{L}_{n+1}$
by the projections $\mathbb{L}_{n+1}\rightarrow \mathbb{L}_{n}$ and 
$\mathbb{L}_{n+1}\to \mathbb{L}(x_{n+1})$.
\end{df}
It will be explained in Section \ref{sec:midconv} 
how $\mathfrak{mc}_{\lambda}(V)$
has a structure of a $U(\mathbb{L}_{n}\oplus_{\theta_{g}} \mathfrak{g})$-module, and 
moreover
we will give a modified definition of $\mathfrak{mc}_{\lambda}$
which is valid even for $\lambda=0$.

The composition law of the middle convolution is a fundamental 
property which was first established by Katz \cite{Katz},
and plays a crucial role in the applications of middle convolution to the study of ordinary differential equations.
In Theorems \ref{thm:mc0} and \ref{thm:compmc2}, we will show that $\mathfrak{mc}_{\lambda}$ 
behaves as an invertible functor, 
and preserves irreducibility for generic modules. 
That is, for modules satisfying certain conditions, we have
\[
	\mathfrak{mc}_{\lambda}\circ\mathfrak{mc}_{\mu}(V)\cong \mathfrak{mc}_{\lambda+\mu}(V),\quad \mathfrak{mc}_{-\lambda}\circ\mathfrak{mc}_{\lambda}(V)\cong V.
\]
Moreover if $V$ is irreducible as a $U(\mathbb{L}_{n})$-module, then $\mathfrak{mc}_{\lambda}(V)$ is again irreducible as a $U(\mathbb{L}_{n})$-module and hence also irreducible as a $U(\mathbb{L}_{n}\oplus_{\theta_{g}}\mathfrak{g})$-module,
see Corollary \ref{cor:irred} for details.

\subsection{Long-Moody functor for Drinfeld-Kohno Lie algebra representations}
The Long-Moody functor is a procedure to construct representations of braid groups from representations of larger braid groups,
which was introduced by Long and Moody in \cite{Lo94} as a higher dimensional analogue of the classical Burau representation of braid groups.
Let $K[B_{n}]$ be the group ring generated by the braid group $B_{n}$ of $n$ strands over $K$.
Then the Long-Moody functor is defined by
\[
	\mathcal{LM}\colon K[B_{n+1}]\text{-}\mathrm{\mathbf{mod}}\to K[B_{n}]\text{-}\mathrm{\mathbf{mod}}; \quad V\mapsto \mathcal{I}_{F_{n}}\otimes_{K[F_{n}]}V,
\]
where $\mathcal{I}_{F_{n}}$ is the augmentation ideal of the group ring $K[F_{n}]$ generated by the free group $F_{n}$ with $n$ generators over $K$.
Then as a natural generalization of it, we consider the infinitesimal version of the Long-Moody functor given by 
\[
	\mathfrak{LM}\colon U(\mathfrak{P}_{n+1})\text{-}\mathrm{\mathbf{mod}}\to U(\mathfrak{P}_{n})\text{-}\mathrm{\mathbf{mod}}; \quad V\mapsto \mathcal{I}_{\mathbb{L}_{n}}\otimes_{U(\mathbb{L}_{n})}V,
\]
where $\mathcal{I}_{\mathbb{L}_{n}}$ is the augmentation ideal of the universal enveloping algebra $U(\mathbb{L}_{n})$.

In Section \ref{sec:longmoody}, we will introduce a one parameter deformation $\mathfrak{LM}_{\lambda}^{\mathrm{df}}$ of the infinitesimal version of the Long-Moody functor.
\begin{df}[Deformed Long-Moody functor]
For $\lambda\in K\setminus\{0\}$, the deformed Long-Moody functor $\mathfrak{LM}_{\lambda}^{\mathrm{df}}$ is defined by
\[
	\mathfrak{LM}^{\mathrm{df}}_{\lambda}\colon U(\mathbb{L}_{n}\oplus_{\theta_{g}} \mathfrak{g})\text{-}\mathbf{mod}\to 
	U(\mathbb{L}_{n}\oplus_{\theta_{g}} \mathfrak{g})\text{-}\mathbf{mod}; \quad V\mapsto H_{1}(\mathbb{L}_{n+1},V\otimes_{K}K_{\lambda}).
\]
\end{df}
Although we supposed that $\lambda\neq 0$ here, we will give a modified definition of $\mathfrak{LM}_{\lambda}^{\mathrm{df}}$ which is valid even for $\lambda=0$.
Then we will show that $\mathfrak{LM}^{\mathrm{df}}_{\lambda}$ is a natural generalization of the infinitesimal version of the Long-Moody functor $\mathfrak{LM}$
as follows.
\begin{prop}[Proposition \ref{prop:compLM}]
Suppose that $\mathfrak{g}=\mathfrak{P}_{n}$ and $g=\mathrm{id}\colon \mathfrak{P}_{n}\to \mathfrak{P}_{n}$.
Then there exists a natural isomorphism of functors
\[
	\mathrm{Res}_{U(\mathfrak{P}_{n})}^{U(\mathfrak{P}_{n+1})}\circ\mathfrak{LM}^{\mathrm{df}}_{0}\cong \mathfrak{LM}.
\]
\end{prop}
Comparing with Definition \ref{df:liemc0},
we see that 
the middle convolution $\mathfrak{mc}_{\lambda}$ factors through the deformed Long-Moody functor $\mathfrak{LM}^{\mathrm{df}}_{\lambda}$.
Based on this background, we can regard the middle convolution $\mathfrak{mc}_{\lambda}$ as a natural generalization of the infinitesimal Long-Moody functor $\mathfrak{LM}$.

\subsection{Dettweiler-Reiter additive middle convolution}
The Dettweiler-Reiter additive middle convolution is an algebraic operation on Fuchsian systems on the punctured Riemann sphere, which was introduced by Dettweiler and Reiter \cite{DR00,DR07} as an algebraic analogue of Katz's middle convolution for local systems on the punctured Riemann sphere.
This operation, however, can be 
regarded as a functor on the category of modules over the free associative algebra
generated by $x_{1},\ldots,x_{n}$ over a field $K$:
\[
	\mathrm{mc}_{\lambda}^{\mathrm{DR}}\colon K\{x_{1},\ldots,x_{n}\}\text{-}\mathrm{\mathbf{mod}}\to K\{x_{1},\ldots,x_{n}\}\text{-}\mathrm{\mathbf{mod}},
\]
see Section 7,5 in \cite{Harbook} for details.
Therefore, the identification $U(\mathbb{L}_{n})\cong K\{x_{1},\ldots,x_{n}\}$
allows us to show that the middle convolution $\mathfrak{mc}_{\lambda}$ for $U(\mathbb{L}_{n}\oplus_{\theta_{g}} \mathfrak{g})$-modules
recovers the Dettweiler-Reiter additive middle convolution $\mathrm{mc}_{\lambda}^{\mathrm{DR}}$
when we specialize $\mathfrak{g}=\{0\}$.
We will moreover show the following compatibility of the middle convolution $\mathfrak{mc}_{\lambda}$ with the 
Dettweiler-Reiter additive middle convolution $\mathrm{mc}_{\lambda}^{\mathrm{DR}}$.
\begin{thm}[Theorem \ref{thm:compmc}]
For $\lambda\in K$, there exists a natural isomorphism of functors
\[
	\mathrm{Res}_{\mathbb{L}_{n}}^{\mathbb{L}_{n}\oplus_{\theta_{g}}\mathfrak{g}}\circ \mathfrak{mc}_{\lambda}
	\cong \mathrm{mc}^{\mathrm{DR}}_{\lambda}\circ \mathrm{Res}_{\mathbb{L}_{n}}^{\mathbb{L}_{n}\oplus_{\theta_{g}}\mathfrak{g}}.
\]
\end{thm}
Here $\mathrm{Res}_{\mathbb{L}_{n}}^{\mathbb{L}_{n}\oplus_{\theta_{g}}\mathfrak{g}}$ is the 
restriction functor from $U(\mathbb{L}_{n}\oplus_{\theta_{g}} \mathfrak{g})$-modules 
to $U(\mathbb{L}_{n})$-modules.
Therefore, the middle convolution $\mathfrak{mc}_{\lambda}$ can be 
regarded as a natural generalization of the Dettweiler-Reiter 
additive middle convolution $\mathrm{mc}_{\lambda}^{\mathrm{DR}}$.
\subsection{Holonomy Lie algebras of complements of hyperplane arrangements}
\label{sec:holonomy}
In Sections \ref{sec:holonomy} and \ref{sec:mcPf},
we will give a geometric interpretation of $\mathfrak{mc}_{\lambda}$
in terms of the middle convolution for logarithmic connections on complements of hyperplane arrangements.

Let $\mathcal{A}$ be an affine complex hyperplane arrangement in $\C^{l}$, 
and $M(\mathcal{A})$ be the complement of $\mathcal{A}$ in $\C^{l}$.
Let us denote the intersection poset 
by $L(\mathcal{A})$. Also $L_{k}(\mathcal{A})$ denotes
the subset of $L(\mathcal{A})$ consisting of codimension $k$ elements.

Then in Section \ref{sec:logPf}, we will 
consider the category $\mathrm{Pf}(\mathrm{log}(\mathcal{A}))$
consisting of flat connections on a trivial bundle over $M(\mathcal{A})$ with logarithmic poles along $\mathcal{A}$.
Then we will recall that
the middle convolution functor can be defined  
on $\mathrm{Pf}(\mathrm{log}(\mathcal{A}))$
which was introduced by Haraoka \cite{Har1} 
as a generalization of the Dettweiler-Reiter additive middle 
convolution for Fuchsian systems on the punctured Riemann sphere.
Haraoka's middle convolution is a functor between the following categories
\[
	\mathrm{mc}_{\lambda}\colon \mathrm{Pf}(\mathrm{log}(\mathcal{A}))\to 
	\mathrm{Pf}(\mathrm{log}(\overline{\mathcal{A}}^{Y})),
\]
which depends on the choice of a line $Y$ in $\C^{l}$ passing through the origin,
and the target category is associated with the hyperplane arrangement 
$\overline{\mathcal{A}}^{Y}$ defined by
\[ 
\overline{\mathcal{A}}^{Y}=\mathcal{A}\cup (Y+L_{2}(\mathcal{A})),
\]
which is called the $Y$-closure of $\mathcal{A}$.

The main purpose of Sections \ref{sec:holonomy} and \ref{sec:mcPf}
is to show that the middle convolution $\mathrm{mc}_{\lambda}$ on 
$\mathrm{Pf}(\mathrm{log}(\mathcal{A}))$ 
is compatible with the middle convolution $\mathfrak{mc}_{\lambda}$ for 
the holonomy Lie algebra $\mathfrak{h}(M(\mathcal{A}))$ of $M(\mathcal{A})$.
Let us give a more precise explanation.
Owing to the explicit description of  
the holonomy Lie algebra $\mathfrak{h}(M(\mathcal{A}))$ of $M(\mathcal{A})$
by Kohno \cite{Koh83},
we can construct an equivalence of categories
\[
	\Xi_{\mathcal{A}}\colon \mathrm{Pf}(\mathrm{log}(\mathcal{A}))
	\xrightarrow{\simeq} U(\mathfrak{h}(M(\mathcal{A})))\text{-}\mathbf{mod}.
\]
Thus we can compare the middle convolutions $\mathrm{mc}_{\lambda}$
and $\mathfrak{mc}_{\lambda}$ by the equivalence $\Xi_{\mathcal{A}}$.
One remarkable point here is that 
the holonomy Lie algebra $\mathfrak{h}(M(\mathcal{A}))$ of $M(\mathcal{A})$ 
does not have the semidirect sum decomposition of the form 
\eqref{eq:generalizedDK} in general, 
however, the result of Cohen, Cohen, and Xicot\'encatl \cite{CCX03} implies 
that the holonomy Lie algebra $\mathfrak{h}(M(\overline{\mathcal{A}}^{Y}))$ 
associated with the $Y$-closure $\overline{\mathcal{A}}^{Y}$
always has such a decomposition,
\[	
	\mathfrak{h}(M(\overline{\mathcal{A}}^{Y}))=
	\mathbb{L}_{n}\oplus_{\theta_{g}} \mathfrak{g},
\]
see Sections \ref{sec:CCX} and \ref{sec:Yclosure} for details.
Therefore, we can apply the middle convolution 
$\mathfrak{mc}_{\lambda}$ for $U(\mathfrak{h}(M(\overline{\mathcal{A}}^{Y})))$-modules,
and obtain the following  geometric realization of 
$\mathfrak{mc}_{\lambda}$ as the middle convolution $\mathrm{mc}_{\lambda}$ 
on $\mathrm{Pf}(\mathrm{log}(\mathcal{A}))$.
\begin{thm}[Theorem \ref{thm:commPfmod}]
We have the following commutative diagram
up to natural isomorphisms of functors:
\[\begin{tikzcd}[column sep=4cm]
\mathrm{Pf}(\mathrm{log}(\mathcal{A}))\arrow[r,"\mathrm{mc}_{\lambda}"]\arrow[d,"\simeq"',"\Xi_{\mathcal{A}}"]&
\mathrm{Pf}(\mathrm{log}(\overline{\mathcal{A}}^{Y}))\arrow[d,"\simeq"', "\Xi_{\overline{\mathcal{A}}^{Y}}"]\\
U(\mathfrak{h}(M(\mathcal{A})))\text{-}\mathrm{mod}\arrow[r,"\mathfrak{mc}_{\lambda}\circ \mathrm{Res}_{\mathfrak{h}(M(\overline{\mathcal{A}}^{Y}))}^{\mathfrak{h}(M(\mathcal{A}))}"]&
U(\mathfrak{h}(M(\overline{\mathcal{A}}^{Y})))\text{-}\mathrm{mod}
\end{tikzcd}.
\]
\end{thm}
\subsection{Riemann-Hilbert correspondence}
Finally, we will compare the middle convolution $\mathfrak{mc}_{\lambda}$ for
Lie algebra representations with the middle convolution $\mathrm{MC}_{\chi}$ for local systems.
The middle convolution 
was originally introduced by Katz \cite{Katz} for 
local systems over the punctured Riemann sphere.
In \cite{AdaHir}, a natural generalization of the middle convolution 
functor for local systems over the complement $M(\mathcal{A})$ 
of a hyperplane arrangement $\mathcal{A}$ is given:
\[
	\mathrm{MC}_{\chi}\colon \mathrm{Loc}(M(\mathcal{A}))\to \mathrm{Loc}(M(\overline{\mathcal{A}}^{Y})),
\]
where $\chi\colon \mathbb{Z}\to \mathbb{C}^{\times}$ is a multiplicative 
character, see Section \ref{sec:MCforLoc} for details.

Then in Section \ref{sec:RH},
we will provide a correspondence between the middle convolution 
$\mathfrak{mc}_{\lambda}$
for the holonomy Lie algebra $\mathfrak{h}(M(\mathcal{A}))$
and the middle convolution $\mathrm{MC}_{\chi}$ for local systems 
on $M(\mathcal{A})$.
For this purpose, 
we define a functor from 
the category of modules over the holonomy Lie algebra 
$\mathfrak{h}(M(\mathcal{A}))$ to the category of local systems 
on $M(\mathcal{A})$
\[
	\mathfrak{DR}\colon U(\mathfrak{h}(M(\mathcal{A})))\text{-}\mathbf{mod}\to \mathrm{Loc}(M(\mathcal{A})).
\]
by combining the de Rham functor 
\[
\mathrm{DR}\colon \mathrm{Pf}(\mathrm{log}(\mathcal{A}))\to \mathrm{Loc}(M(\mathcal{A})); \quad \nabla \mapsto \mathrm{Ker}(\nabla)
\]
with the equivalence $\Xi_{\mathcal{A}}^{-1}\colon U(\mathfrak{h}(M(\mathcal{A})))\text{-}\mathbf{mod}
\xrightarrow{\simeq} \mathrm{Pf}(\mathrm{log}(\mathcal{A}))$.
Then we will show the following compatibility of the middle convolution
$\mathfrak{mc}_{\lambda}$ for Lie algebra representations 
with the middle convolution $\mathrm{MC}_{\chi}$ for local systems 
on $M(\mathcal{A})$.
\begin{thm}[Theorem \ref{thm:RHmc}]
	Let $\mathcal{A}$ be an affine hyperplane arrangement in $\C^{l}$ and $Y$ a line in $\C^{l}$ passing through the origin.
	Let us take $\lambda\in K\setminus\{0\}$ and let $\chi\colon \mathbb{Z}\to \mathbb{C}^{\times}$ be the character defined by $\chi(1)=\exp(2\pi i\lambda)$. 
	Let $V$ be a $U(\mathfrak{h}(M(\mathcal{A})))$-module satisfying the following conditions:
	\begin{itemize}
		\item For any $H\in \mathcal{A}_{\parallel Y}^{c}=\{H\in \mathcal{A}\mid H\not\parallel Y\}$, the eigenvalues of $H$ on $V$ are not in $\mathbb{Z}\setminus\{0\}$.
		\item The eigenvalues of $\left(\sum_{H\in \mathcal{A}_{\parallel Y}^{c}}H\right)+\lambda$ on $V$ are not in $\mathbb{Z}\setminus\{0\}$. 
	\end{itemize}
Then there exists an isomorphism of local systems on $M(\mathcal{A})$,
\[
	\mathfrak{DR}_{\overline{\mathcal{A}}^{Y}}\circ \mathfrak{mc}_{\lambda}\circ \mathrm{Res}_{\mathfrak{h}(M(\overline{\mathcal{A}}^{Y}))}^{\mathfrak{h}(M(\mathcal{A}))}(V)\simeq \mathrm{MC}_{\chi}\circ\mathfrak{DR}_{\mathcal{A}}(V).
\]
\end{thm}
\subsection{Acknowledgements}
The author is supported by JSPS KAKENHI Grant Number 25K07043.
\section{Long-Moody functor for Lie algebra representations}\label{sec:longmoody}
The Long-Moody functor is a procedure to construct representations of braid groups from representations of larger braid groups,
which was introduced by Long and Moody in \cite{Lo94} as a higher dimensional analogue of the classical Burau representation of braid groups.
In this section, we introduce a Lie algebra analogue of the Long-Moody functor, as a functor between modules over Drinfeld-Kohno Lie algebras, which is an infinitesimal analogue of the pure braid groups.
We moreover consider a one parameter deformation of the Lie algebra analogue of the Long-Moody functor.

Let us fix $K$ as a field. Let us denote by $\mathbb{L}(x_{1},\ldots,x_{n})$
the free Lie algebra generated by $x_{1},\ldots,x_{n}$ over $K$,
which is frequently written as $\mathbb{L}_{n}$ for simplicity.
Also for a vector space $V$ over $K$, we denote by $\mathbb{L}(V)$ 
the free Lie algebra generated by $V$ over $K$.
\subsection{Long-Moody functor}
Let us  recall the definition of the Long-Moody functor.
Let $B_{n}$ denote the \emph{braid group of $n$-strands} 
with $\sigma_{i}$, $i=1,\ldots,n-1$ as the standard generators
satisfying the braid relation:
\begin{align*}
	\sigma_{i}\sigma_{j}&=\sigma_{j}\sigma_{i} \quad \text{if $|i-j|\ge 2$},\\
	\sigma_{i}\sigma_{i+1}\sigma_{i}&=\sigma_{i+1}\sigma_{i}\sigma_{i+1} \quad \text{for $i=1,\ldots,n-2$}.
\end{align*}
The \emph{pure braid group of $n$-strands} $P_{n}$ is defined as the kernel of the natural homomorphism $B_{n}\to \mathfrak{S}_{n}$ to
the symmetric group $\mathfrak{S}_{n}$ on $n$-letters.

As is well-known, $B_{n}$
is isomorphic to a subgroup of the automorphism group $\mathrm{Aut}(F_{n})$ of the free group $F_{n}$ with the generators $x_{1},\ldots,x_{n}$, where the isomorphism is given by
\[
	\theta_{\mathrm{Art}}\colon B_{n}\hookrightarrow \mathrm{Aut}(F_{n}), \quad \sigma_{i}\mapsto \begin{cases}
		x_{i}\mapsto x_{i+1},\\
	x_{i+1}\mapsto x_{i+1}^{-1}x_{i}x_{i+1},\\
	x_{j}\mapsto x_{j} \quad \text{for $j\neq i,i+1$}.
	\end{cases}
\]
This inclusion $\theta_{\mathrm{Art}}$ is called the \emph{Artin representation} of $B_{n}$
on the free group $F_{n}$.

For the braid group $B_{n+1}$ of $n+1$ strands
with the standard generators $\sigma_{0},\sigma_{1},\ldots,\sigma_{n-1}$,
we consider the inclusion $\iota\colon B_{n}\hookrightarrow B_{n+1}$ defined by 
$\iota(\sigma_{i})=\sigma_{i}$ for $i=1,\ldots,n-1$, and 
then the subgroup $B_{1,n}$ of $B_{n+1}$
generated by $\sigma_{0}^{2},\sigma_{1},\ldots,\sigma_{n-1}$ has 
the 
semidirect product decomposition 
\[
	B_{1,n} = F_{n}\rtimes_{\theta_{\mathrm{Art}}}B_{n}
\]
associated with the Artin representation $\theta_{\mathrm{Art}}$.
Namely, the elements
\[
x_{1}:=\sigma_{0}^{2}, x_{2}:=\sigma_{1}x_{1}\sigma_{1}^{-1},
\ldots, x_{n}:=\sigma_{n-1}x_{n-1}\sigma_{n-1}^{-1}
\]
generate the free group $F_{n}$ and 
we have the semidirect product decomposition.

Let $K[B_{n}]$ denote the group ring generated by $B_{n}$ over $K$
and take a finite dimensional module $V$ over $K[B_{n+1}]$. 
Then we construct a $K[B_{n}]$-module $\mathcal{LM}(V)$ by the following procedure.
Let $\mathcal{I}_{F_{n}}$ be the augmentation ideal of the group ring $K[F_{n}]$ generated by $F_{n}$ over $K$. Namely,
\[
	\mathcal{I}_{F_{n}} := \ker\left(\epsilon\colon K[F_{n}]\to K; \sum_{g\in F_{n}}a_{g}g\mapsto \sum_{g\in F_{n}}a_{g}\right).
\]
Then $\mathcal{I}_{F_{n}}$ is a free $K[F_{n}]$-module 
generated by $x_{1}-1,\ldots,x_{n}-1$, and 
$\mathcal{I}_{F_{n}}$ is closed under 
the action of $B_{n}$ on $K[F_{n}]$ induced	by the Artin representation $\theta_{\mathrm{Art}}$.

Let us regard $V$ as a $K[F_{n}]$-module
through the inclusion 
\[F_{n}\hookrightarrow 
F_{n}\rtimes_{\theta_{\mathrm{Art}}}B_{n}=B_{1,n}\subset B_{n+1}\]
and
then we consider the tensor product 
\[
\mathcal{LM}(V):=\mathcal{I}_{F_{n}}\otimes_{K[F_{n}]}V
\]
which has a $K[B_{n}]$-module structure defined  by
\[
	b\cdot (i\otimes v):=\theta_{\mathrm{Art}}(b)(i)\otimes (b\cdot v)
	\quad \text{for $b\in B_{n}, i\in \mathcal{I}_{F_{n}}, v\in V$}.
\]
\begin{df}[Long-Moody functor]\label{df:longmoody}
	The \emph{Long-Moody functor} is
	the functor defined by 
	\[
		\mathcal{LM}\colon K[B_{n+1}]\text{-}\mathrm{\mathbf{mod}}\to K[B_{n}]\text{-}\mathrm{\mathbf{mod}}; \quad V\mapsto \mathcal{I}_{F_{n}}\otimes_{K[F_{n}]}V.
	\] 
	Here $R\text{-}\mathrm{\mathbf{mod}}$ is the category of finitely generated left modules over a ring $R$.
\end{df}	
Since the augmentation ideal $\mathcal{I}_{F_{n}}$ is a free $K[F_{n}]$-module,
the Long-Moody functor $\mathcal{LM}$ is an exact functor.

\subsection{Long-Moody functor for Drinfeld-Kohno Lie algebras}
We now introduce 
a Lie algebra analogue of the Long-Moody functor,
as a functor between modules over Drinfeld-Kohno Lie algebras,
which are infinitesimal analogues of the pure braid groups. 

We first recall the definition of Drinfeld-Kohno Lie algebras.
\begin{df}[Drinfeld-Kohno Lie algebra]
	Let $\mathbb{L}(A_{i,j})$ be the free Lie algebra over $K$ 
	generated by $A_{i,j}$ for $1\le i\neq j\le n$.
	Then the \emph{Drinfeld-Kohno Lie algebra $\mathfrak{P}_{n}$}  
	is defined as the quotient 
	of $\mathbb{L}(A_{i,j})$ by the ideal generated by the \emph{infinitesimal braid relation}:
	\begin{align*}
	A_{i,j}-A_{j,i}&=0, \\
	[A_{i,k},A_{i,j}+A_{j,k}]&=0,\\
	[A_{i,j},A_{k,l}]&=0, 
\end{align*}
for mutually distinct $i,j,k,l$.
\end{df}
As an analogue of the decomposition of the pure braid group $P_{n+1}$,
\[
	P_{n+1}=F_{n}\rtimes P_{n},
\]
associated with the Artin representation $P_{n}\hookrightarrow B_{n}\xrightarrow{\theta_{\mathrm{Art}}}
\mathrm{Aut}(F_{n})$,
the infinitesimal braid relation tells us that 
$\mathfrak{P}_{n+1}$ also has the semidirect sum decomposition
\[
	\mathfrak{P}_{n+1}=\mathbb{L}(A_{1,n+1},\ldots,A_{n,n+1})\oplus \mathfrak{P}_{n}.
\]

For later reference, let us give a name to
the adjoint action of $\mathfrak{P}_{n}$ on the free Lie algebra $\mathbb{L}(A_{1,n+1},\ldots,A_{n,n+1})$,
which is an infinitesimal analogue of the Artin representation restricted to $P_{n}$.
\begin{df}[Infinitesimal braid action]\label{Infbraidaction}
	Let $\mathbb{L}_{n}=\mathbb{L}(x_{1},\ldots,x_{n})$ be the free Lie algebra generated by $x_{1},\ldots,x_{n}$
	over $K$. Let $\mathrm{Der}(\mathbb{L}_{n})$ be the Lie algebra of derivations of $\mathbb{L}_{n}$.
	
	We call the Lie algebra homomorphism $\theta\colon \mathfrak{P}_{n}\to \mathrm{Der}(\mathbb{L}_{n})$
	defined by 
	\[
	\theta(A_{i,j})(x_{k}):=
	\begin{cases}
		[x_{i},x_{j}] & \text{if $i=k$},\\
		[x_{j},x_{i}] & \text{if $j=k$},\\
		0 & \text{otherwise},
	\end{cases}
\]
	the \emph{infinitesimal braid action} of $\mathfrak{P}_{n}$ on the free Lie algebra $\mathbb{L}_{n}$.
\end{df}
This immediately tells us that 
\[
	\theta(\sigma)(x_{1}+\ldots+x_{n})=0\quad \text{for $\sigma\in \mathfrak{P}_{n}$}.
\] 

The free Lie algebra $\mathbb{L}_{n}$
is closed under the
infinitesimal braid action of $\mathfrak{P}_{n}$,
and moreover, free generators $x_{1},\ldots,x_{n}$
are translated to the adjoint of themselves by the infinitesimal braid action.
\begin{lem}\label{lem:purebraid}
For $x_{i}\in \{x_{1},\ldots,x_{n}\}$ and $\sigma\in \mathfrak{P}_{n}$,
there exists $v\in \mathbb{L}(x_{1},\ldots,x_{n})$ such that 
\[
	\theta(\sigma)(x_{i})=[x_{i},v].
\]
\end{lem}
\begin{proof}
Since elements of $\mathfrak{P}_{n}$
are linear combinations of such elements,
we may assume 
$\sigma=\sigma_{k}=[A_{i_{k},j_{k}},[A_{i_{k-1},j_{k-1}},[\cdots,[A_{i_{2},j_{2}},A_{i_{1},j_{1}}]\cdots]]]$
for $k\ge 1$ (where 
we set $\sigma=A_{i_{1},j_{1}}$ if $k=1$).
We proceed by induction on $k$.
If $k=1$, then the claim directly follows from the definition of $\theta$.
We assume that the claim holds for all $\sigma_{k-1}$ of the form $\sigma_{k-1}=[A_{i_{k-1},j_{k-1}},[\cdots,[A_{i_{2},j_{2}},A_{i_{1},j_{1}}]\cdots]]$.
Then there exists $v\in \mathbb{L}(x_{1},\ldots,x_{n})$
such that $\theta(\sigma_{k-1})(x_{i})=[x_{i},v]$
and we have
\begin{align*}
\theta(\sigma)(x_{i})&=\theta([A_{i_{k},j_{k}},\sigma_{k-1}])(x_{i})\\
&=\theta(A_{i_{k},j_{k}})(\theta(\sigma_{k-1})(x_{i}))-\theta(\sigma_{k-1})(\theta(A_{i_{k},j_{k}})(x_{i}))\\
&=\theta(A_{i_{k},j_{k}})([x_{i},v])-\theta(\sigma_{k-1})(\theta(A_{i_{k},j_{k}})(x_{i}))\\
&=[\theta(A_{i_{k},j_{k}})(x_{i}),v]+[x_{i},\theta(A_{i_{k},j_{k}})(v)]-\theta(\sigma_{k-1})(\theta(A_{i_{k},j_{k}})(x_{i})).
\end{align*}
The infinitesimal braid relation tells us that 
\[
	\theta(A_{i_{k},j_{k}})(x_{i})=\pm[x_{i_{k}},x_{j_{k}}]\text{ or }0.
\]
If $\theta(A_{i_{k},j_{k}})(x_{i})=0$, the claim follows from the above equation,
therefore, we may assume $\theta(A_{i_{k},j_{k}})(x_{i})=[x_{i_{k}},x_{j_{k}}]$
and $i=i_{k}$. 
The case $-[x_{i_{k}},x_{j_{k}}]$ is similar.
Then the above computation proceeds as follows, 
\begin{align*}
\theta(\sigma)(x_{i})&=[\theta(A_{i_{k},j_{k}})(x_{i}),v]+[x_{i},\theta(A_{i_{k},j_{k}})(v)]-\theta(\sigma_{k-1})(\theta(A_{i_{k},j_{k}})(x_{i}))\\
&=[[x_{i},x_{j_{k}}],v]	+[x_{i},\theta(A_{i_{k},j_{k}})(v)]-\theta(\sigma_{k-1})([x_{i},x_{j_{k}}])\\
&=[[x_{i},x_{j_{k}}],v]	+[x_{i},\theta(A_{i_{k},j_{k}})(v)]-[\theta(\sigma_{k-1})(x_{i}),x_{j_{k}}]
-[x_{i},\theta(\sigma_{k-1})(x_{j_{k}})]\\
&=[[x_{i},x_{j_{k}}],v]	+[x_{i},\theta(A_{i_{k},j_{k}})(v)]-[[x_{i},v],x_{j_{k}}]
-[x_{i},\theta(\sigma_{k-1})(x_{j_{k}})]\\
&=[x_{i},[x_{j_{k}},v]]
+[x_{i},\theta(A_{i_{k},j_{k}})(v)]
-[x_{i},\theta(\sigma_{k-1})(x_{j_{k}})].
\end{align*}
Here for the last equality, we used the Jacobi identity
\[
[[x_{i},x_{j_{k}}],v]+[[x_{j_{k}},v],x_{i}]+[[v,x_{i}],x_{j_{k}}]=0.
\]
This shows the claim for $\sigma_{k}$.
\end{proof}
The following basic property of the infinitesimal braid action will be referred to in later sections.
\begin{lem}\label{lem:minpurebraid}
	The Lie subalgebra $\mathbb{L}(x_{1},\ldots,x_{n-1})
	\subset \mathbb{L}(x_{1},\ldots,x_{n})$ 
	is closed under the infinitesimal braid action of 
	the Lie subalgebra $\mathfrak{P}_{n-1}$ of 
	$\mathfrak{P}_{n}$.
\end{lem} 
\begin{proof}
Let us consider $\mathfrak{P}_{n+1}$
and set $A_{i,n+1}=x_{i}$ and $A_{i,n}=y_{i}$ for $i=1,\ldots,n-1$
and recall that $\mathfrak{P}_{n-1}$ is generated by $A_{i,j}$ for $1\le i\neq j\le n-1$.
Then we can see that the infinitesimal braid relation of 
$A_{i,j}$ and $x_{l}$
for $1\le i\neq j\le n-1$ and $l=1,\ldots,n-1$
are the same as those of 
$A_{i,j}$ and $y_{l}$ for $1\le i\neq j\le n-1$ and $l=1,\ldots,n-1$.
Therefore, since $\mathbb{L}(A_{1,n},\ldots,A_{n-1,n})=\mathbb{L}(y_{1},\ldots,y_{n-1})$
is an ideal of $\mathfrak{P}_{n}$ and thus
is closed under the adjoint action of $\mathfrak{P}_{n-1}$,
$\mathbb{L}(x_{1},\ldots,x_{n-1})$
is also closed under the adjoint action of $\mathfrak{P}_{n-1}$, and thus the claim follows.
\end{proof}
Let us now construct a Lie algebra analogue of the Long-Moody functor for modules over the Drinfeld-Kohno Lie algebra.
Let $V$ be a finite dimensional module over the universal enveloping algebra $U(\mathfrak{P}_{n+1})$ of $\mathfrak{P}_{n+1}$.
We denote the augmentation ideal of 
the universal enveloping algebra $U(\mathbb{L}_{n})$ of the free Lie algebra $\mathbb{L}_{n}$
by $\mathcal{I}_{\mathbb{L}_{n}}$, namely,
since $U(\mathbb{L}_{n})$ can be indentified with the free associative algebra $K\{x_{1},\ldots,x_{n}\}$ generated by $x_{1},\ldots,x_{n}$,
the augmentation ideal is given by
\[
	\mathcal{I}_{\mathbb{L}_{n}}:= \ker\left(\epsilon\colon K\{x_{1},\ldots,x_{n}\}\to K; f(x_{1},\ldots,x_{n})\mapsto f(0,\ldots,0)\right).	
\]
Therefore $\mathcal{I}_{\mathbb{L}_{n}}$ is a left or right free $U(\mathbb{L}_{n})$-module
generated by $x_{1},\ldots,x_{n}$, and
closed under
the infinitesimal braid action of $\mathfrak{P}_{n}$.

Regarding $V$ as a $U(\mathbb{L}_{n})$-module
through the inclusion 
\[
	\mathbb{L}_{n}\hookrightarrow \mathbb{L}_{n}\oplus \mathfrak{P}_{n}=\mathfrak{P}_{n+1},
\]
we can consider the tensor product
\[
	\mathfrak{LM}(V):=\mathcal{I}_{\mathbb{L}_{n}}\otimes_{U(\mathbb{L}_{n})}V
\]
with a $U(\mathfrak{P}_{n})$-module structure defined by
\[
	X\cdot (i\otimes v):=\theta(X)(i)\otimes v
	+i\otimes (X\cdot v)
	\quad \text{for $X\in \mathfrak{P}_{n}, i\in \mathcal{I}_{\mathbb{L}_{n}}, v\in V$}.	
\]
This action is actually well-defined since we have
\begin{align*}
X\cdot ((i\cdot l)\otimes v)&=\theta(X)(i\cdot l)\otimes v+i\cdot l\otimes (X\cdot v)\\
&=(\theta(X)(i)\cdot l+i\cdot \theta(X)(l))\otimes v+i\cdot l\otimes (X\cdot v)\\
&=\theta(X)(i)\otimes (l\cdot v)+i\otimes (\theta(X)(l))\cdot v+i\otimes ((l\cdot X)\cdot v)\\
&=\theta(X)(i)\otimes (l\cdot v)+i\otimes ((\theta(X)(l)+(l\cdot X))\cdot v)\\
&=\theta(X)(i)\otimes (l\cdot v)+i\otimes ((X\cdot l)\cdot v)\\
&=X\cdot (i\otimes (l\cdot v)),
\end{align*}
for $X\in \mathfrak{P}_{n}, i\in \mathcal{I}_{\mathbb{L}_{n}}, l\in U(\mathbb{L}_{n}), v\in V$.

\begin{df}[Long-Moody functor for modules over Drinfeld-Kohno Lie algebras]
	Let us define a 
	functor by 
	\[
		\mathfrak{LM}\colon U(\mathfrak{P}_{n+1})\text{-}\mathbf{mod}\to U(\mathfrak{P}_{n})\text{-}\mathbf{mod}; \quad V\mapsto \mathcal{I}_{\mathbb{L}_{n}}\otimes_{U(\mathbb{L}_{n})}V,
	\]
	which is called the \emph{Long-Moody functor} as well.
\end{df}
The Long-Moody functor $\mathfrak{LM}$ is an exact functor since $\mathcal{I}_{\mathbb{L}_{n}}$ is a free $U(\mathbb{L}_{n})$-module.
\subsection{Deformed Long-Moody functor for Lie algebra representations}
We moreover consider a one parameter deformation of the above Lie algebra analogue of the Long-Moody functor.
Let us take a Lie algebra $\mathfrak{g}$ over $K$ with 
a Lie algebra homomorphism $g\colon \mathfrak{g}\to \mathfrak{P}_{n}$.
\begin{df}[Infinitesimal braid action of $\mathfrak{g}$]
	We call the composition map $\theta\circ g\colon \mathfrak{g}\to \mathrm{Der}(\mathbb{L}_{n})$ 
	the \emph{infinitesimal braid action}
	associated with $g$. The pair $(\mathfrak{g}, g)$
	is referred to as the infinitesimal braid action.
\end{df}	
The composition map 
\[
	\theta_{g}:=\theta\circ g\colon \mathfrak{g}\to \mathrm{Der}(\mathbb{L}_{n})
\]
allows us to define the semidirect sum Lie algebra
\[
	\mathbb{L}_{n}\oplus_{\theta_{g}} \mathfrak{g}.
\]
Then the Lie algebra homomorphism
\[
	\mathbb{L}_{n}\oplus_{\theta_{g}} \mathfrak{g}
	\xrightarrow[]{\mathrm{id}_{\mathbb{L}_{n}}\oplus g}
	\mathbb{L}_{n}\oplus_{\theta} \mathfrak{P}_{n}=\mathfrak{P}_{n+1},
\]
gives an infinitesimal braid action
on the free Lie algebra $\mathbb{L}_{n+1}$ as well, 
which is denoted by 
\[
	\theta_{\mathrm{id}\oplus g}\colon \mathbb{L}_{n}\oplus_{\theta_{g}} \mathfrak{g}
	\longrightarrow \mathrm{Der}(\mathbb{L}_{n+1}).
\]

Let us take a finite dimensional module $V$ over  
$\mathbb{L}_{n}\oplus_{\theta_{g}} \mathfrak{g}$.
Also for $\lambda\in K\backslash\{0\}$, 
consider the one-dimensional module $K_{\lambda}$ over $\mathbb{L}(x_{n+1})$ defined by
\[
	x_{n+1}\cdot 1 = \lambda\cdot 1, \quad (1\in K=K_{\lambda}).
\]
Then
we construct an
$\mathbb{L}_{n}\oplus_{\theta_{g}} \mathfrak{g}$-module 
$\mathfrak{LM}^{\mathrm{df}}_{\lambda}(V)$ by the following procedure.
The projection maps
\begin{align*}
	\mathrm{pr}_{1,\ldots,n}&\colon \bigoplus_{i=1}^{n+1}K\cdot x_{i}
	\rightarrow \bigoplus_{i=1}^{n}K\cdot x_{i},&
	\mathrm{pr}_{n+1}&\colon \bigoplus_{i=1}^{n+1}K\cdot x_{i}
	\rightarrow K\cdot x_{n+1}
\end{align*}
as $K$-vector spaces
induce the Lie algebra homomorphisms
\[
	\mathrm{pr}_{1,\ldots,n}\colon \mathbb{L}_{n+1}\to \mathbb{L}_{n}, \quad
	\mathrm{pr}_{n+1}\colon \mathbb{L}_{n+1}\to \mathbb{L}(x_{n+1}).
\]
By regarding $V$ and $K_{\lambda}$
as modules over $\mathbb{L}_{n+1}$ through 
these homomorphisms $\mathrm{pr}_{1,\ldots,n}$ and $\mathrm{pr}_{n+1}$, respectively,
we consider the tensor product
\[
	V\otimes_{K} K_{\lambda}
\]
as the module over $\mathbb{L}_{n+1}$ and its universal enveloping algebra $U(\mathbb{L}_{n+1})$.
Then we will define $\mathfrak{LM}^{\mathrm{df}}_{\lambda}(V)$ as the 1st homology group of 
the free Lie algebra $\mathbb{L}_{n+1}$, 
\[
	H_1(\mathbb{L}_{n+1}, V\otimes_{K} K_{\lambda}).
\]
For this purpose, let us
equip $H_1(\mathbb{L}_{n+1}, V\otimes_{K} K_{\lambda})$
with an $\mathbb{L}_{n}\oplus_{\theta_{g}} \mathfrak{g}$-module structure.
Let us recall that the short exact sequence 
\[
	0\to \mathcal{I}_{\mathbb{L}_{n+1}}\to U(\mathbb{L}_{n+1})\to K\to 0
\]
induces the following long exact sequence:
\begin{multline}\label{eq:longexact}
	0\to H_1(\mathbb{L}_{n+1}, V\otimes_{K} K_{\lambda})
	\to 
	\mathcal{I}_{\mathbb{L}_{n+1}}\otimes_{U(\mathbb{L}_{n+1})} (V\otimes_{K} K_{\lambda})\\
	\to V\otimes_{K} K_{\lambda}\to H_0(\mathbb{L}_{n+1}, V\otimes_{K} K_{\lambda})\to 0,
\end{multline}
where the third arrow is given by $i\otimes v\mapsto i\cdot v$ for $i\in \mathcal{I}_{\mathbb{L}_{n+1}}$ and $v\in V\otimes_{K} K_{\lambda}$,
see \cite{Wei94} for example.
Let us notice that 
the assumption $\lambda\neq 0$ implies that
$\mathbb{L}_{n+1}\cdot (V\otimes_{K} K_{\lambda})
=V\otimes_{K} K_{\lambda}$, and therefore 
the vanishing of the $0$-th homology group follows:
\[
H_{0}(\mathbb{L}_{n+1}, V\otimes_{K} K_{\lambda})
\cong V\otimes_{K} K_{\lambda}/\left(\mathbb{L}_{n+1}\cdot (V\otimes_{K} K_{\lambda})\right)
=0.
\]

Now let us 
equip $\mathcal{I}_{\mathbb{L}_{n+1}}\otimes_{U(\mathbb{L}_{n+1})} (V\otimes_{K} K_{\lambda})$ with an $\mathbb{L}_{n}\oplus_{\theta_{g}} \mathfrak{g}$-module structure.
First we see the $\mathfrak{g}$-action.
As well as for the Long-Moody functor, 
\[
	\mathcal{I}_{\mathbb{L}_{n+1}}\otimes_{U(\mathbb{L}_{n+1})} (V\otimes_{K} K_{\lambda})
\]
is equipped with the well-defined $\mathfrak{g}$-action given
by
\[
	X\cdot (i\otimes v):=
	\theta_{\mathrm{id}\oplus g}(X)(i)\otimes v+i\otimes (X\cdot v),
	\quad\text{for }
	X\in \mathfrak{g},
i\in \mathcal{I}_{\mathbb{L}_{n+1}}, v\in V\otimes_{K}K_{\lambda}.
\]
Here  
$K_{\lambda}$ is regarded as the trivial 
$\mathfrak{g}$-module.
Furthermore, the $K$-linear map
\[
\mathcal{I}_{\mathbb{L}_{n+1}}\otimes_{U(\mathbb{L}_{n+1})} (V\otimes_{K} K_{\lambda})
\rightarrow 
V\otimes_{K} K_{\lambda};\quad
i\otimes v\mapsto i\cdot v 
\]
intertwines the $\mathfrak{g}$-actions, since we have
\begin{align*}
(\theta(X)_{\mathrm{id}\oplus g}(i))\cdot v+i\cdot (X\cdot v)=X\cdot (i\cdot v).
\end{align*}
This shows that $H_{1}(\mathbb{L}_{n+1}, V\otimes_{K} K_{\lambda})$ 
which is the kernel of the above map,
is closed under the  $\mathfrak{g}$-action.

Next, the $\mathbb{L}_{n}$-action on $\mathcal{I}_{\mathbb{L}_{n+1}}\otimes_{U(\mathbb{L}_{n+1})} (V\otimes_{K} K_{\lambda})$
is introduced by the following way:
\[
x\cdot (i\otimes v):=\theta_{\mathrm{id}\oplus g}(x)(i)	\otimes v \quad \text{for $x\in \mathbb{L}_{n}, i\in \mathcal{I}_{\mathbb{L}_{n+1}}, v\in V\otimes_{K}K_{\lambda}$}.
\]
The well-definedness of this action is not trivial, and we need the following lemma to show it.
\begin{lem}
We have 
\[
	\theta_{\mathrm{id}\oplus g}(x)(i\cdot l)\otimes v
	=\theta_{\mathrm{id}\oplus g}(x)(i)\cdot l\otimes v 
\]
for $x\in \mathbb{L}_{n}$, $i\in \mathcal{I}_{\mathbb{L}_{n+1}}$, $l\in U(\mathbb{L}_{n+1})$, and $v\in V\otimes_{K}K_{\lambda}$.
\end{lem} 
\begin{proof}
	It is enough to show that 
	\[
		\theta_{\mathrm{id}\oplus g}(x)(l)\cdot v=0.
	\]
Since elements of $U(\mathbb{L}_{n+1})$ are linear combinations of 
products 
of elements of $\mathbb{L}_{n+1}$,
it suffices to show the above equality for 
$l\in \mathbb{L}_{n+1}$.
Moreover since $\theta_{\mathrm{id}\oplus g}|_{\mathbb{L}_{n}}$ coincides with the natural inclusion 
$
\mathbb{L}_{n}\hookrightarrow\mathfrak{P}_{n+1},
$
the equation 
\begin{equation}\label{braidvanish}
	\theta(A_{i,n+1})(l)\cdot v=0
\end{equation}
deduces the above required equation. 

We may assume that 
$l$ is of the form 
$l=[x_{i_{k}},[x_{i_{k-1}},[\ldots,[x_{i_2},x_{i_1}]\ldots]]]$
for $k\ge 1$ where we set $l=x_{i_1}$ if $k=1$,
and show the equation \eqref{braidvanish}
by induction on $k$. 

If $k=1$, we have
\[
\theta(A_{i,n+1})(l)=\theta(A_{i,n+1})(x_{i_1})=
\begin{cases}
\pm [x_{i},x_{n+1}] & \text{if $i_{1}=i$ or $i_{1}=n+1$},\\
0 & \text{otherwise}.
\end{cases}
\]
Since $x_{n+1}\cdot v=\lambda v$, 
we have $[x_{i},x_{n+1}]\cdot v=x_{i}\cdot (\lambda v)-\lambda (x_{i}\cdot v)=0$,
which shows the equation \eqref{braidvanish} in the case $k=1$.

Suppose that the equation \eqref{braidvanish}
holds for $l_{k-1}=[x_{i_{k-1}},[\ldots,[x_{i_2},x_{i_1}]\ldots]]$,
namely, $\theta(A_{i,n+1})(l_{k-1})\cdot v=0$
for any $v\in V\otimes_{K}K_{\lambda}$.
Then we have
\begin{align*}
\theta(A_{i,n+1})(l)\cdot v&=\theta(A_{i,n+1})([x_{i_{k}},l_{k-1}])\cdot v\\
&=
[\theta(A_{i,n+1})(x_{i_{k}}),l_{k-1}]\cdot v
+ [x_{i_{k}},\theta(A_{i,n+1})(l_{k-1})]\cdot v.
\end{align*}
By the induction hypothesis,
$[x_{i_{k}},\theta(A_{i,n+1})(l_{k-1})]\cdot v=0$.
Also the equation 
\[
\theta(A_{i,n+1})(x_{i_k})=
\begin{cases}
\pm [x_{i},x_{n+1}] & \text{if $i_{k}=i$ or $i_{k}=n+1$},\\
0 & \text{otherwise}
\end{cases}
\]
implies 
the first term $[\theta(A_{i,n+1})(x_{i_{k}}),l_{k-1}]\cdot v$ is also zero, 
since $[x_{i},x_{n+1}]\cdot v=0$ for all $v\in V\otimes_{K}K_{\lambda}$.
Thus we are done.
\end{proof}
This lemma ensures the well-definedness of the $\mathbb{L}_{n}$-action
since we have 
\[
	x\cdot ((i\cdot l)\otimes v)=\theta_{\mathrm{id}\oplus g}(x)(i\cdot l)\otimes v
=\theta_{\mathrm{id}\oplus g}(x)(i)\cdot l\otimes v
=\theta_{\mathrm{id}\oplus g}(x)(i)\otimes (l\cdot v)
=x\cdot (i\otimes (l\cdot v)).
\]
Let us also check that the $\mathbb{L}_{n}$-action and the $\mathfrak{g}$-action are compatible, namely, the equation
\[
	[x,X]\cdot (i\otimes v)=\theta_{g}(X)(x)\cdot (i\otimes v)
\]
holds for $x\in \mathbb{L}_{n}, X\in \mathfrak{g}, i\in \mathcal{I}_{\mathbb{L}_{n+1}}$, and $v\in V\otimes_{K}K_{\lambda}$.
Indeed, we have 
\begin{align*}
[x,X]\cdot (i\otimes v)&=x\cdot (X\cdot (i\otimes v))-X\cdot (x\cdot (i\otimes v))\\
&=x\cdot (\theta_{\mathrm{id}\oplus g}(X)(i)\otimes v+i\otimes (X\cdot v))-X\cdot (\theta_{\mathrm{id}\oplus g}(x)(i)\otimes v)\\
&=\theta_{\mathrm{id}\oplus g}(x)(\theta_{\mathrm{id}\oplus g}(X)(i))\otimes v+\theta_{\mathrm{id}\oplus g}(x)(i)\otimes (X\cdot v)\\
&\quad -\theta_{\mathrm{id}\oplus g}(X)(\theta_{\mathrm{id}\oplus g}(x)(i))\otimes v
-\theta_{\mathrm{id}\oplus g}(x)(i)\otimes (X\cdot v)
\\
&=\theta_{g}(X)(x)\cdot (i\otimes v).
\end{align*}

Finally, let us see that $H_1(\mathbb{L}_{n+1}, V\otimes_{K} K_{\lambda})$ is closed under the $\mathbb{L}_{n}$-action.
By recalling that 
\[
	\theta_{\mathrm{id}\oplus g}(x)(l)\cdot v=0
\]
for $x\in \mathbb{L}_{n}, l\in U(\mathbb{L}_{n+1})$, and $v\in V\otimes_{K}K_{\lambda}$, we can see that 
the kernel of the map
$\mathcal{I}_{\mathbb{L}_{n+1}}\otimes_{U(\mathbb{L}_{n+1})} (V\otimes_{K} K_{\lambda})\rightarrow V\otimes_{K} K_{\lambda}$
is closed under the $\mathbb{L}_{n}$-action, and it means that  $H_1(\mathbb{L}_{n+1}, V\otimes_{K} K_{\lambda})$ is closed under the $\mathbb{L}_{n}$-action.

As a consequence, we obtain a well-defined action of 
$U(\mathbb{L}_{n}\oplus_{\theta_{g}}\mathfrak{g})$ 
 on $H_1(\mathbb{L}_{n+1}, V\otimes_{K} K_{\lambda})$.
 Moreover we can check that the correspondence between $V$ and $H_1(\mathbb{L}_{n+1}, V\otimes_{K} K_{\lambda})$
 is functorial, namely, 
for a morphism $f\colon V\to W$ in $U(\mathbb{L}_{n}\oplus_{\theta_{g}} \mathfrak{g})\text{-}\mathbf{mod}$, 
the induced map for homology groups
\[
	H_{f}(f)\colon H_{1}(\mathbb{L}_{n+1}, V\otimes_{K} K_{\lambda})\to H_{1}(\mathbb{L}_{n+1}, W\otimes_{K} K_{\lambda})
\]	
is a morphism in $U(\mathbb{L}_{n}\oplus_{\theta_{g}} \mathfrak{g})\text{-}\mathbf{mod}$ as well. 
\begin{df}[Deformed Long-Moody functor]
	Let $(\mathfrak{g},g)$ be a infinitesimal braid action on $\mathbb{L}_{n}$,
	and $\lambda\in K\backslash\{0\}$.
	Let us define an endofunctor of $U(\mathbb{L}_{n}\oplus_{\theta_{g}} \mathfrak{g})\text{-\textbf{mod}}$ 
	by 
	\[
		\begin{array}{cccc}
			\mathfrak{LM}^{\mathrm{df}}_{\lambda}\colon & U(\mathbb{L}_{n}\oplus_{\theta_{g}} \mathfrak{g})\text{-}\mathbf{mod} & \to & U(\mathbb{L}_{n}\oplus_{\theta_{g}} \mathfrak{g})\text{-}\mathbf{mod}\\
			& V & \mapsto & H_1(\mathbb{L}_{n+1}, V\otimes_{K} K_{\lambda})
		\end{array},
	\]
	which is called the \emph{deformed Long-Moody functor} with respect to $\lambda$.
\end{df}
\begin{prop}\label{prop:exact}
The deformed Long-Moody functor $\mathfrak{LM}^{\mathrm{df}}_{\lambda}$ is an exact functor.
\end{prop}
\begin{proof}
Let $0\to U \to V \to W \to 0$ be an exact sequence in $U(\mathbb{L}_{n}\oplus_{\theta_{g}} \mathfrak{g})\text{-\textbf{mod}}$.
Then we obtain the long exact sequence 
\begin{multline*}
			H_{2}(\mathbb{L}_{n+1}, U\otimes_{K} K_{\lambda}) \to H_{1}(\mathbb{L}_{n+1}, U\otimes_{K} K_{\lambda}) \to H_{1}(\mathbb{L}_{n+1}, V\otimes_{K} K_{\lambda})\\
			\to H_{1}(\mathbb{L}_{n+1}, W\otimes_{K} K_{\lambda}) \to H_{0}(\mathbb{L}_{n+1}, W\otimes_{K} K_{\lambda}).
\end{multline*}
As we noted above, we have $H_{0}(\mathbb{L}_{n+1}, W\otimes_{K} K_{\lambda})=0$ and 
moreover $H_{2}(\mathbb{L}_{n+1}, U\otimes_{K} K_{\lambda})=0$ since $\mathbb{L}_{n+1}$ is a free Lie algebra, see Corollary 7.2.5 in \cite{Wei94} for example.
Thus the desired short exact sequence follows:
\[
0\to H_{1}(\mathbb{L}_{n+1}, U\otimes_{K} K_{\lambda}) \to H_{1}(\mathbb{L}_{n+1}, V\otimes_{K} K_{\lambda}) \to H_{1}(\mathbb{L}_{n+1}, W\otimes_{K} K_{\lambda}) \to 0.
\]	
\end{proof}

\subsection{The limit $\lambda\to 0$ of the deformed Long-Moody functor and the original Long-Moody functor}
We explain that the deformed Long-Moody functor $\mathfrak{LM}^{\mathrm{df}}_{\lambda}$ can be actually seen as 
a deformation of the original Long-Moody functor $\mathfrak{LM}$.

Let us see the homology group $H_1(\mathbb{L}_{n+1}, V\otimes_{K} K_{\lambda})$ has a 
standard $K$-basis, and by which 
we can relate the deformed Long-Moody functor $\mathfrak{LM}^{\mathrm{df}}_{\lambda}$ to the original Long-Moody functor $\mathfrak{LM}$.
\begin{prop}
The homology group
$H_1(\mathbb{L}_{n+1}, V\otimes_{K} K_{\lambda})$
is generated over $K$ by  
\[
	[x_{i},x_{n+1}]\otimes v\in \mathcal{I}_{\mathbb{L}_{n+1}}\otimes_{U(\mathbb{L}_{n+1})}
	(V\otimes_{K}K_{\lambda})
\]
for $i=1,\ldots,n$ and $v\in V\otimes_{K}K_{\lambda}$.
In particular, 
\[
	\mathrm{dim}_{K}H_1(\mathbb{L}_{n+1}, V\otimes_{K} K_{\lambda})
	=n\cdot \mathrm{dim}_{K}V.
\]
\end{prop}
\begin{proof}
Recalling that $x_{n+1}$ acts on $V\otimes_{K} K_{\lambda}$ as the scalar multiplication by $\lambda$,
we have
\[
	[x_{i},x_{n+1}]\cdot v
	=x_{i}\cdot (x_{n+1}\cdot v)-x_{n+1}\cdot (x_{i}\cdot v)
	=\lambda (x_{i}\cdot v)-\lambda (x_{i}\cdot v)=0,
\]
namely,
\[
	 [x_{i},x_{n+1}]\otimes v
	 \in  \mathrm{Ker}\left(\mathcal{I}_{\mathbb{L}_{n+1}}\otimes_{U(\mathbb{L}_{n+1})} (V\otimes_{K} K_{\lambda})\rightarrow V\otimes_{K} K_{\lambda}\right)
\]
for $i=1,\ldots,n$ and $v\in V\otimes_{K}K_{\lambda}$.
 
Since 
\[
	\mathrm{dim}_{K}K\text{-span}\langle [x_{i},x_{n+1}]\otimes v\mid  
	i=1,\ldots,n, v\in V\otimes_{K}K_{\lambda}\rangle
	=n\cdot \mathrm{dim}_{K}V,
\]
it suffices to show that the dimension of 
$H_1(\mathbb{L}_{n+1}, V\otimes_{K} K_{\lambda})$
is $n\cdot \mathrm{dim}_{K}V$ as well.

The exact sequence \eqref{eq:longexact} gives us the dimension formula
\begin{multline*}
	\mathrm{dim}_{K}H_1(\mathbb{L}_{n+1}, V\otimes_{K} K_{\lambda})
	\\=\mathrm{dim}_{K}(\mathcal{I}_{\mathbb{L}_{n+1}}\otimes_{U(\mathbb{L}_{n+1})} (V\otimes_{K} K_{\lambda}))
	-\mathrm{dim}_{K}(V\otimes_{K} K_{\lambda})
	+\mathrm{dim}_{K}H_{0}(\mathbb{L}_{n+1}, V\otimes_{K} K_{\lambda}).
\end{multline*}
Then since $\mathcal{I}_{\mathbb{L}_{n+1}}$ is a free $U(\mathbb{L}_{n+1})$-module generated by $x_{1},\ldots,x_{n+1}$, 
and also since $H_{0}(\mathbb{L}_{n+1}, V\otimes_{K} K_{\lambda})=0$, we have
\[
\mathrm{dim}_{K}H_1(\mathbb{L}_{n+1}, V\otimes_{K} K_{\lambda})
	=(n+1)\cdot \mathrm{dim}_{K}V-\mathrm{dim}_{K}V
	=n\cdot \mathrm{dim}_{K}V.
\]
\end{proof}

This proposition leads to the following identification,
\[
	\mathfrak{LM}^{\mathrm{df}}_{\lambda}(V)\cong \bigoplus_{i=1}^{n} K\cdot  [x_{i},x_{n+1}] \otimes_{K} (V\otimes_{K} K_{\lambda}),
\]
and the right hand side is well-defined even for $\lambda=0$.
Thus we can 
define 
\[
	\mathfrak{LM}^{\mathrm{df}}_{0}(V):=\bigoplus_{i=1}^{n} K\cdot  [x_{i},x_{n+1}] \otimes_{K} V,
\]
with the $\mathbb{L}_{n}\oplus_{\theta_{g}} \mathfrak{g}$-module structure 
defined by the same formula as the $\mathbb{L}_{n}\oplus_{\theta_{g}} \mathfrak{g}$-module structure on $\mathfrak{LM}^{\mathrm{df}}_{\lambda}(V)$
with putting $\lambda=0$.

Then we have the following relation between the functors $\mathfrak{LM}^{\mathrm{df}}_{0}$ and $\mathfrak{LM}$.
\begin{prop}\label{prop:compLM}
Suppose that $\lambda=0$ and moreover suppose $\mathfrak{g}=\mathfrak{P}_{n}$
and $g\colon \mathfrak{P}_{n}\rightarrow \mathfrak{P}_{n}$ is the 
identity map, which implies that 
$\mathbb{L}_{n}\oplus_{\theta_{g}} \mathfrak{g}$
is naturally isomorphic to $\mathfrak{P}_{n+1}$.
Let us take the inclusion 
$\mathfrak{P}_{n}\hookrightarrow \mathfrak{P}_{n+1}$
along the decomposition 
$\mathfrak{P}_{n+1}=\mathbb{L}_{n}\oplus_{\theta}\mathfrak{P}_{n}$,
and consider the restriction functor 
\[
	\mathrm{Res}_{\mathfrak{P}_{n}}^{\mathfrak{P}_{n+1}}
	\colon U(\mathfrak{P}_{n+1})\text{-}\mathbf{mod}
	\longrightarrow U(\mathfrak{P}_{n})\text{-}\mathbf{mod}
\]
with respect to this inclusion. 

Then there exists a natural isomorphism of functors
\[	
\mathrm{Res}^{\mathfrak{P}_{n}}_{\mathfrak{P}_{n+1}}\circ \mathfrak{LM}^{\mathrm{df}}_{0}\cong \mathfrak{LM}.
\]
\end{prop}
\begin{proof}
Since $\mathcal{I}_{\mathbb{L}(x_{1},\ldots,x_{n+1})}$
is a free $U(\mathbb{L}(x_{1},\ldots,x_{n+1}))$-module generated by $x_{1},\ldots,x_{n+1}$, 
we have the following isomorphism of $K$-vector spaces:
\[
	\mathfrak{LM}(V)=\mathcal{I}_{\mathbb{L}(x_{1},\ldots,x_{n})}\otimes_{U(\mathbb{L}(x_{1},\ldots,x_{n}))} V\cong \bigoplus_{i=1}^{n} K\cdot x_{i} \otimes_{K} V.
\] 
Thus it suffices to show that 
the map   
\[
	\Phi\colon \bigoplus_{i=1}^{n} K\cdot [x_{i},x_{n+1}] \otimes_{K} V\rightarrow 
	\bigoplus_{i=1}^{n} K\cdot x_{i} \otimes_{K} V;\quad
	[x_i,x_{n+1}]\otimes v\mapsto x_{i}\otimes v,
\]
is a $U(\mathfrak{P}_{n})$-module isomorphism, which is obviously a $K$-isomorphism.
For $X\in \mathfrak{P}_{n}$,
we have 
\begin{align*}
	X\cdot ([x_i,x_{n+1}]\otimes v)&=X\cdot (x_{i}\otimes x_{n+1}\cdot v)-X\cdot (x_{n+1}\otimes x_{i}\cdot v)\\
	&=-x_{n+1}\otimes X\cdot (x_{i}\cdot v)\\
	&=-x_{n+1}\otimes (\theta(X)(x_{i})\cdot v+x_{i}\cdot(X\cdot v))\\
	&=[\theta(X)(x_{i}),x_{n+1}]\otimes v+[x_{i},x_{n+1}]\otimes (X\cdot v).
\end{align*}
Here we used equations 
$\theta(X)(x_{n+1})=0$ and $x_{n+1}\cdot v=0$.
This shows that 
\begin{align*}
	\Phi(X\cdot ([x_i,x_{n+1}]\otimes v))&=
	\Phi([\theta(X)(x_{i}),x_{n+1}]\otimes v)+\Phi([x_{i},x_{n+1}]\otimes (X\cdot v))\\
	&=\theta(X)(x_{i})\otimes v+x_{i}\otimes (X\cdot v)\\
	&=X\cdot (x_{i}\otimes v)\\
	&=X\cdot \Phi([x_i,x_{n+1}]\otimes v),
\end{align*}
as desired.
\end{proof}
\section{Middle convolution for Lie algebra representations}\label{sec:midconv}
Throughout this section, we fix an infinitesimal braid action $(\mathfrak{g},g)$ on $\mathbb{L}_{n}$.
In this section, we introduce the middle convolution functor for $U(\mathbb{L}_{n}\oplus_{\theta_{g}} \mathfrak{g})\text{-}\mathbf{mod}$, 
as a modification of the deformed Long-Moody functor $\mathfrak{LM}^{\mathrm{df}}_{\lambda}$.
Then we show that the middle convolution functor is a generalization of the Dettweiler-Reiter additive 
middle convolution functor, which appears as the special case $\mathfrak{g}=0$.
Also we show that as well as the classical case of Dettweiler-Reiter, 
our middle convolution functor satisfies the composition law, which is a fundamental property of the middle convolution functor. 

\subsection{Middle convolution functor}
In the previous section, we defined the deformed Long-Moody functor $\mathfrak{LM}^{\mathrm{df}}_{\lambda}$ for $\lambda\in K\backslash\{0\}$
as the 1st homology group $H_1(\mathbb{L}_{n+1}, V\otimes_{K} K_{\lambda})$.
The  inclusion map $\mathbb{L}(x_{i})\hookrightarrow \mathbb{L}_{n+1}=\mathbb{L}(x_{1},\ldots,x_{n+1})$
for each $i=0,1,\ldots,n+1$,
where we set $x_{0}:=-\sum_{i=1}^{n+1} x_{i}$, 
induces the map
\[
	\mathcal{I}_{\mathbb{L}(x_{i})}\otimes_{U(\mathbb{L}(x_{i}))} (V\otimes_{K} K_{\lambda})\to \mathcal{I}_{\mathbb{L}_{n+1}}\otimes_{U(\mathbb{L}_{n+1}))} (V\otimes_{K} K_{\lambda}),
\]
which
is injective since we have the commutative diagram
\[
	\begin{tikzcd}
	\mathcal{I}_{\mathbb{L}(x_{i})}\otimes_{U(\mathbb{L}(x_{i}))} (V\otimes_{K} K_{\lambda})
	\arrow[r]\arrow[d,"\cong"]& \mathcal{I}_{\mathbb{L}_{n+1}}\otimes_{U(\mathbb{L}_{n+1}))} (V\otimes_{K} K_{\lambda})\arrow[d,"\cong"]\\
	K\cdot x_{i}\otimes_{K} (V\otimes_{K} K_{\lambda}) \arrow[r,hookrightarrow] & \bigoplus_{j=1}^{n+1} K\cdot x_{j}\otimes_{K} (V\otimes_{K} K_{\lambda})
	\end{tikzcd}
\]
as $K$-vector spaces.
Then 
 the following commutative diagram 
\small
\begin{center}
	\begin{tikzcd}
	0\arrow[r] & H_1(\mathbb{L}(x_{i}), V\otimes_{K} K_{\lambda}) \arrow[r]  & \mathcal{I}_{\mathbb{L}(x_{i})}\otimes_{U(\mathbb{L}(x_{i}))} (V\otimes_{K} K_{\lambda}) \arrow[r] \arrow[d,hookrightarrow] & V\otimes_{K} K_{\lambda} \arrow[r] \arrow[d, equal] & 0\\
0\arrow[r] & H_1(\mathbb{L}_{n+1}, V\otimes_{K} K_{\lambda}) \arrow[r] & \mathcal{I}_{\mathbb{L}_{n+1}}\otimes_{U(\mathbb{L}_{n+1})} (V\otimes_{K} K_{\lambda}) \arrow[r] & V\otimes_{K} K_{\lambda} \arrow[r] & 0
	\end{tikzcd}.
\end{center}
\normalsize
with exact rows induces the injective map
\[
	H_1(\mathbb{L}(x_{i}), V\otimes_{K} K_{\lambda})\hookrightarrow H_1(\mathbb{L}_{n+1}, V\otimes_{K} K_{\lambda}),\quad 	
\]
for each $i=0,1,\ldots,n+1$. 
Let us recall the isomorphism 
\[
	H_{1}(\mathbb{L}(x),M)\cong K\cdot x\otimes_{K} M^{x}
\]
for $\mathbb{L}(x)$-module $M$, where $M^{x}$
denotes $x$-invariant part of $M$, namely, it consists of all elements $m\in M$ such that $x\cdot m=0$.
Now we note that $H_1(\mathbb{L}(x_{i}), V\otimes_{K} K_{\lambda})=0$ for $i=n+1$ since $x_{n+1}$ 
acts on $V\otimes_{K} K_{\lambda}$ as the scalar multiplication by $\lambda\neq 0$
and thus $(V\otimes_{K} K_{\lambda})^{x_{n+1}}=0$.
\begin{prop}\label{prop:klsubmodule}
For a finite dimensional $U(\mathbb{L}_{n}\oplus_{\theta_{g}}\mathfrak{g})$-module 
$V$,
\[
	\mathrm{Im}\left(
		H_1(\mathbb{L}(x_{i}), V\otimes_{K} K_{\lambda})\to H_1(\mathbb{L}_{n+1}, V\otimes_{K} K_{\lambda})
	\right)
\] 
is closed under the $U(\mathbb{L}_{n}\oplus_{\theta_{g}}\mathfrak{g})$-action for each $i=0,1,\ldots,n$.
\end{prop}
\begin{proof}
Under the isomorphism
\[
	H_{1}(\mathbb{L}(x_{i}), V\otimes_{K} K_{\lambda})\cong K\cdot x_{i}\otimes_{K} (V\otimes_{K} K_{\lambda})^{x_{i}},
\]
let us take $x_{i}\otimes v$ for $v\in (V\otimes_{K} K_{\lambda})^{x_{i}}$,
and show that
\[
	X\cdot (x_{i}\otimes v)\in K\cdot x_{i}\otimes_{K} (V\otimes_{K} K_{\lambda})^{x_{i}}
\]
for $X\in \mathbb{L}_{n}\oplus_{\theta_{g}}\mathfrak{g}$.
Here $\mathbb{L}_{n}\oplus_{\theta_{g}}\mathfrak{g}$-action 
is defined by regarding  
 $x_{i}\otimes v
\in \mathcal{I}_{U(\mathbb{L}_{n+1})}\otimes_{U(\mathbb{L}_{n+1})}(V\otimes_{K} K_{\lambda})$
through the injection 
$\mathcal{I}_{\mathbb{L}(x_{i})}\otimes_{U(\mathbb{L}(x_{i}))}(V\otimes_{K} K_{\lambda}) \to 
\mathcal{I}_{\mathbb{L}_{n+1}}\otimes_{U(\mathbb{L}_{n+1})}(V\otimes_{K} K_{\lambda})$.

For $i,j=1,\ldots,n$,
we have 
\begin{align*}
x_{j}\cdot (x_{i}\otimes v)&=\theta_{\mathrm{id}\oplus g}(x_{j})(x_{i})\otimes v
=\begin{cases}
[x_{i},x_{n+1}]\otimes v& i= j\\
0& i\neq j
\end{cases}
=\begin{cases}
x_{i}\otimes \lambda v& i=j\\
0& i\neq j
\end{cases}.
\end{align*}
Also for $i=0$, $\theta_{\mathrm{id}\oplus g}(x_{j})(x_{0})
=-\theta_{\mathrm{id}\oplus g}(x_{j})(x_{1}+\cdots x_{n+1})=0$
as we noted in below Definition \ref{Infbraidaction}.
Therefore in all cases, we have 
$x_{j}\cdot (x_{i}\otimes v)\in K\cdot x_{i}\otimes (V\otimes_{K}K_{\lambda})^{x_{i}}$
for $j=1,\ldots,n$, which
implies that 
$\mathbb{L}_{n}$
preserves $K\cdot x_{i}\otimes_{K} (V\otimes_{K}K_{\lambda})^{x_{i}}$ for $i=0,1,\ldots,n$.

Next, take $X\in \mathfrak{g}$. 
Lemma \ref{lem:purebraid} gives us $w\in \mathbb{L}_{n+1}$
for each $x_{i}$, $i=0,1,\ldots,n$
such that $\theta_{\mathrm{id}\oplus g}(X)(x_{i})=[x_{i},w]$.
This is also true for $i=0$ since $\theta_{\mathrm{id}\oplus g}(X)(x_{0})=-\theta_{\mathrm{id}\oplus g}(X)(x_{1}+\cdots+x_{n+1})=0=[x_{0},0]$.
Then we have  
\begin{align*}
	X\cdot (x_{i}\otimes v)&=[x_{i},w]\otimes v+x_{i}\otimes X\cdot v=x_{i}\otimes (w+X)\cdot v.
\end{align*}
Let us check $(w+X)\cdot v$ is $x_{i}$-invariant.
Indeed, we have 
\begin{align*}
	x_{i}\cdot ((w+X)\cdot v)&=((w+X)\cdot x_{i}+([x_{i},w]-\theta_{\mathrm{id}\oplus g}(X)(x_{i})))\cdot v\\
	&=((w+X)\cdot x_{i}+([x_{i},w]-[x_{i},w]))\cdot v=0.
\end{align*}
Thus we have 
$X\cdot (x_{i}\otimes v)\in K\cdot x_{i}\otimes_{K}(V\otimes_{K}K_{\lambda})^{x_{i}}$.
\end{proof}
This proposition allows us to regard  
\[
	\mathfrak{mc}_{\lambda}(V):=\mathrm{Coker}\left(
		\bigoplus_{i=0}^{n} H_1(\mathbb{L}(x_{i}), V\otimes_{K} K_{\lambda})\to H_1(\mathbb{L}(x_{1},\ldots,x_{n+1}), V\otimes_{K} K_{\lambda})	
	\right)
\]
as an $\mathbb{L}_{n}\oplus_{\theta_{g}}\mathfrak{g}$-module.

Also for $\lambda=0$, we can define $\mathfrak{mc}_{0}(V)$ as follows. 
For this purpose, let us give characterizations of the images of the maps
\[
	H_1(\mathbb{L}(x_{i}), V\otimes_{K} K_{\lambda})\to H_1(\mathbb{L}_{n+1}, V\otimes_{K} K_{\lambda})
\]
for $i=0,1,\ldots,n$.
\begin{lem}\label{lem:Linvariant}
	For $\lambda\in K\backslash\{0\}$, we have 
	\[
	\mathrm{Im}\left(
		H_1(\mathbb{L}(x_{0}), V\otimes_{K} K_{\lambda})\to H_1(\mathbb{L}_{n+1}, V\otimes_{K} K_{\lambda})
	\right)
	=
	H_1(\mathbb{L}_{n+1}, V\otimes_{K} K_{\lambda})	
	^{\mathbb{L}_{n}}.
	\]
	Here $M^{\mathbb{L}_{n}}$ denotes the submodule of an $\mathbb{L}_{n}$-module $M$ consisting of all elements $m\in M$ such that $x\cdot m=0$ for all $x\in \mathbb{L}_{n}$.
\end{lem}
\begin{proof}
It suffices to show that 
\[
	K\cdot x_{0}\otimes_{K} (V\otimes_{K} K_{\lambda})^{x_{0}}
	= \left(
	\bigoplus_{i=1}^{n}K\cdot [x_{i},x_{n+1}]\otimes_{K}(V\otimes_{K} K_{\lambda})
	\right)^{\mathbb{L}_{n}}.
\]

Since $\theta_{\mathrm{id}\oplus g}(X)(x_{0})=0$ for all $X\in \mathbb{L}_{n}\oplus_{\theta_{g}}\mathfrak{g}$, we have 
the inclusion $\subset$ is obvious.

Take an element $\sum_{i=1}^{n} [x_i, x_{n+1}]\otimes v_i
$ from the right hand side of our desired equation.
Then since
\begin{equation}\label{eq:DRformula}
\begin{aligned}
	&\theta_{\mathrm{id}\oplus \theta_{g}}(x_{j})([x_{i},x_{n+1}])\\
	&=[\theta_{\mathrm{id}\oplus \theta_{g}}(x_{j})(x_{i}),x_{n+1}]+[x_{j},\theta_{\mathrm{id}\oplus \theta_{g}}(x_{j})(x_{n+1})]\\
	&=\begin{cases}
	[[x_{j},x_{n+1}],x_{n+1}]+[x_{j},[x_{n+1},x_{i}]]=[[x_{j},x_{n+1}],x_{j}+x_{n+1}]&i=j\\
	[x_{i},[x_{n+1},x_{j}]]=[[x_{j},x_{n+1}],x_{i}]&i\neq j
	\end{cases}\\
	&=[[x_{j},x_{n+1}],x_{i}+\delta_{i,j}x_{n+1}],
\end{aligned}
\end{equation}
we have 
\begin{align*}
	0=x_{j}\cdot \left(\sum_{i=1}^{n} [x_i, x_{n+1}]\otimes v_i\right)
	&=\sum_{i=1}^{n} [x_j, x_{n+1}]\otimes (x_{i}+\delta_{i,j}x_{n+1})v_i\\
	&=[x_{j},x_{n+1}]\otimes \left(\sum_{i=1}^{n} (x_{i}+\delta_{i,j}x_{n+1})v_i\right)
\end{align*}
for $j=1,\ldots,n$.
Namely, we have 
$
	\sum_{i=1}^{n} (x_{i}+\delta_{i,j}x_{n+1})v_i=0
$
or equivalently, 
\[
	\sum_{i=1}^{n} x_{i}v_i=-x_{n+1}v_j=-\lambda v_j
\]
for all $j=1,\ldots,n$.

Since $\lambda\neq 0$, 
the above equation implies that $v_1=v_2=\cdots=v_n$, and
also implies that   
$v:=v_1=v_2=\cdots=v_n$ satisfies the equation $\sum_{i=1}^{n}x_{i}v=-x_{n+1}v$
which is equivalent to
\(
	\sum_{i=1}^{n+1} x_{i}v=0,
\)
i.e., 
\[
	v\in (V\otimes_{K}K_{\lambda})^{x_{0}}.
\]
Therefore we have
\[
	\sum_{i=1}^{n} [x_i, x_{n+1}]\otimes v_i
	=\sum_{i=1}^{n} [x_i, x_{n+1}]\otimes v
	=-[x_{0}, x_{n+1}]\otimes v
	=-x_{0}\otimes \lambda v.
\]
This shows the opposite inclusion.
\end{proof}
\begin{lem}\label{lem:K}
For $\lambda\in K\backslash\{0\}$ and $i=1,\ldots,n$, we have
\[
	\mathrm{Im}\left(
		H_1(\mathbb{L}(x_{i}), V\otimes_{K} K_{\lambda})\to H_1(\mathbb{L}_{n+1}, V\otimes_{K} K_{\lambda})
	\right)
	\cong K\cdot [x_{i},x_{n+1}]\otimes_{K} (V\otimes_{K} K_{\lambda})^{x_{i}}.
\]
\end{lem}
\begin{proof}
It suffices to show the equality 
\[
	K\cdot x_{i}\otimes (V\otimes_{K} K_{\lambda})^{x_{i}}=K\cdot [x_{i},x_{n+1}]\otimes (V\otimes_{K} K_{\lambda})^{x_{i}}.
\]
in $\mathcal{I}_{\mathbb{L}_{n+1}}\otimes_{U(\mathbb{L}_{n+1})}(V\otimes_{K} K_{\lambda})$.
This follows  easily from the equation
\begin{align*}
x_{i}\otimes v&=x_{i}\otimes (x_{n+1}\cdot \lambda^{-1}v)-x_{n+1}\otimes ((x_{i}\cdot \lambda^{-1}v))=
[x_{i},x_{n+1}]\otimes \lambda^{-1}v
\end{align*}
for $v\in (V\otimes_{K} K_{\lambda})^{x_{i}}$.
\end{proof}

Then for $\lambda=0$,
as a quotient space of 
\[
	\mathfrak{LM}_{0}^{\mathrm{df}}= 
	\bigoplus_{i=1}^{n} K\cdot [x_{i},x_{n+1}]\otimes_{K} V,
\]
let us define  
\[
	\mathfrak{mc}_{0}(V):=\mathfrak{LM}^{\mathrm{df}}_{0}(V)/
	\left(\sum_{i=1}^{n} \left(K\cdot [x_{i},x_{n+1}]\otimes_{K} V^{x_{i}}
	\right)+\mathfrak{LM}^{\mathrm{df}}_{0}(V)^{\mathbb{L}_{n}}\right).
\]
Here the invariant subspace $\mathfrak{LM}^{\mathrm{df}}_{0}(V)^{\mathbb{L}_{n}}$
also has the following characterization.
\begin{lem}\label{lem:Linvariant2}
We have the equality
\[
	\mathfrak{LM}^{\mathrm{df}}_{0}(V)^{\mathbb{L}_{n}}
	=
	\left\{
		\sum_{i=1}^{n} [x_{i},x_{n+1}]\otimes v_i\in \mathfrak{LM}^{\mathrm{df}}_{0}(V)\,\middle|\,
		\sum_{i=1}^{n} x_{i}v_i=0
	\right\}.
\]
\end{lem}
\begin{proof}
As we saw in the proof of Lemma \ref{lem:Linvariant}, 
$\sum_{i=1}^{n} [x_{i},x_{n+1}]\otimes v_i\in \mathfrak{LM}^{\mathrm{df}}_{0}(V)^{\mathbb{L}_{n}}$
if and only if 
\[
	[x_{j},x_{n+1}]\otimes \left(\sum_{i=1}^{n} x_{i}v_i\right)=0
\]
for all $j=1,\ldots,n$, which is equivalent to $\sum_{i=1}^{n} x_{i}v_i=0$.
\end{proof}

Now we are ready to define the middle convolution functor.
\begin{df}\label{df:midconv}
Let $(\mathfrak{g},g)$ be a infinitesimal braid action 
on $\mathbb{L}_{n}$, and $\lambda\in K$.
Let us define an endofunctor of $U(\mathbb{L}_{n}\oplus_{\theta_{g}}\mathfrak{g})$-$\mathbf{mod}$
by 
\[
	\begin{array}{cccc}
		\mathfrak{mc}_{\lambda}\colon &
		U(\mathbb{L}_{n}\oplus_{\theta_{g}}\mathfrak{g})\text{-}\mathbf{mod}
		&\longrightarrow &U(\mathbb{L}_{n}\oplus_{\theta_{g}}\mathfrak{g})\text{-}\mathbf{mod}\\
		&V&\longmapsto&
		\mathfrak{mc}_{\lambda}(V)
	\end{array},
\]
which is called the \emph{middle convolution functor} with respect to $\lambda$.
\end{df}
The functor $\mathfrak{mc}_{\lambda}$ is no longer exact, but it preserves injectivity and surjectivity.
\begin{prop}
For $\lambda\in K\backslash\{0\}$, 
the middle convolution functor $\mathfrak{mc}_{\lambda}$ preserves injectivity and surjectivity, i.e., 
for a morphism $f\colon V\to W$ in $U(\mathbb{L}_{n}\oplus_{\theta_{g}}\mathfrak{g})\text{-}\mathbf{mod}$, if $f$ is injective (resp. surjective), 
then $\mathfrak{mc}_{\lambda}(f)\colon \mathfrak{mc}_{\lambda}(V)\to \mathfrak{mc}_{\lambda}(W)$ is also injective (resp. surjective).
\end{prop}
\begin{proof}
Let us consider the following commutative diagram with exact rows:
\small
\begin{center}
	\begin{tikzcd}
	0\arrow[r] & \bigoplus_{i=0}^{n} H_1(\mathbb{L}(x_{i}), V\otimes_{K} K_{\lambda}) \arrow[r]  \arrow[d,"\phi"] & H_1(\mathbb{L}_{n+1}, V\otimes_{K} K_{\lambda}) \arrow[r] \arrow[d,"\Phi"] & \mathfrak{mc}_{\lambda}(V) \arrow[r] \arrow[d,"\overline{\Phi}"] & 0\\
	0\arrow[r] & \bigoplus_{i=0}^{n} H_1(\mathbb{L}(x_{i}), W\otimes_{K} K_{\lambda}) \arrow[r]  & H_1(\mathbb{L}_{n+1}, W\otimes_{K} K_{\lambda}) \arrow[r] & \mathfrak{mc}_{\lambda}(W) \arrow[r] & 0
	\end{tikzcd}.
\end{center}
\normalsize
Here the vertical maps $\phi$, $\Phi$ 
are induced map from $f\colon V\to W$
and the left square commutes, 
and thus $\overline{\Phi}$
is induced by the universality of the cokernel from the commutative diagram of exact rows, which also makes the right square commutative.
Since the deformed Long-Moody functor $\mathfrak{LM}^{\mathrm{df}}_{\lambda}$ is exact, the middle vertical map $\Phi$ is injective (resp. surjective) if $f$ is injective (resp. surjective).
Therefore, if $f$ is surjective, then the right vertical map is obviously surjective as well.

Next we suppose that $f$ is injective and show that 
the induced map 
$\mathrm{Coker}(\phi)\rightarrow \mathrm{Coker}(\Phi)$
is injective, which implies that the right vertical map is also injective by the snake lemma.

Let us consider the
short exact sequence
\[
0\to V\to W\to W/V\to 0.
\]
By the long exact sequence of homology groups, we have the following exact sequence:
\[
	0\to \bigoplus_{i=0}^{n} H_1(\mathbb{L}(x_{i}), V\otimes_{K} K_{\lambda}) \xrightarrow{\phi} \bigoplus_{i=0}^{n} H_1(\mathbb{L}(x_{i}), W\otimes_{K} K_{\lambda}) \to \bigoplus_{i=0}^{n} H_1(\mathbb{L}(x_{i}), (W/V)\otimes_{K} K_{\lambda}),	
\]
which induces the injective map 
\[
	\mathrm{Coker}(\phi)
	\hookrightarrow \bigoplus_{i=0}^{n} H_1(\mathbb{L}(x_{i}), (W/V)\otimes_{K} K_{\lambda}).
\]
By combining with the injection $\bigoplus_{i=0}^{n} H_1(\mathbb{L}(x_{i}), (W/V)\otimes_{K} K_{\lambda})
\hookrightarrow H_{1}(\mathbb{L}_{n+1}, (W/V)\otimes_{K} K_{\lambda})$,
we obtain the injective map
\[
	\mathrm{Coker}(\phi)
	\hookrightarrow H_{1}(\mathbb{L}_{n+1}, (W/V)\otimes_{K} K_{\lambda}).
\]
Moreover, this injection makes the following diagram commutative:
\small
\begin{center}
	\begin{tikzcd}
	 \bigoplus_{i=0}^{n} H_1(\mathbb{L}(x_{i}), V\otimes_{K} K_{\lambda}) \arrow[r]  \arrow[d,"\phi"] & H_1(\mathbb{L}_{n+1}, V\otimes_{K} K_{\lambda})  \arrow[d,"\Phi"]\\
	 \bigoplus_{i=0}^{n} H_1(\mathbb{L}(x_{i}), W\otimes_{K} K_{\lambda}) \arrow[r]  \arrow[d]& H_1(\mathbb{L}_{n+1}, W\otimes_{K} K_{\lambda}) \arrow[d]\\
	\mathrm{Coker}(\phi) \arrow[r]  & H_{1}(\mathbb{L}_{n+1}, (W/V)\otimes_{K} K_{\lambda}) 
	\end{tikzcd}
\end{center}
\normalsize
by the construction.
Therefore, since $H_{1}(\mathbb{L}_{n+1}, (W/V)\otimes_{K} K_{\lambda})\cong \mathrm{Coker}(\Phi)$
by the exacteness of the deformed Long-Moody functor,
the universality of the cokernel implies that
this injective map $\mathrm{Coker}(\phi)
	\hookrightarrow H_{1}(\mathbb{L}_{n+1}, (W/V)\otimes_{K} K_{\lambda})$
coincides with the induced map $\mathrm{Coker}(\phi)\to \mathrm{Coker}(\Phi)$
from the commutative diagram of exact rows, which is thus injective.
\end{proof}

\subsection{Comparison with the Dettweiler-Reiter additive middle convolution}
Let us recall the additive middle convolution 
introduced by Dettweiler and Reiter in \cite{DR00,DR07}. 
Let us take a finite dimensional 
$U(\mathbb{L}_{n})$-module $V$ or equivalently 
a Lie algebra 
representation $\rho\colon \mathbb{L}_{n}\rightarrow \mathrm{End}_{K}(V)$.
Then for $\lambda\in K$,
Dettweiler and Reiter defined the representation 
$c^{\mathrm{DR}}_{\lambda}(\rho)\colon \mathbb{L}_{n}\rightarrow \mathrm{End}_{K}(V\otimes_{K}K^{n})$,
called the \emph{additive convolution of $\rho$}, by setting 
\[
	c^{\mathrm{DR}}_{\lambda}(\rho)(x_{i}):=\sum_{j=1}^{n}
	\left(
		(\rho(x_{j})+\lambda\delta_{i,j}\mathrm{id}_{V})\otimes I_{i,j}
	\right)
\]
for $i=1,\ldots,j$. Here $I_{i,j}\in \mathrm{End}_{K}(K^{n})=M_{n}(K)$
denotes the $(i,j)$-matrix unit for each $1\le i,j\le n$.
This correspondence is functorial as it is easy to be checked,
and defines a functor of 
the category $\mathbb{L}_{n}\text{-\textbf{rep}}$
of finite dimensional $\mathbb{L}_{n}$-representations.

\begin{df}[Dettweiler-Reiter additive convolution]
	For $\lambda\in K$, let us define an endofunctor 
	of $\mathbb{L}_{n}\text{-}\mathbf{rep}$ by 
	\[
		\begin{array}{cccc}
			c^{\mathrm{DR}}_{\lambda}\colon &\mathbb{L}_{n}\text{-}\mathbf{rep}&
			\longrightarrow &\mathbb{L}_{n}\text{-}\mathbf{rep}\\
			&\rho&\longmapsto&c^{\mathrm{DR}}_{\lambda}(\rho),
		\end{array}
	\]
	which is called the \emph{Dettweiler-Reiter additive convolution}.
\end{df}
Through the standard isomorphism 
$\mathbb{L}_{n}\text{-}\mathbf{rep}\cong 
U(\mathbb{L}_{n})\text{-}\mathbf{mod}$,
we identify them.
Let us denote the restriction functor with respect to the inclusion $\mathbb{L}_{n}\hookrightarrow \mathbb{L}_{n}\oplus_{\theta_{g}}\mathfrak{g}$ 
by 
\[
	\mathrm{Res}_{\mathbb{L}_{n}}^{\mathbb{L}_{n}\oplus_{\theta_{g}}\mathfrak{g}}
	\colon U(\mathbb{L}_{n}\oplus_{\theta_{g}}\mathfrak{g})\text{-}\mathbf{mod}
	\longrightarrow U(\mathbb{L}_{n})\text{-}\mathbf{mod}.
\]
Then we can relate the deformed Long-Moody functor $\mathfrak{LM}^{\mathrm{df}}_{\lambda}$ and the Dettweiler-Reiter additive convolution $c^{\mathrm{DR}}_{\lambda}$ as follows.
\begin{prop}\label{prop:compconv}
For $\lambda\in K$, there exists a natural isomorphism of functors
\[
	\mathrm{Res}_{\mathbb{L}_{n}}^{\mathbb{L}_{n}\oplus_{\theta_{g}}\mathfrak{g}}\circ \mathfrak{LM}^{\mathrm{df}}_{\lambda}
	\cong c^{\mathrm{DR}}_{\lambda}\circ \mathrm{Res}_{\mathbb{L}_{n}}^{\mathbb{L}_{n}\oplus_{\theta_{g}}\mathfrak{g}}.
\]
\end{prop}
\begin{proof}
Let us take $V$ to be a finite dimensional $\mathbb{L}_{n}\oplus_{\theta_{g}}\mathfrak{g}$-module
and show that there exists a natural isomorphism of $\mathbb{L}_{n}$-modules
\[
	\mathrm{Res}_{\mathbb{L}_{n}}^{\mathbb{L}_{n}\oplus_{\theta_{g}}\mathfrak{g}}(\mathfrak{LM}^{\mathrm{df}}_{\lambda}(V))
	\cong c^{\mathrm{DR}}_{\lambda}(\mathrm{Res}_{\mathbb{L}_{n}}^{\mathbb{L}_{n}\oplus_{\theta_{g}}\mathfrak{g}}(V)).
\]

Recall the $K$-isomorphism
\[
	\mathfrak{LM}_{\lambda}^{\mathrm{df}}(V)\cong \bigoplus_{i=1}^{n} K\cdot [x_{i},x_{n+1}]\otimes_{K} (V\otimes_{K}K_{\lambda}),
\]
and define a $K$-isomorphism
\begin{align*}
	\Phi\colon \mathfrak{LM}_{\lambda}^{\mathrm{df}}(V)= 
	\bigoplus_{i=1}^{n} K\cdot [x_{i},x_{n+1}]\otimes_{K} (V\otimes_{K}K_{\lambda})
	&\longrightarrow V\otimes_{K} K^{n};\\
	[x_{i},x_{n+1}]\otimes v&\longmapsto v\otimes e_{i},
\end{align*}
where $e_{i}\in K^{n}$ is the $i$-th standard basis vector for $i=1,\ldots,n$.
Then by regarding  the right hand side as the $\mathbb{L}_{n}$-module under the identification
$c^{\mathrm{DR}}_{\lambda}(V)=V\otimes_{K}K^{n}$,
we need to do is showing  that $\Phi$ interwines the $\mathbb{L}_{n}$-module structures.

By recalling the equation \eqref{eq:DRformula},
we have
\begin{align*}
	\Phi(x_{i}\cdot ([x_{j},x_{n+1}]\otimes v))&=
	\Phi(\theta_{\mathrm{id}\oplus \theta_{g}}(x_{i})([x_{j},x_{n+1}])\otimes v)\\
	&=\Phi([[x_{i},x_{n+1}],x_{j}+\delta_{i,j}x_{n+1}]\otimes v)\\
	&=\Phi([x_{i},x_{n+1}]\otimes (x_{j}+\delta_{i,j}\lambda)\cdot v)\\
	&=((x_{j}+\delta_{i,j}\lambda)\cdot v)\otimes e_{i}\\
	&=\sum_{k=1}^{n} ((x_{k}+\delta_{i,k}\lambda)\cdot v)\otimes I_{i,k}(e_{j})\\
	&=c^{\mathrm{DR}}_{\lambda}(x_{i})(v\otimes e_{j})\\
	&=x_{i}\cdot \Phi([x_{j},x_{n+1}]\otimes v),
\end{align*}
as desired.	
\end{proof}

Dettweiler and Reiter further defined the middle convolution functor $\mathrm{mc}_{\lambda}^{\mathrm{DR}}$ as follows.
For a finite dimensional $\mathbb{L}_{n}$-module $V$, 
let us consider  the following $\mathbb{L}_{n}$-submodule of $c^{\mathrm{DR}}_{\lambda}(V)=
V\otimes_{K}K^{n}$:
\begin{align*}
	\mathfrak{k}&:=\bigoplus_{i=1}^{n} V^{x_{i}}\otimes _{K}K\cdot e_{i},&
	\mathfrak{l}&:=c^{\mathrm{DR}}_{\lambda}(V)^{\mathbb{L}_{n}}.
\end{align*}
And then the quotient module is defined:
\[
	\mathrm{mc}^{\mathrm{DR}}_{\lambda}(V):=c^{\mathrm{DR}}_{\lambda}(V)/(\mathfrak{k}+\mathfrak{l}).
\]
\begin{df}[Dettweiler-Reiter additive middle convolution]
	For $\lambda\in K$, let us define an endofunctor 
	of $U(\mathbb{L}_{n})\text{-\textbf{mod}}$ by 
	\[
		\begin{array}{cccc}
			\mathrm{mc}^{\mathrm{DR}}_{\lambda}\colon &U(\mathbb{L}_{n})\text{-\textbf{mod}}&
			\longrightarrow &U(\mathbb{L}_{n})\text{-\textbf{mod}}\\
			&V&\longmapsto&\mathrm{mc}^{\mathrm{DR}}_{\lambda}(V),
		\end{array}
	\]
	which is called the \emph{Dettweiler-Reiter additive middle convolution}.
\end{df}
\begin{thm}\label{thm:compmc}
For $\lambda\in K$, there exists a natural isomorphism of functors
\[
	\mathrm{Res}_{\mathbb{L}_{n}}^{\mathbb{L}_{n}\oplus_{\theta_{g}}\mathfrak{g}}\circ \mathfrak{mc}_{\lambda}
	\cong \mathrm{mc}^{\mathrm{DR}}_{\lambda}\circ \mathrm{Res}_{\mathbb{L}_{n}}^{\mathbb{L}_{n}\oplus_{\theta_{g}}\mathfrak{g}}.
\]
\end{thm}
\begin{proof}
	We can check that the isomorphism $\Phi$ in the proof of Proposition \ref{prop:compconv} induces the isomorphisms
	\begin{align*}
		&\Phi\colon \bigoplus_{i=1}^{n} K[x_{i},x_{n+1}]\otimes_{K}(V\otimes_{K}K_{\lambda})^{x_{i}}\to \bigoplus_{i=1}^{n} V^{x_{i}}\otimes_{K} K\cdot e_{i},\\
		&\Phi\colon \mathfrak{LM}^{\mathrm{df}}_{\lambda}(V)^{\mathbb{L}_{n}}\to c^{\mathrm{DR}}_{\lambda}(V)^{\mathbb{L}_{n}},
	\end{align*}
	as $\mathbb{L}_{n}$-modules.
	Then since the functors $\mathfrak{LM}^{\mathrm{df}}_{\lambda}$ and $c^{\mathrm{DR}}_{\lambda}$ are exact, 
	the natural isomorphism in  Proposition \ref{prop:compconv}
	induces the desired natural isomorphism of functors
	by Lemmas \ref{lem:Linvariant} and \ref{lem:K}.
\end{proof}
\subsection{Composition law of middle convolution functors}
In \cite{DR00,DR07}, Dettweiler and Reiter showed the following remarkable property
as an analogue of the composition law of Katz's middle convolution functors for local systems \cite{Katz}.
A detailed treatment of this property is also found in \cite{Harbook}.
\begin{thm}[Dettweiler-Reiter \cite{DR00,DR07}]\label{thm:compmcDR}
Let us take $\lambda,\mu\in K$, and an $U(\mathbb{L}_{n})$-module $V$.
Suppose that $V$ satisfies the following conditions:
\begin{itemize}
	\item[$(*)$] 
	\[
		V^{x_{i}-c}\cap \bigcap_{j\neq i}^{n} V^{x_{j}}=0,\quad \text{for any }c\in K\text{ and }i=1,2,\ldots,n,
	\]
	\item[$(**)$] 
	\[
		\left((x_{i}-c)\cdot V+\sum_{j\neq i}^{n} x_{j}\cdot V\right)= V,\quad \text{for any }c\in K\text{ and }i=1,2,\ldots,n.
	\]
\end{itemize}
Then there exists isomorphisms of $U(\mathbb{L}_{n})$-modules
\[
	\mathrm{mc}^{\mathrm{DR}}_{\lambda}\circ \mathrm{mc}^{\mathrm{DR}}_{\mu}(V)\cong \mathrm{mc}^{\mathrm{DR}}_{\lambda+\mu}(V),\quad \mathrm{mc}^{\mathrm{DR}}_{0}(V)\cong V.
\]
\end{thm}
\begin{rem}
The above conditions $(*)$ and $(**)$ are obviously satisfied if $V$ is irreducible
and $\mathrm{dim}_{K}V\ge 1$, or $\mathrm{dim}_{K}V=1$ with 
at least two $x_i$ acting on $V$ non-trivially.
\end{rem}
A remarkable consequence of the above theorem is that
the middle convolution functors preseve irreducibility.
\begin{cor}[Dettweiler-Reiter \cite{DR00}]\label{cor:compmcDR}
	For $V\in U(\mathbb{L}_{n})\text{-}{\mathrm{\bf{mod}}}$ satisfying the conditions $(*)$ and $(**)$ in Theorem \ref{thm:compmcDR},
	 and $\lambda\in K$, 
	$V$ is irreducible if and only if $\mathrm{mc}^{\mathrm{DR}}_{\lambda}(V)$ is irreducible.
\end{cor}
If $\mathbb{L}_{n}\oplus_{\theta_{g}}\mathfrak{g}$-module 
$V$ is irreducible as an $\mathbb{L}_{n}$-module, then obviously
irreducible as an $\mathbb{L}_{n}\oplus_{\theta_{g}}\mathfrak{g}$-module as well.
This corollary implies that the middle convolution functor
$\mathfrak{mc}_{\lambda}$ produces irreducible 
$U(\mathbb{L}_{n}\oplus_{\theta_{g}}\mathfrak{g})\text{-}\mathbf{mod}$ objects.
\begin{cor}\label{cor:irred}
	Take $V\in U(\mathbb{L}_{n}\oplus_{\theta_{g}}\mathfrak{g})\text{-}{\mathrm{\bf{mod}}}$
	satisfying the conditions $(*)$ and $(**)$ in Theorem \ref{thm:compmcDR},
	 and $\lambda\in K$.
	If $V$ is irreducible as a $U(\mathbb{L}_{n})$-module, then $\mathfrak{mc}_{\lambda}(V)$ is irreducible as a 
	$U(\mathbb{L}_{n})$-module again and hence also irreducible as a
	$U(\mathbb{L}_{n}\oplus_{\theta_{g}}\mathfrak{g})$-module.
	Conversely, if $\mathfrak{mc}_{\lambda}(V)$ is irreducible as a $U(\mathbb{L}_{n})$-module, 
	then $V$ is irreducible as a 
	$U(\mathbb{L}_{n})$-module again and hence also irreducible as a
	$U(\mathbb{L}_{n}\oplus_{\theta_{g}}\mathfrak{g})$-module.
\end{cor}
\begin{proof}
Suppose that $V$ is irreducible as a $U(\mathbb{L}_{n})$-module.
Then by Theorem \ref{thm:compmc} and Corollary \ref{cor:compmcDR}, 
$\mathfrak{mc}_{\lambda}(V)$ is irreducible as a $U(\mathbb{L}_{n})$-module, which implies that $\mathfrak{mc}_{\lambda}(V)$ is irreducible as a $U(\mathbb{L}_{n}\oplus_{\theta_{g}}\mathfrak{g})$-module. 
The converse is also shown by the same argument.
\end{proof}

In this section, 
we explain that this composition law is 
also valid for the middle convolution functor $\mathfrak{mc}_{\lambda}$ of $U(\mathbb{L}_{n}\oplus_{\theta_{g}}\mathfrak{g})\text{-}\mathbf{mod}$
as well.

Let us  
consider a $K$-linear map defined as follows, 
which was introduce by Dettweiler and Reiter in \cite{DR00} to 
establish the composition law, and 
see also \cite{Harbook} for a detailed property of this map.
Combining the projection map  
\[
	\mathrm{pr}_{x_{n+1}}\colon \mathcal{I}_{\mathbb{L}_{n+1}}\otimes_{U(\mathbb{L}_{n+1})} M
	\cong \bigoplus_{i=1}^{n+1} K\cdot x_{i}\otimes_{K} M
	\longrightarrow K\cdot x_{n+1}\otimes_{K} M
\]
with the  $K$-isomorphism $K\cdot x_{n+1}\otimes_{K} M\cong M
:\ x_{n+1}\otimes m\to m$,
we obtain the $K$-linear map
\[
	\phi_{M} \colon \mathcal{I}_{\mathbb{L}_{n+1}}\otimes_{U(\mathbb{L}_{n+1})} M
	\rightarrow M.
\]
Here a remarkable point is that 
\[
	\phi_{V\otimes_{K}K_{\lambda}}\colon     
\mathcal{I}_{\mathbb{L}_{n+1}}\otimes_{U(\mathbb{L}_{n+1})} (V\otimes_{K}K_{\lambda})
\rightarrow V\otimes_{K}K_{\lambda}
\]
is a $\mathfrak{g}$-module homomorphism.
Indeed, for $X\in \mathfrak{g}$ and 
$\sum_{i=1}^{n+1} x_i\otimes v_{i}\in \mathcal{I}_{\mathbb{L}_{n+1}}\otimes_{U(\mathbb{L}_{n+1})} (V\otimes_{K}K_{\lambda})$
with $v_{i}\in V\otimes_{K}K_{\lambda}$, we have
\begin{align*}
	\phi_{V\otimes_{K}K_{\lambda}}(X\cdot (\sum_{i=1}^{n+1} x_i\otimes v_{i}))
	&=\phi_{V\otimes_{K}K_{\lambda}}(\sum_{i=1}^{n+1} \theta_{\mathrm{id}\oplus \theta_{g}}(X)(x_i)\otimes v_{i}+\sum_{i=1}^{n+1} x_i\otimes X\cdot v_{i})\\
	&= X\cdot v_{n+1}\\
	&=X\cdot \phi_{V\otimes_{K}K_{\lambda}}(\sum_{i=1}^{n+1} x_i\otimes v_{i}),
\end{align*}
as desired.
Here for the second equation, we used the fact that
\[
	\theta_{\mathrm{id}\oplus \theta_{g}}(X)(x_{n+1})=0\quad \text{and}\quad \theta_{\mathrm{id}\oplus \theta_{g}}(X)(x_{i})\in \mathbb{L}(x_{1},\ldots,x_{n})
	\text{ for }i=1,\ldots,n,
\]
which follow from the infinitesimal braid relation and Lemma \ref{lem:minpurebraid}.

\begin{thm}\label{thm:mc0}
	Suppose $\lambda=0$.
	Then $\phi_{V}$ induces an $\mathbb{L}_{n}\oplus_{\theta_{g}}\mathfrak{g}$-module homomorphism
	\[
		\mathfrak{mc}_{0}(V)\to V,
	\]
	which is an isomorphism 
	under the assumptions $(*),\, (**)$ in Theorem \ref{thm:compmcDR}.
\end{thm}
\begin{proof}
By restricting $\phi_{V}$ to $\mathfrak{LM}_{0}^{\mathrm{df}}(V)$, we obtain the $\mathfrak{g}$-module homomorphism
\begin{equation}\label{eq:phiV}
	\begin{array}{cccc}	
\phi_{V}|_{\mathfrak{LM}_{0}^{\mathrm{df}}(V)}\colon &\mathfrak{LM}_{0}^{\mathrm{df}}(V)=\bigoplus_{i=1}^{n} K\cdot [x_{i},x_{n+1}]\otimes_{K} V&\longmapsto &V\\
&[x_{i},x_{n+1}]\otimes v&\longmapsto &x_{i}\cdot v
	\end{array}.
\end{equation}
This is an $\mathbb{L}_{n}$-module homomorphism as well. Indeed, by the equation $\eqref{eq:DRformula}$, we have
\[
	x_{i}\cdot ([x_{j},x_{n+1}]\otimes v)=[[x_{i},x_{n+1}],x_{j}+\delta_{i,j}x_{n+1}]\otimes v=[x_{i},x_{n+1}]\otimes (x_{j}\cdot v).
\]
Then we have 
\begin{align*}
	\phi_{V}(x_{i}\cdot ([x_{j},x_{n+1}]\otimes v))&=\phi_{V}([x_{i},x_{n+1}]\otimes (x_{j}\cdot v))=x_{i}\cdot (x_{j}\cdot v)=x_{i}\cdot \phi_{V}([x_{j},x_{n+1}]\otimes v)
\end{align*}
as desired. 
Thus $\phi_{V}|_{\mathfrak{LM}_{0}^{\mathrm{df}}(V)}$ is an $\mathbb{L}_{n}\oplus_{\theta_{g}}\mathfrak{g}$-module homomorphism.

Next we see that 
$\phi_{V}|_{\mathfrak{LM}_{0}^{\mathrm{df}}(V)}$
sends $\sum_{i=1}^{n} K\cdot [x_{i},x_{n+1}]\otimes_{K} V^{x_{i}}$ to $0$,
which is obvious from the definition $\eqref{eq:phiV}$ of $\phi_{V}|_{\mathfrak{LM}_{0}^{\mathrm{df}}(V)}$.
Also Lemma \ref{lem:Linvariant2} implies that $\mathfrak{LM}_{0}^{\mathrm{df}}(V)^{\mathbb{L}_{n}}$ 
is the kernel of $\phi_{V}|_{\mathfrak{LM}_{0}^{\mathrm{df}}(V)}$.
Therefore $\phi_{V}|_{\mathfrak{LM}_{0}^{\mathrm{df}}(V)}$ factors through the $\mathbb{L}_{n}\oplus_{\theta_{g}}\mathfrak{g}$-module homomorphism
\[
	\overline{\phi}_{V}\colon \mathfrak{mc}_{0}(V)
	=\mathfrak{LM}_{0}^{\mathrm{df}}(V)/\left(\sum_{i=1}^{n} K\cdot [x_{i},x_{n+1}]\otimes_{K} V^{x_{i}}
	+\mathfrak{LM}_{0}^{\mathrm{df}}(V)^{\mathbb{L}_{n}}\right)
	\longrightarrow V.
\]

Moreover, under the conditions $(*),\,(**)$, Dettweiler and Reiter \cite{DR00} showed that 
the induced map $\overline{\phi}_{V}$ is
bijective (see Proposition 3.2 in \cite{DR00}, and also see Section 7.5 in \cite{Harbook}).
\end{proof}

For $\lambda,\mu\in K$, we further define the following $\mathfrak{g}$-module homomorphism
\begin{multline*}
	\phi_{\lambda,\mu}\colon
	\mathcal{I}_{\mathbb{L}_{n+1}}\otimes_{U(\mathbb{L}_{n+1})} ((\mathcal{I}_{\mathbb{L}_{n+1}}\otimes_{U(\mathbb{L}_{n+1})} (V\otimes_{K}K_{\mu}))\otimes_{K} K_{\lambda})	
	\\
	\longrightarrow
	\mathcal{I}_{\mathbb{L}_{n+1}}\otimes_{U(\mathbb{L}_{n+1})} (V\otimes_{K}K_{\lambda+\mu}),
\end{multline*}
as the combination of $\phi_{(\mathcal{I}_{\mathbb{L}_{n+1}}\otimes_{U(\mathbb{L}_{n+1})} (V\otimes_{K}K_{\mu}))\otimes_{K} K_{\lambda}}$
and the natural $K$-isomorphism
\[
	\mathcal{I}_{\mathbb{L}_{n+1}}\otimes_{U(\mathbb{L}_{n+1})} ((V\otimes_{K}K_{\mu})
	\otimes_{K} K_{\lambda})
	\cong
	\mathcal{I}_{\mathbb{L}_{n+1}}\otimes_{U(\mathbb{L}_{n+1})} (V\otimes_{K}K_{\lambda+\mu}),
\]
which is moreover $\mathfrak{g}$-isomorphism since $\mathfrak{g}$ kills $x_{n+1}$
and preserves $\mathbb{L}_{n}(x_{1},\ldots,x_{n})$, see Lemma \ref{lem:minpurebraid}.
Therefore this $\phi_{\lambda,\mu}$ is a $\mathfrak{g}$-homomorpism as well.

\begin{thm}\label{thm:compmc2}
The above $\phi_{\mu,\lambda}$ 
induces an $\mathbb{L}_{n}\oplus_{\theta_{g}}\mathfrak{g}$-homomorphism 
\[
	\mathfrak{mc}_{\lambda}(
	\mathfrak{mc}_{\mu}(V))
	\rightarrow 
	\mathfrak{mc}_{\lambda+\mu}(V),
\]
which is an isomorphism under the assumptions $(*),\,(**)$ in Theorem \ref{thm:compmcDR}.
\end{thm}
\begin{proof}
By restricting $\phi_{(\mathcal{I}_{\mathbb{L}_{n+1}}\otimes_{U(\mathbb{L}_{n+1})} (V\otimes_{K}K_{\mu}))\otimes_{K} K_{\lambda}}$
to $\mathfrak{LM}_{\lambda}^{\mathrm{df}}(
\mathfrak{LM}_{\mu}^{\mathrm{df}}(V))$,
we obtain the $\mathfrak{g}$-module homomorphism 
\[
	\mathfrak{LM}_{\lambda}^{\mathrm{df}}(
	\mathfrak{LM}_{\mu}^{\mathrm{df}}(V))
	\rightarrow 
	(\mathcal{I}_{\mathbb{L}_{n+1}}\otimes_{U(\mathbb{L}_{n+1})}
	(V\otimes_{K} K_{\mu}))\otimes_{K}K_{\lambda}
\]
given by 
\[
	[x_{i},x_{n+1}]\otimes
	([x_{j},x_{n+1}]\otimes (v\otimes 1_{\mu}))\otimes 1_{\lambda}
	\mapsto 
	x_{i}\cdot (([x_{j},x_{n+1}]\otimes (v\otimes 1_{\mu}))\otimes 1_{\lambda}).
\]
By recalling the equation $\eqref{eq:DRformula}$,
we have 
\begin{align*}
&x_{k}\cdot([x_{i},x_{n+1}]\otimes
	([x_{j},x_{n+1}]\otimes v\otimes 1_{\mu})\otimes 1_{\lambda})\\
	&=[x_{k},x_{n+1}]\otimes 
	\begin{cases}
	(x_{i}+\lambda)\cdot([x_{j},x_{n+1}]\otimes v\otimes 1_{\mu})\otimes 1_{\lambda}
	&(k=i)\\
	x_{i}\cdot ([x_{j},x_{n+1}]\otimes v\otimes 1_{\mu})\otimes 1_{\lambda}
	&(k\neq i)
	\end{cases}\\
	&=[x_{k},x_{n+1}]\otimes\\
	&\quad \begin{cases}
	[x_{i},x_{n+1}]\otimes (x_{i}+\lambda+\mu)v \otimes1_{\mu}\otimes 1_{\lambda}
	&(k=i=j)\\
	[x_{j},x_{n+1}]\otimes \lambda v\otimes 1_{\mu}\otimes 1_{\lambda}+
	[x_{i},x_{n+1}]\otimes x_{j}v\otimes 1_{\mu}\otimes 1_{\lambda}
	&(k=i, i\neq j)\\
	x_{i}\cdot ([x_{j},x_{n+1}]\otimes v\otimes 1_{\mu})\otimes 1_{\lambda}
	&(k\neq i)
	\end{cases}.
\end{align*}
Therefore $x_{k}\cdot([x_{i},x_{n+1}]\otimes
	([x_{j},x_{n+1}]\otimes v\otimes 1_{\mu})\otimes 1_{\lambda})$
	is sent to the following element
	by $\phi_{(\mathcal{I}_{\mathbb{L}_{n+1}}\otimes_{U(\mathbb{L}_{n+1})} (V\otimes_{K}K_{\mu}))\otimes_{K} K_{\lambda}}$:
\[
\begin{cases}
	[x_{k},x_{n+1}]\otimes(x_{i}+\lambda+\mu)^{2}v\otimes 1_{\mu}\otimes 1_{\lambda}
	&(k=i=j)\\
	[x_{k},x_{n+1}]\otimes (x_{i}+\lambda+\mu)x_{j}v\otimes 1_{\mu}\otimes 1_{\lambda}
	&(k=i, i\neq j)\\
	x_{k}x_{i}\cdot ([x_{j},x_{n+1}]\otimes v\otimes 1_{\mu})\otimes 1_{\lambda}
	&(k\neq i)
\end{cases}.
\]
Namely, $\phi_{\lambda,\mu}$ send $x_{k}\cdot ([x_{i},x_{n+1}]\otimes
	([x_{j},x_{n+1}]\otimes v\otimes 1_{\mu})\otimes 1_{\lambda})$ to  
\begin{equation*}
\begin{cases}
	[x_{k},x_{n+1}]\otimes (x_{i}+\lambda+\mu)^{2}v\otimes 1_{\lambda+\mu}
	&(k=i=j)\\
	[x_{k},x_{n+1}]\otimes (x_{i}+\lambda+\mu)x_{j}v\otimes 1_{\lambda+\mu}
	&(k=i, i\neq j)\\
	x_{k}x_{i}\cdot ([x_{j},x_{n+1}]\otimes v\otimes 1_{\lambda+\mu})
	&(k\neq i)
\end{cases}
=x_{k}x_{i}\cdot ([x_{j},x_{n+1}]\otimes v\otimes 1_{\lambda+\mu}).
\end{equation*}	
As a consequence, we have 
\begin{equation*}
\phi_{\lambda,\mu}(x_{k}\cdot ([x_{i},x_{n+1}]\otimes
	([x_{j},x_{n+1}]\otimes v\otimes 1_{\mu})\otimes 1_{\lambda}))=\\
	x_{k}\cdot \phi_{\lambda,\mu}([x_{i},x_{n+1}]\otimes
	([x_{j},x_{n+1}]\otimes v\otimes 1_{\mu})\otimes 1_{\lambda}),
\end{equation*}
i.e., $\phi_{\lambda,\mu}$ is an $\mathbb{L}_{n}$-homomorphism,
and we moreover have 
$\mathrm{Im}\phi_{\lambda,\mu}\subset \mathfrak{LM}^{\mathrm{df}}_{\lambda+\mu}$.
Therefore $\phi_{\lambda,\mu}$ restricts to the $\mathbb{L}_{n}\oplus_{\theta_{g}}\mathfrak{g}$-homomorphism
\[
	\phi_{\lambda,\mu}|_{\mathfrak{LM}_{\lambda}^{\mathrm{df}}(\mathfrak{LM}_{\mu}^{\mathrm{df}}(V))}\colon \mathfrak{LM}_{\lambda}^{\mathrm{df}}(
	\mathfrak{LM}_{\mu}^{\mathrm{df}}(V))
	\rightarrow 
	\mathfrak{LM}_{\lambda+\mu}^{\mathrm{df}}(V).
\]

Next we check that this homomorphism induces $\mathfrak{mc}_{\lambda}(\mathfrak{mc}_{\mu}(V))
\rightarrow \mathfrak{mc}_{\lambda+\mu}(V)$.
As we saw in Proposition \ref{prop:klsubmodule},
$K\cdot [x_{i},x_{n+1}]\otimes_{K} (V\otimes_{K} K_{\mu})^{x_{i}}$
are $\mathbb{L}_{n}\oplus_{\theta_{g}}\mathfrak{g}$-submodules of $\mathfrak{LM}_{\mu}^{\mathrm{df}}(V)$
for $i=1,\ldots,n$.
Therefore we can see that 
\[
	\phi_{\lambda,\mu}\left(
	\mathfrak{LM}_{\lambda}^{\mathrm{df}}\left(
		\sum_{i=1}^{n} K\cdot [x_{i},x_{n+1}]\otimes_{K} (V\otimes_{K} K_{\mu})^{x_{i}}
		\right)
	\right)
	\subset \sum_{i=1}^{n} K\cdot [x_{i},x_{n+1}]\otimes_{K} (V\otimes_{K} K_{\lambda+\mu})^{x_{i}}.
\]
Also we have
\[
	\phi_{\lambda,\mu}\left(
	\mathfrak{LM}_{\lambda}^{\mathrm{df}}\left(
		\mathfrak{LM}_{\mu}^{\mathrm{df}}(V)^{\mathbb{L}_{n}}
		\right)
	\right)
	=0.
\] 
Thus Lemmas \ref{lem:Linvariant} and \ref{lem:K} imply that
the following diagram is commutative:
\scriptsize
\[
\begin{tikzcd}[column sep=small]
	0 \arrow[r] & \mathfrak{LM}_{\lambda}^{\mathrm{df}}\left(
		\sum_{i=1}^{n} K\cdot [x_{i},x_{n+1}]\otimes_{K} (V\otimes_{K} K_{\mu})^{x_{i}}
		+\mathfrak{LM}_{\mu}^{\mathrm{df}}(V)^{\mathbb{L}_{n}}
		\right) \arrow[r] \arrow[d,"\phi_{\lambda,\mu}"] & \mathfrak{LM}_{\lambda}^{\mathrm{df}}(\mathfrak{LM}_{\mu}^{\mathrm{df}}(V)) \arrow[r] \arrow[d, "\phi_{\lambda,\mu}"] 
		& \mathfrak{LM}^{\mathrm{df}}_{\lambda}(\mathfrak{mc}_{\mu}(V)) \arrow[r]  & 0\\
	0 \arrow[r] & \sum_{i=1}^{n} K\cdot [x_{i},x_{n+1}]\otimes_{K} (V\otimes_{K} K_{\lambda+\mu})^{x_{i}}
	+\mathfrak{LM}_{\lambda+\mu}^{\mathrm{df}}(V)^{\mathbb{L}_{n}} \arrow[r] & \mathfrak{LM}_{\lambda+\mu}^{\mathrm{df}}(V) \arrow[r] & \mathfrak{mc}_{\lambda+\mu}(V) \arrow[r] & 0	
\end{tikzcd}.
\]
\normalsize
Here horizontal arrows are exact since $\mathfrak{LM}_{\lambda}^{\mathrm{df}}$ is exact.
Thus we obtain the induced $\mathbb{L}_{n}\oplus_{\theta_{g}}\mathfrak{g}$-homomorphism
\[
	\widetilde{\phi}_{\lambda,\mu}\colon 
	\mathfrak{LM}_{\lambda}^{\mathrm{df}}(
	\mathfrak{mc}_{\mu}(V))
	\rightarrow 
	\mathfrak{mc}_{\lambda+\mu}(V).
\]
Let us check that $\widetilde{\phi}_{\lambda,\mu}$  factors through $\mathfrak{mc}_{\lambda}(\mathfrak{mc}_{\mu}(V))$,
i.e., we need to show  
\[
	\widetilde{\phi}_{\lambda,\mu}\left(
K\cdot [x_{i},x_{n+1}]\otimes_{K} (\mathfrak{mc}_{\mu}(V)\otimes_{K} K_{\lambda})^{x_{i}}
	\right)=0
\]
for $i=1,\ldots,n$, and
\[
	\widetilde{\phi}_{\lambda,\mu}\left(
		\mathfrak{LM}_{\lambda}^{\mathrm{df}}\left(\mathfrak{mc}_{\mu}(V)\right)^{\mathbb{L}_{n}}
	\right)=0.
\] 
Since $\phi_{\lambda,\mu}$ factors through the projection $\mathrm{pr}_{x_{n+1}}$, 
the first equations hold.
Moreover, since $\phi_{\lambda,\mu}$ is an $\mathbb{L}_{n}$-homomorphism,
$\mathbb{L}_{n}$-invariant subspace 
is sent to $\mathbb{L}_{n}$-invariant subspace by $\phi_{\lambda,\mu}$,
and $\mathbb{L}_{n}$-invariant subspace 
$\mathfrak{mc}_{\lambda+\mu}(V)^{\mathbb{L}_{n}}$ of $\mathfrak{mc}_{\lambda+\mu}(V)$ is zero by Lemma \ref{lem:Linvariant}. 
This shows the second equation as well.

Finally, under the conditions $(*),\,(**)$, Dettweiler and Reiter showed that the above induced map
$\mathfrak{mc}_{\lambda}(\mathfrak{mc}_{\mu}(V)) \to \mathfrak{mc}_{\lambda+\mu}(V)$
is bijective (see Theorem 3.5 
in \cite{DR00}, and also see Section 7.5 in \cite{Harbook}).
\end{proof}

\section{Holonomy Lie algebras of complements of hyperplane arrangements}\label{sec:holonomy}
We recall the notion of the holonomy Lie algebra of a complement of a hyperplane arrangement, and explain
the fact that the holonomy Lie algebra of 
a strictly linearly fibered arrangement is isomorphic to the semidirect sum of the holonomy Lie algebra of the base arrangement and the free Lie algebra generated by the fiber arrangement,
following \cite{Koh83} and \cite{CCX03}.
From this section, we fix $K$ to be the field of the complex numbers $K=\mathbb{C}$.
\subsection{Holonomy Lie algebras and associated graded Lie algebras}
Let us recall the definition of the holonomy Lie algebra introduced by Chen \cite{Ch73}.
\begin{df}[Holonomy Lie algebra]
	Let $X$ be a connected $CW$-complex with 2-skelton. 
	We denote the dual of the cup product map $H^{1}(X,\mathbb{C})\wedge H^{1}(X,\mathbb{C})\to H^{2}(X,\mathbb{C})$ by 
	$\delta\colon H_{2}(X,\mathbb{C})\to H_{1}(X,\mathbb{C})\wedge H_{1}(X,\mathbb{C})$.
	Then by regarding $H_{1}(X,\mathbb{C})\wedge H_{1}(X,\mathbb{C})$ as a subspace of $\mathbb{L}(H_{1}(X,\mathbb{C}))$, we define the \emph{holonomy Lie algebra} $\mathfrak{h}(X)$ of $X$ by
	\[
	\mathfrak{h}(X) := \mathbb{L}(H_{1}(X,\mathbb{C}))/\langle \delta(H_{2}(X,\mathbb{C})) \rangle,
	\]
	where $\langle \delta(H_{2}(X,\mathbb{C})) \rangle$ is the ideal of $\mathbb{L}(H_{1}(X,\mathbb{C}))$ generated by $\delta(H_{2}(X,\mathbb{C}))$.
\end{df}

We now recall Kohno's description of the holonomy Lie algebra of the complement of a hyperplane arrangement. 
Let $\mathcal{A}=\{H_{1},\ldots,H_{n}\}$ be a finite affine  complex hyperplane arrangement in $\C^{l}$,
i.e., a finite set of affine complex hyperplanes in $\C^{l}$.
We denote the complement of $\mathcal{A}$ by 
$M(\mathcal{A}):=\C^{l}\setminus\bigcup_{i=1}^{n}H_{i}$.
\begin{thm}[Kohno \cite{Koh83}]\label{kohnohollie}
	Let $\mathcal{A}=\{H_{1},\ldots,H_{n}\}$ be a finite affine complex hyperplane arrangement in $\C^{l}$ and $M(\mathcal{A})$ be the complement of $\mathcal{A}$.
	By regarding the hyperplanes $H_{1},\ldots,H_{n}$ as indeterminate variables,  we define
	the free Lie algebra $\mathbb{L}(H_{1},\ldots,H_{n})$ generated by $H_{1},\ldots,H_{n}$ and 
	its ideal $\mathfrak{a}$ generated by all the following elements:
	\[
		[H_{i_{j}},H_{i_{1}}+\cdots+H_{i_{m}}],\quad 1\le j<m,
	\]
	where
	$\{H_{i_{1}},\ldots,H_{i_{m}}\}$ is a maximal 
	family of hyperplanes in $\mathcal{A}$ such that
	$\mathrm{codim}(H_{i_{1}}\cap\cdots\cap H_{i_{m}})=2$.
	Then there exists an isomorphism of Lie algebras
	\[
	\mathfrak{h}(M(\mathcal{A})) \simeq \mathbb{L}(H_{1},\ldots,H_{n})/\mathfrak{a}.
	\]
\end{thm}

Let us next introduce a graded Lie algebra associated with the lower central 
series of a group $G$.
For subgroups $A, B$ of $G$, $(A,B)$ denote the subgroup of $G$
generated by the commutators $(a,b):=aba^{-1}b^{-1}$ for $a\in A, b\in B$.
Then we can consider the lower central series 
$\{\Gamma_{n}G\}_{n\in \mathbb{Z}_{\ge 1}}$ defined inductively by 
$\Gamma_{n+1}G:=(G,\Gamma_{n}G)$ and $\Gamma_{1}G:=G$.
Since quotient spaces $\Gamma_{n}G/\Gamma_{n+1}G$ are abelian 
groups, i.e., $\mathbb{Z}$-modules, the associated graded group 
\[
	\mathrm{gr}(G):=\bigoplus_{n\in \mathbb{Z}_{\ge 1}}\Gamma_{n}G/\Gamma_{n+1}G
\]
becomes $\mathbb{Z}$-module as well.
Moreover for $x\in \Gamma_{k}G, y\in \Gamma_{l}G$, 
we have $(x+\Gamma_{k+1}G, y+\Gamma_{l+1}G)
=(x,y)+\Gamma_{k+l+1}G$. Therefore the commutator $(x,y)$
induces the well-defined (graded) Lie algebra structure on $\mathrm{gr}(G)$.
This Lie algebra $\mathrm{gr}(G)$ is called the \emph{associated graded Lie algebra} of $G$.
We also consider the Lie algebra 
$\mathrm{gr}(G,\mathbb{C}):=\mathrm{gr}(G)\otimes_{\mathbb{Z}}\mathbb{C}$
over $\mathbb{C}$ as well.

For example, let us consider the free group $F_{n}$
with the generators $x_{1},\ldots,x_{n}$. Then it is well known that 
Magnus \cite{Mag37} showed that the associated graded Lie algebra $\mathrm{gr}(F_{n},\mathbb{C})$
is isomorphic to the free Lie algebra $\mathbb{L}(x_{1},\ldots,x_{n})$. 

Then, owing to the works of Sullivan \cite{Sul77} and Chen \cite{Ch73}, 
Kohno showed that holonomy Lie algebras of complements of hyperplane arrangements are isomorphic to the associated graded Lie algebras of their fundamental groups as follows. 
\begin{thm}[Kohno \cite{Koh83}]\label{kohnochen}
Let $\mathcal{A}$ be an arrangement of affine hyperplanes 
in $\mathbb{C}^{l}$. Then there exists an isomorphism of Lie algebras 
\[
	\mathrm{gr}(\pi_{1}(M(\mathcal{A})),\mathbb{C})\cong \mathfrak{h}(M(\mathcal{A})).
\]
\end{thm}

This theorem shows us the explicit structure of 
$\mathrm{gr}(P_{n},\mathbb{C})$, the associated graded 
Lie algebra of the pure braid group $P_{n}$ of $n$-strands as follows.
Let $\mathcal{A}$ be the \emph{braid arrangement}
in $\mathbb{C}^{n}$. Namely, $\mathcal{A}$
consists of hyperplanes 
$H_{i,j}$ defined as the 
locus of the linear forms $z_{i}-z_{j}$
for $1\le i<j\le n$. 
Then the fundamental group of the complement 
\[
	M(\mathcal{A})=\{(z_{1},\ldots,z_{n})\in \mathbb{C}^{n}\mid z_{i}\neq z_{j}, i\neq j\},
\]
is known to be isomorphic to the 
pure braid group $P_{n}$.
Therefore, under the isomorphism in Theorem \ref{kohnochen},
we see that the associated graded Lie algebra 
$\mathrm{gr}(P_{n},\mathbb{C})$ is generated by 
elements $A_{i,j}$, $1\le i\neq j\le n$ with the following relation, 
which is nothing but the infinitesimal braid relation:
\begin{align*}
	A_{i,j}-A_{j,i}&=0, \\
	[A_{i,k},A_{i,j}+A_{j,k}]&=0,\\
	[A_{i,j},A_{k,l}]&=0, 
\end{align*}
for mutually distinct $i,j,k,l$.
Namely, we obtain the isomorphism
\[
	\mathrm{gr}(P_{n},\mathbb{C})\cong \mathfrak{P}_{n}.
\]

\subsection{Semidirect sum decomposition of holonomy Lie algebras of strictly linearly fibered arrangements}\label{sec:CCX}
Let $\mathcal{A}$ be an 
affine hyperplane arrangement in $\C^{l}$ and
$Y$ an affine line in $\C^{l}$ 
passing through the origin, i.e.,
$Y$ is a one-dimensional linear subspace of $\C^{l}$.
Let us consider the subarrangement of $\mathcal{A}$ consisting of hyperplanes parallel to $Y$,
\[
	\mathcal{A}_{\parallel Y}:=\{H\in \mathcal{A}\mid H\text{ is parallel to }Y\}
\] 
and its complement in $\mathcal{A}$,
\[
	\mathcal{A}_{\parallel Y}^{c}:=\mathcal{A}\setminus \mathcal{A}_{\parallel Y}.
\]

The 
projection $\pi_{Y}\colon \C^{l}\to \C^{l}/Y$
defines the arrangement of affine hyperplanes 
\[
	p\mathcal{A}_{\parallel Y}:=\{\pi_{Y}(H)\mid H\in \mathcal{A}_{\parallel Y}\}
\] 
in the base space $\C^{l}/Y$,
and then by restricting $\pi_{Y}$ to $M(\mathcal{A})$, we also obtain 
the projection map 
\[
	\pi_{Y}\colon M(\mathcal{A})\to M(p\mathcal{A}_{\parallel Y}).
\]
We say that 
the line $Y$ is \emph{good} if
the  projection 
$\pi_{Y}\colon M(\mathcal{A})\to M(p\mathcal{A}_{\parallel Y})$
is a fiber bundle with the fiber $\C\setminus\{n\text{-points}\}$ for $n=|\mathcal{A}_{\parallel Y}^{c}|$,
see Theorem 2.9 in \cite{Tera1}.
\begin{df}[Strictly linearly fibered arrangement]
	An affine hyperplane arrangement $\mathcal{A}$ in $\C^{l}$ is called \emph{strictly linearly fibered}
	 if there exists a good line $Y$.
\end{df}

The braid arrangement 
$\{H_{i,j}\mid 1\le i<j\le n\}$ 
defined previously
is a typical example of strictly linearly fibered arrangements,
where every $z_{i}$-axis is a good line.
As we have seen, the associated graded Lie algebra $\mathrm{gr}(P_{n+1},\mathbb{C})$ of the pure braid group $P_{n+1}$ has the semidirect sum decomposition
\[	\mathrm{gr}(P_{n+1},\mathbb{C})
	=\mathbb{L}(A_{1,n+1},\ldots,A_{n,n+1})\oplus \mathrm{gr}(P_{n},\mathbb{C}),
\]
with respect to the infinitesimal braid action $\theta\colon \mathrm{gr}(P_{n},\mathbb{C})\to \mathrm{Der}(\mathbb{L}(A_{1,n+1},\ldots,A_{n,n+1}))$.

As a generalization of this fact,
Cohen, Cohen and Xicot\'encatl \cite{CCX03} showed 
that for any strictly linearly fibered arrangement $\mathcal{A}$,
the associated graded Lie algebra $\mathrm{gr}(\pi_{1}(M(\mathcal{A})),\mathbb{C})$
has the following semidirect sum decomposition.
\begin{thm}[Cohen, Cohen and Xicot\'encatl, Theorem 4.5 in \cite{CCX03}]\label{thm:CCX}
Let $\mathcal{A}$ be a strictly linearly fibered arrangement in $\C^{l}$ and $Y$ a good line for $\mathcal{A}$.
Then the associated graded Lie algebra $\mathrm{gr}(\pi_{1}(M(\mathcal{A})),\mathbb{C})$ of the fundamental group $\pi_{1}(M(\mathcal{A}))$ of the complement $M(\mathcal{A})$ has the semidirect sum decomposition
\[
	\mathrm{gr}(\pi_{1}(M(\mathcal{A})),\mathbb{C})=
	\mathbb{L}(\mathcal{A}_{\parallel Y}^{c})\oplus \mathrm{gr}(\pi_{1}(M(p\mathcal{A}_{\parallel Y})),\mathbb{C}),
\]
with respect to the Lie algebra homomorphism $\Theta \colon \mathrm{gr}(\pi_{1}(M(p\mathcal{A}_{\parallel Y})),\mathbb{C})\to \mathrm{Der}(\mathbb{L}(\mathcal{A}_{\parallel Y}^{c}))$.
Here $\Theta$ factors 
through the infinitesimal braid action 
$\theta\colon \mathrm{gr}(P_{n},\mathbb{C})\to \mathrm{Der}(\mathbb{L}(\mathcal{A}_{\parallel Y}^{c}))$
defined in Definition \ref{Infbraidaction} for $n=|\mathcal{A}_{\parallel Y}^{c}|$.
Namely, there exists a Lie algebra homomorphism $g\colon \mathrm{gr}(\pi_{1}(M(p\mathcal{A}_{\parallel Y})),\mathbb{C})\to \mathrm{gr}(P_{n},\mathbb{C})$ such that $\Theta=\theta\circ g$.
\end{thm}
This theorem implies that the holonomy Lie algebra $\mathfrak{h}(M(\mathcal{A}))$ of a strictly linearly fibered arrangement $\mathcal{A}$ has the semidirect sum decomposition
\[	
	\mathfrak{h}(M(\mathcal{A}))=
	\mathbb{L}(\mathcal{A}_{\parallel Y}^{c})\oplus \mathfrak{h}(M(p\mathcal{A}_{\parallel Y})),
\]
and our middle convolution functor $\mathfrak{mc}_{\lambda}$ 
works on the category of $U(\mathfrak{h}(M(\mathcal{A})))$-modules.
\subsection{$Y$-closure of a hyperplane arrangement}\label{sec:Yclosure}
In \cite{Tera1}, Terao gave a characterization of strictly linearly fibered arrangements in terms of intersection lattices of a hyperplane arrangement.
We now recall Terao's characterization of good lines.
Let $L(\mathcal{A})$ be the intersection poset of $\mathcal{A}$, i.e., the set of all nonempty intersections of hyperplanes in $\mathcal{A}$ ordered by reverse inclusion.
Let $L_{k}(\mathcal{A})$ be the set of all elements in $L(\mathcal{A})$ of codimension $k$.
\begin{thm}[Terao, Theorem 2.9 in \cite{Tera1}. See also  Theorem 1.8 in Adachi-Hiroe \cite{AdaHir}]\label{thm:goodline}
Let $\mathcal{A}$ be an affine hyperplane arrangement in $\C^{l}$ and $Y$ a line in $\C^{l}$ passing through the origin.
Then the line $Y$ is good if and only if 
\[
	X+Y\in L(\mathcal{A})\text{ for all }X\in L_{2}(\mathcal{A}).
\]
\end{thm}
By following Oshima \cite{Oshi3}, we call arramgements satisfying the above condition as follows.
\begin{df}[Oshima \cite{Oshi3}, $Y$-closed arrangement]
	Let $\mathcal{A}$ be an affine hyperplane arrangement in $\C^{l}$ and $Y$ a line in $\C^{l}$ passing through the origin.
	We say that $\mathcal{A}$ is \emph{$Y$-closed} if $X+Y\in L(\mathcal{A})$ for all $X\in L_{2}(\mathcal{A})$.
\end{df}

We can easily see that for any affine arrangement $\mathcal{A}$ in $\C^{l}$ and any line $Y$ in $\C^{l}$ passing through the origin, there exists a unique minimal $Y$-closed arrangement.
\begin{df}[$Y$-closure of a hyperplane arrangement]
	Let $\mathcal{A}$ be an affine hyperplane arrangement in $\C^{l}$ and $Y$ a line in $\C^{l}$ passing through the origin.
	We call the unique minimal $Y$-closed arrangement containing $\mathcal{A}$ the \emph{$Y$-closure} of $\mathcal{A}$, and denote this by $\overline{\mathcal{A}}^{Y}$.
\end{df}
This is constructed as follows:
\begin{equation*}
	\overline{\mathcal{A}}^{Y}=\mathcal{A}\cup \{X+Y\mid X\in L_{2}(\mathcal{A}),
	X\text{ is not parallel to }Y\}. 
\end{equation*}

The inclusion map $M(\overline{\mathcal{A}}^{Y})\hookrightarrow M(\mathcal{A})$ induces 
the Lie algebra homomorphism of holonomy Lie algebras $\mathfrak{h}(M(\overline{\mathcal{A}}^{Y}))\to \mathfrak{h}(M(\mathcal{A}))$,
which is explicitly described as follows:
\[
	\mathfrak{h}(M(\overline{\mathcal{A}}^{Y}))\to \mathfrak{h}(M(\mathcal{A}));\ \
	\overline{\mathcal{A}}^{Y}\ni H\mapsto \begin{cases}
	H & (H\in \mathcal{A})\\
	0 & (H\notin \mathcal{A})
\end{cases}.
\]

\section{Middle convolution for logarithmic Pfaffian systems on complements of hyperplane arrangements}\label{sec:mcPf}
As a generalization of the Dettweiler-Reiter additive middle convolution,
Haraoka \cite{Har1} defined the middle convolution for logarithmic Pfaffian systems on complements of hyperplane arrangements.
In this section, we show the equivalence
between the category of logarithmic Pfaffian systems on the complement of a hyperplane arrangement 
and the category of modules over the holonomy Lie algebra of the complement of a hyperplane arrangement.
Then we moreover show that under this equivalence of cathegories,
the middle convolution functor $\mathfrak{mc}_{\lambda}$ for modules over the holonomy Lie algebra
recovers Haraoka's middle convolution for logarithmic Pfaffian systems.
\subsection{Middle convolution for logarithmic Pfaffian systems}\label{sec:logPf}
Let us consider an affine hyperplane arrangement $\mathcal{A}$ in $\C^{l}$.
For an affine hyperplane $H$ in $\mathbb{C}^{l}$, let $f_{H}(x)$ denote a defining affine linear form of $H$.

We now introduce the category 
\[
	\mathrm{Pf}(\mathrm{log}(\mathcal{A}))
\]
of logarithmic Pfaffian systems with poles along $\mathcal{A}$. 
The objects of $\mathrm{Pf}(\mathrm{log}(\mathcal{A}))$ 
consist of 
pairs $(E,\nabla_{A})$
of a finite dimensional $\mathbb{C}$-vector space $E$ and a flat connection $\nabla_{A}$ on the trivial vector bundle $\mathbb{C}^{l}\times E$ over $\mathbb{C}^{l}$, 
where $\nabla_{A}=d-\Omega_{A}$
with the following coefficient $1$-form $\Omega_{A}$: 
\[
	\Omega_{A}=\sum_{H\in \mathcal{A}}A_{H}\frac{df_{H}}{f_{H}}\quad (A_{H}\in \mathrm{End}_{\mathbb{C}}(E))
\]
satisfying the integrability condition
\[	
	\Omega_{A}\wedge \Omega_{A}=0.
\]
For two such objects $(E_1, \nabla_{A_1})$ and $(E_2, \nabla_{A_2})$
a morphism 
is defined to be a $\mathbb{C}$-linear map $\phi\colon E_{1}\rightarrow E_{2}$
such that the following diagram commutes,
\[
\begin{tikzcd}
\mathcal{O}_{\mathbb{C}^{l}}\otimes E_{1}\arrow[r,"\nabla_{A_{1}}"]\arrow[d,"\mathrm{id}\otimes \phi"]&
\varOmega_{\mathbb{C}^{l}}^{1}(*\mathcal{A})\otimes E_{1}\arrow[d,"\mathrm{id}\otimes \phi"]\\
\mathcal{O}_{\mathbb{C}^{l}}\otimes E_{2}\arrow[r,"\nabla_{A_{2}}"]&
\varOmega_{\mathbb{C}^{l}}^{1}(*\mathcal{A})\otimes E_{2}
\end{tikzcd}.
\]
Here $\mathcal{O}_{\mathbb{C}^{l}}$ is the sheaf of holomorphic functions on $\mathbb{C}^{l}$, $\varOmega_{\mathbb{C}^{l}}^{1}(*\mathcal{A})$ is the sheaf of meromorphic $1$-forms on $\mathbb{C}^{l}$ with poles along $\mathcal{A}$.

Haraoka \cite{Har1} defined the middle convolution for logarithmic Pfaffian systems
as a generalization of the Dettweiler-Reiter additive middle convolution.
Let us recall the definition of the middle convolution for logarithmic Pfaffian systems.

Let us take an affine line $Y$ in $\mathbb{C}^{l}$ passing through the origin,
and then consider the vector space generated by the hyperplanes in $\mathcal{A}_{\parallel Y}^{c}$, i.e., 
\[
	\mathbb{C}\cdot \mathcal{A}_{\parallel Y}^{c}=\bigoplus_{H\in \mathcal{A}_{\parallel Y}^{c}}\mathbb{C}\cdot e_{H}.
\]
Take an object $(E,\nabla_{A})\in \mathrm{Pf}(\mathrm{log}(\mathcal{A}))$
with the coefficient $1$-form  
\[
	\Omega_{A}=
	\sum_{H\in \mathcal{A}}A_{H}\frac{df_{H}}{f_{H}}\quad (A_{H}\in \mathrm{End}_{\mathbb{C}}(E)),
\]
and also take a parameter $\lambda\in \mathbb{C}$.
Haraoka constructed a logarithmic Pfaffian system
\[	
	(E\otimes_{\mathbb{C}}\mathbb{C}\cdot \mathcal{A}_{\parallel Y}^{c},\nabla_{c_{\lambda}(A)})\in \mathrm{Pf}(\mathrm{log}(\overline{\mathcal{A}}^{Y}))	
\]
with the coefficient $1$-form
\[
	\Omega_{c_{\lambda}(A)}=
	\sum_{H\in \overline{\mathcal{A}}^{Y}}c_{\lambda}(A)_{H}\frac{df_{H}}{f_{H}}
\]
defined as follows.
Let $I_{H,H'}\in \mathrm{End}_{\mathbb{C}}(\mathbb{C}\cdot \mathcal{A}_{\parallel Y}^{c})$
be the $(H,H')$-matrix units for $H,H'\in \mathcal{A}_{\parallel Y}^{c}$,
i.e., the endomorphism defined by 
\[
I_{H,H'}e_{H''}=\delta_{H',H''}e_{H}
\]
for $H,H',H''
\in \mathcal{A}_{\parallel Y}^{c}$.
Then for $H\in \mathcal{A}_{\parallel Y}^{c}$, we set  
\[
	c_{\lambda}(A)_{H}:=\sum_{H'\in \mathcal{A}_{\parallel Y}^{c}}(A_{H'}+\lambda\delta_{H,H'}\mathrm{Id}_{E})\otimes I_{H,H'}.
\]
For $H\in\overline{\mathcal{A}}^{Y}\setminus\mathcal{A}$, set $A_H=0$.
For $X\in L_2(\overline{\mathcal{A}}^{Y})$, put
\[
    \mathcal{F}_X:=\{H'\in\mathcal{A}_{\parallel Y}^{c}\mid X\subset H'\}.
\]
For $H\in\overline{\mathcal{A}}^{Y}\setminus\mathcal{A}_{\parallel Y}^{c}$,
define
\[
    \mathscr{C}_H:=\{X\in L_2(\overline{\mathcal{A}}^{Y})\mid
                       X+Y=H,\ |\mathcal{F}_X|\geq2\}.
\]
Then we set
\[
\begin{aligned}
    c_{\lambda}(A)_H
       &:=A_H\otimes\mathrm{Id}_{\mathbb{C}\cdot\mathcal{A}_{\parallel Y}^{c}}\\
       &\quad+\sum_{X\in\mathscr{C}_H}\sum_{H'\in\mathcal{F}_X}
          A_{H'}\otimes
             \left(\sum_{H''\in\mathcal{F}_X}
                    (I_{H'',H''}-I_{H'',H'})\right).
\end{aligned}
\]
For fixed $H$, the sets $\mathcal{F}_X$ for $X\in\mathscr{C}_H$
are pairwise disjoint, since $H'\in\mathcal{F}_X$ implies $X=H\cap H'$.
The sum runs over all such collision classes, and singleton classes
contribute zero. The agreement of these residue formulas with Haraoka's
original definition is established in \cite[Section~3.2, Proposition~3.1]{AdaHir}.
The integrability condition for $\Omega_{c_{\lambda}(A)}$
follows from that for $\Omega_{A}$.
It is easily checked that the above construction is functorial.
\begin{rem}
We may suppose that $Y$ is the $z_{l}$-axis 
of $\mathbb{C}^{l}=\{(z_{1},\ldots,z_{l})\in \mathbb{C}^{l}\}$
by applying a suitable linear transformation to $\mathbb{C}^{l}$.
Then 
Haraoka constructed this  $d-\Omega_{c_{\lambda}(A)}$
as the flat connection which has 
the following function as its horizontal sections: 
\[
	\left(\int_{\Delta}U(z_{1},\ldots,z_{l-1},t)(t-z_{l})^{\lambda}\left(\frac{\partial}{\partial t}\log f_{H}(z_{1},\ldots,z_{l-1},t)\right)dt\right)_{H\in \mathcal{A}_{\parallel Y}^{c}}.
\]
This is  the convolution product 
of the horizontal sections $U(z_{1},\ldots,z_{l})$ 
of the connection $d-\Omega_{A}$ 
with the power function $t^{\lambda}$ along the $z_{l}$-axis.
Here $\Delta$ is a suitable integral contour in the $t$-plane.
See Proposition 2.1 in \cite{Har1} for details.
\end{rem}
\begin{df}[Haraoka \cite{Har1}]
For $\lambda\in \mathbb{C}$, we call the resulting functor
\[
	\begin{array}{cccc}
	c_{\lambda}\colon & \mathrm{Pf}(\mathrm{log}(\mathcal{A}) & \longrightarrow & \mathrm{Pf}(\mathrm{log}(\overline{\mathcal{A}}^{Y}))\\
& (E,\nabla_{A}) & \mapsto & (E\otimes_{\mathbb{C}}\mathbb{C}\cdot \mathcal{A}_{\parallel Y}^{c},\nabla_{c_{\lambda}(A)})
	\end{array}
\]
the \emph{convolution functor along the line $Y$ with parameter $\lambda$}.
\end{df}
As well as the Dettweiler-Reiter additive middle convolution, Haraoka moreover
considered the following sub-connections of $\nabla_{c_{\lambda}(A)}$:
\begin{align*}
	&\nabla_{c_{\lambda}(A)}^{K}\colon
	\mathcal{O}_{\mathbb{C}^{l}}\otimes K
	\longrightarrow
	\varOmega_{\mathbb{C}^{l}}^{1}(*\overline{\mathcal{A}}^{Y}))\otimes K,\\
	&\nabla_{c_{\lambda}(A)}^{L}\colon
	\mathcal{O}_{\mathbb{C}^{l}}\otimes L
	\longrightarrow
	\varOmega_{\mathbb{C}^{l}}^{1}(*\overline{\mathcal{A}}^{Y})\otimes L,
\end{align*}
with the subspaces $K$ and $L$ of $E\otimes_{\mathbb{C}}\mathbb{C}\cdot \mathcal{A}_{\parallel Y}^{c}$ defined by
\[
	K:=\bigoplus_{H\in \mathcal{A}_{\parallel Y}^{c}}\left(\mathrm{Ker\,}A_{H}\otimes \mathbb{C}e_{H}\right),\quad\quad
	L:=\bigcap_{H\in \mathcal{A}_{\parallel Y}^{c}}\mathrm{Ker\,}c_{\lambda}(A)_{H}.
\]
Then we further obtain the quotient-connection
\[
	\nabla_{\mathrm{mc}_{\lambda}(A)}\colon
	\mathcal{O}_{\mathbb{C}^{l}}\otimes (E\otimes_{\mathbb{C}}\mathbb{C}\cdot \mathcal{A}_{\parallel Y}^{c}/(K\oplus L))
	\longrightarrow
	\varOmega_{\mathbb{C}^{l}}^{1}(*\overline{\mathcal{A}}^{Y})
	\otimes (E\otimes_{\mathbb{C}}\mathbb{C}\cdot \mathcal{A}_{\parallel Y}^{c}/(K\oplus L)).
\]
\begin{df}[Haraoka \cite{Har1}]
We call the resulting functor
\[
	\mathrm{mc}_{\lambda}\colon \mathrm{Pf}(\mathrm{log}(\mathcal{A}))
	\longrightarrow 
	\mathrm{Pf}(\mathrm{log}(\overline{\mathcal{A}}^{Y}));\quad
	\nabla_{A}\mapsto \nabla_{\mathrm{mc}_{\lambda}(A)},
\]
the {\em middle convolution functor along the line $Y$ with parameter $\lambda\in \mathbb{C}$}.
\end{df}
As well as the Dettweiler-Reiter additive middle convolution,
this middle convolution functor also satisfies the composition law under the following assumptions:
\begin{equation}\label{eq:star1}
	\bigcap_{\substack{H'\in \mathcal{A}_{\parallel Y}^{c},\\H'\neq H}}\mathrm{Ker\,}A_{H'}\cap \mathrm{Ker\,} (A_{H}+\tau \mathrm{Id}_{E})=\{0\}
	\quad \text{ for any }H\in \mathcal{A}_{\parallel Y}^{c}\text{ and }\tau\in \mathbb{C},
\end{equation}
\begin{equation}\label{eq:star2}
	\sum_{\substack{H'\in \mathcal{A}_{\parallel Y}^{c},\\H'\neq H}}\mathrm{Im\,}A_{H} + \mathrm{Im\,}(A_{H}+\tau \mathrm{Id}_{E})=E
	\quad \text{ for any }H\in \mathcal{A}_{\parallel Y}^{c}\text{ and }\tau\in \mathbb{C}.
\end{equation}
\begin{thm}[Theorem 3.1 in \cite{Har1}]\label{thm:MCcompos}
Suppose $\nabla_{A}\in \mathrm{Pf}(\mathrm{log}(\mathcal{A}))$ satisfies the assumptions \eqref{eq:star1} and \eqref{eq:star2}.
Then
the following holds for $\lambda,\mu\in \mathbb{C}$,
\begin{align*}
	\mathrm{mc}_{\lambda}\circ \mathrm{mc}_{\mu}(\nabla_{A})
	&\cong \mathrm{mc}_{\lambda+\mu}(\nabla_{A}),\\
	\mathrm{mc}_{-\lambda}\circ \mathrm{mc}_{\lambda}(\nabla_{A})
	&\cong \nabla_{A}.
\end{align*}
\end{thm}
\subsection{Logarithmic Pfaffian systems and modules over holonomy Lie algebras}
Let us take $(E,\nabla_{A})\in \mathrm{Pf}(\mathrm{log}(\mathcal{A}))$
with the coefficient $1$-form
\[	
	\Omega_{A}=
	\sum_{H\in \mathcal{A}}A_{H}\frac{df_{H}}{f_{H}}.
\]
Then the integrability condition $\Omega_{A}\wedge \Omega_{A}=0$ has the following explicit description.
Let us write $\mathcal{A}=\{H_1,\ldots,H_n\}$ and 
set 
\[
	S:=\left\{
		\{H_{i_{1}},\ldots,H_{i_{m}}\}\subset \mathcal{A}\,\middle|\,
		\begin{array}{l}
			\{H_{i_{1}},\ldots,H_{i_{m}}\}\text{ is maximal such that }\\
			\mathrm{codim}(H_{i_{1}}\cap\cdots \cap H_{i_{m}})=2, \\
			i_{1}<i_{2}	\cdots <i_{m}
		\end{array}
	\right\}.
\]
\begin{prop}[Kohno \cite{Koh83}, Proposition 2.1]\label{prop:integrability}
	The integrability condition $\Omega_{A}\wedge \Omega_{A}=0$ is equivalent to the following system of equations:
	\[
	[A_{H_{i_{j}}},A_{H_{i_{1}}}+\cdots +A_{H_{i_{m}}}]=0\quad  \text{ for } j=1,\ldots,m-1\text{ and }\{H_{i_{1}},\ldots,H_{i_{m}}\}\in S.
	\]
\end{prop}
\begin{proof}
This fact follows essentially from 
Proposition 2.1 in \cite{Koh83},  however we give a direct proof of it
for the completeness of the paper.

Let us denote by $\omega_{i}$ the logarithmic $1$-form $\frac{df_{H_{i}}}{f_{H_{i}}}$ for $i=1,\ldots,n$.
Then Orlik and Solomon \cite{OrSol80} showed that the graded $\mathbb{C}$-algebra 
$A$ of holomorphic differential forms on $M(\mathcal{A})$ generated by $\omega_{1},\ldots,\omega_{n}$ is isomorphic to the Orlik-Solomon algebra of $\mathcal{A}$, 
which is the quotient of the exterior algebra $\bigwedge\mathbb{C}\text{-span}\{\omega_{1},\ldots,\omega_{n}\}$ 
by the ideal generated by the following elements:
\begin{equation}\label{eq:OSrel}
	\sum_{k=1}^{m}(-1)^{k-1}\wedge \omega_{i_{1}}\wedge \cdots \wedge \widehat{\omega_{i_{k}}}\wedge \cdots \wedge \omega_{i_{m}}
\end{equation}
for all subsets $\{H_{i_{1}},\ldots,H_{i_{m}}\}\subset \mathcal{A}$ such that 
\(
	\mathrm{codim}(H_{i_{1}}\cap \cdots \cap H_{i_{m}})<m.
\)
Here $\widehat{\omega_{i_{k}}}$ means that $\omega_{i_{k}}$ is omitted.
This implies that the dgree $2$ part $A^{2}$ of $A$ has the following basis:
\[
	\bigsqcup_{\{H_{i_{1}},\ldots,H_{i_{m}}\}\in S}\left\{\omega_{i_{j}}\wedge \omega_{i_{m}}\,\middle|\,
		1\le j<m
	\right\},
\]
see Proposition 2.1 in \cite{Koh83}.

On the other hand, the relation $\eqref{eq:OSrel}$ leads to the following expansion of $\Omega_{A}\wedge \Omega_{A}$:
\[
		\Omega_{A}\wedge \Omega_{A}=
		\sum_{\{H_{i_{1}},\ldots,H_{i_{m}}\}\in S}\sum_{j=1}^{m-1}[A_{H_{i_{j}}},A_{H_{i_{1}}}+\cdots +A_{H_{i_{m}}}]\,\omega_{i_{j}}\wedge \omega_{i_{m}}.
\]
Then since $\{\omega_{i_{j}}\wedge \omega_{i_{m}}\}$ forms a basis of $A^{2}$, 
we have $\Omega_{A}\wedge \Omega_{A}=0$ if and only if 
\[
[A_{H_{i_{j}}},A_{H_{i_{1}}}+\cdots +A_{H_{i_{m}}}]=0
\] for all $j=1,\ldots,m-1$ and $\{H_{i_{1}},\ldots,H_{i_{m}}\}\in S$.
\end{proof}

This proposition leads to the following equivalence of categories.
\begin{prop}\label{prop:equivPfmod}
	There exists an equivalence of categories
	\[
		\Xi_{\mathcal{A}}\colon \mathrm{Pf}(\mathrm{log}(\mathcal{A}))\xrightarrow{\sim} U(\mathfrak{h}(M(\mathcal{A})))\text{-}\mathbf{mod}.
	\]
\end{prop}
\begin{proof}
For $(E,\nabla_{A})\in \mathrm{Pf}(\mathrm{log}(\mathcal{A}))$
with the coefficient $1$-form
\[	\Omega_{A}=
	\sum_{H\in \mathcal{A}}A_{H}\frac{df_{H}}{f_{H}}
\]
satisfying the integrability condition $\Omega_{A}\wedge \Omega_{A}=0$, we can equip the vector space $E$ with the structure of a $\mathfrak{h}(M(\mathcal{A}))$-module by setting 
\[
	H\cdot v:=A_{H}v\quad \text{ for }H\in \mathcal{A}\text{ and }v\in E.
\]
This is well-defined since the integrability condition $\Omega_{A}\wedge \Omega_{A}=0$ is equivalent to the defining relation of $\mathfrak{h}(M(\mathcal{A}))$
by the previous proposition.
This construction gives a functor from $\mathrm{Pf}(\mathrm{log}(\mathcal{A}))$ to $U(\mathfrak{h}(M(\mathcal{A})))\text{-}\mathbf{mod}$.
Conversely, for a $\mathfrak{h}(M(\mathcal{A}))$-module $E$, let us consider the the correpsonding reprsentation $\rho \colon \mathfrak{h}(M(\mathcal{A})) \to \mathrm{End}_{\mathbb{C}}(E)$.
Then we can construct an object $(E,\nabla_{A})\in \mathrm{Pf}(\mathrm{log}(\mathcal{A}))$ by setting $A_{H}:=\rho(H)$ for 
$H\in \mathcal{A}$. The above proposition implies that the integrability condition $\Omega_{A}\wedge \Omega_{A}=0$ is satisfied.
This gives a functor from $U(\mathfrak{h}(M(\mathcal{A})))\text{-}\mathbf{mod}$ to $\mathrm{Pf}(\mathrm{log}(\mathcal{A}))$.
It is easily checked that these two functors are quasi-inverse to each other.
\end{proof}

Let $Y$ be a line in $\C^{l}$ passing through the origin and $\overline{\mathcal{A}}^{Y}$ the $Y$-closure of $\mathcal{A}$.
Then the inclusion map $\iota_{Y}\colon M(\overline{\mathcal{A}}^{Y})\hookrightarrow M(\mathcal{A})$ induces the embedding functor
\[
	\iota_{Y}^{-1}\colon \mathrm{Pf}(\mathrm{log}(\mathcal{A}))\to \mathrm{Pf}(\mathrm{log}(\overline{\mathcal{A}}^{Y}));\quad
	(E,\nabla_{A})\mapsto (E,\nabla_{A}).
\]
Then we can check that the embedding functor $\iota_{Y}^{-1}$ corresponds to the
restriction functor $\mathrm{Res}_{\mathfrak{h}(M(\overline{\mathcal{A}}^{Y}))}^{\mathfrak{h}(M(\mathcal{A}))}$
under the above equivalence of categories, i.e., the following diagram commutes:
\begin{equation}\label{eq:commPfmod}
\begin{tikzcd}[column sep=large]
\mathrm{Pf}(\mathrm{log}(\mathcal{A}))\arrow[r,"\iota_{Y}^{-1}"]\arrow[d,"\simeq"',"\Xi_{\mathcal{A}}"]&
\mathrm{Pf}(\mathrm{log}(\overline{\mathcal{A}}^{Y}))\arrow[d,"\simeq"', "\Xi_{\overline{\mathcal{A}}^{Y}}"]\\
U(\mathfrak{h}(M(\mathcal{A})))\text{-}\mathrm{mod}\arrow[r,"\mathrm{Res}_{\mathfrak{h}(M(\overline{\mathcal{A}}^{Y}))}^{\mathfrak{h}(M(\mathcal{A}))}"]&
U(\mathfrak{h}(M(\overline{\mathcal{A}}^{Y})))\text{-}\mathrm{mod}
\end{tikzcd}.
\end{equation}
\subsection{Middle convolution functors $\mathrm{mc}_{\lambda}$ and $\mathfrak{mc}_{\lambda}$}
This section is devoted to the proof of the following theorem, which states that 
the middle convolution functor $\mathfrak{mc}_{\lambda}$ for modules over the holonomy Lie algebra recovers the middle convolution functor $\mathrm{mc}_{\lambda}$ for logarithmic Pfaffian systems.
\begin{thm}\label{thm:commPfmod}
Under the equivalence of categories $\Xi_{\mathcal{A}}\colon \mathrm{Pf}(\mathrm{log}(\mathcal{A}))\xrightarrow{\sim} U(\mathfrak{h}(M(\mathcal{A})))\text{-}\mathbf{mod}$
in Proposition \ref{prop:equivPfmod}, we have the natural equivalence of functors
\[
	\Xi_{\overline{\mathcal{A}}^{Y}}\circ \mathrm{mc}_{\lambda}\simeq \mathfrak{mc}_{\lambda}\circ \mathrm{Res}_{\mathfrak{h}(M(\overline{\mathcal{A}}^{Y}))}^{\mathfrak{h}(M(\mathcal{A}))}\circ \Xi_{\mathcal{A}}.
\]
Namely, we have the following commutative diagram
up to natural isomorphisms of functors:
\[\begin{tikzcd}[column sep=4cm]
\mathrm{Pf}(\mathrm{log}(\mathcal{A}))\arrow[r,"\mathrm{mc}_{\lambda}"]\arrow[d,"\simeq"',"\Xi_{\mathcal{A}}"]&
\mathrm{Pf}(\mathrm{log}(\overline{\mathcal{A}}^{Y}))\arrow[d,"\simeq"', "\Xi_{\overline{\mathcal{A}}^{Y}}"]\\
U(\mathfrak{h}(M(\mathcal{A})))\text{-}\mathrm{mod}\arrow[r,"\mathfrak{mc}_{\lambda}\circ \mathrm{Res}_{\mathfrak{h}(M(\overline{\mathcal{A}}^{Y}))}^{\mathfrak{h}(M(\mathcal{A}))}"]&
U(\mathfrak{h}(M(\overline{\mathcal{A}}^{Y})))\text{-}\mathrm{mod}
\end{tikzcd}.
\]
\end{thm}
First we note that the following diagram commutes:
\[
	\begin{tikzcd}
	\mathrm{Pf}(\mathrm{log}(\mathcal{A}))\arrow[d,"\iota_{Y}^{-1}"]\arrow[r,"\mathrm{mc}_{\lambda}"]&\mathrm{Pf}(\mathrm{log}(\overline{\mathcal{A}}^{Y}))\\
	\mathrm{Pf}(\mathrm{log}(\overline{\mathcal{A}}^{Y}))\arrow[ur,"\mathrm{mc}_{\lambda}"]&\\
	\end{tikzcd},
\]
since the pullback functor $\iota_{Y}^{-1}$ dones not change the coefficient $1$-form of the connection.
Thus combining this diagram with the diagram $\eqref{eq:commPfmod}$,
we may assume that $\mathcal{A}$ is $Y$-closed, i.e., $\mathcal{A}=\overline{\mathcal{A}}^{Y}$
and it suffices to show that the following diagram commutes:
\[\begin{tikzcd}
\mathrm{Pf}(\mathrm{log}(\mathcal{A}))\arrow[r,"\mathrm{mc}_{\lambda}"]\arrow[d,"\simeq"',"\Xi_{\mathcal{A}}"]&
\mathrm{Pf}(\mathrm{log}(\mathcal{A}))\arrow[d,"\simeq"', "\Xi_{\mathcal{A}}"]\\
U(\mathfrak{h}(M(\mathcal{A})))\text{-}\mathrm{mod}\arrow[r,"\mathfrak{mc}_{\lambda}"]&
U(\mathfrak{h}(M(\mathcal{A})))\text{-}\mathrm{mod}
\end{tikzcd}.
\]
\begin{prop}
	Let $\mathcal{A}$ be a $Y$-closed affine hyperplane arrangement in $\C^{l}$.
	Then we have the following commutative diagram up to natural isomorphisms of functors:
	\[
		\begin{tikzcd}
		\mathrm{Pf}(\mathrm{log}(\mathcal{A}))\arrow[r,"c_{\lambda}"]\arrow[d,"\simeq"',"\Xi_{\mathcal{A}}"]&
		\mathrm{Pf}(\mathrm{log}(\mathcal{A}))\arrow[d,"\simeq"', "\Xi_{\mathcal{A}}"]\\
		U(\mathfrak{h}(M(\mathcal{A})))\text{-}\mathrm{mod}\arrow[r,"\mathfrak{LM}^{\mathrm{df}}_{\lambda}"]&
		U(\mathfrak{h}(M(\mathcal{A})))\text{-}\mathrm{mod}
		\end{tikzcd}.
	\]
\end{prop}
\begin{proof}
Let us take a $U(\mathfrak{h}(M(\mathcal{A})))$-module $E$ and 
the corresponding object in $\mathrm{Pf}(\mathrm{log}(\mathcal{A}))$ denoted by $(E, \nabla_A)$
with the coefficent $1$-form
$\Omega_{A}=\sum_{H\in \mathcal{A}}A_{H}\frac{d f_{H}}{f_{H}}$.

By following Theorem \ref{thm:CCX}, we have the semidirect sum decomposition of $\mathfrak{h}(M(\mathcal{A}))$,
\[
	\mathfrak{h}(M(\mathcal{A}))=
	\mathbb{L}(\mathcal{A}_{\parallel Y}^{c})\oplus \mathfrak{h}(M(p\mathcal{A}_{\parallel Y})).
\]
Let us write $\mathcal{A}_{\parallel Y}^{c}=\{H_{1},\ldots,H_{n}\}$
and identify $\mathbb{L}(\mathcal{A}_{\parallel Y}^{c})$ with the free Lie algebra generated by $x_{1},\ldots,x_{n}$, where $x_{i}$ corresponds to $H_{i}$ for $i=1,\ldots,n$.
Then we can identify
\[
	\mathfrak{LM}^{\mathrm{df}}_{\lambda}(E)=\bigoplus_{i=1}^{n}\mathbb{C}\cdot [x_{i},x_{n+1}]\otimes_{\mathbb{C}}(E\otimes_{\mathbb{C}}\mathbb{C}_{\lambda}).
\]
As we saw in Proposition \ref{prop:compconv}, we have 
\[
	x_{i}\cdot ([x_{j},x_{n+1}]\otimes v)=[x_{i},x_{n+1}]\otimes (x_{j}+\delta_{i,j}x_{n+1})v\\
\]
for $i,j=1,\ldots,n$ and $v\in E\otimes_{\mathbb{C}}\mathbb{C}_{\lambda}$.
Therefore, under the identifications
$\bigoplus_{i=1}^{n}\mathbb{C}\cdot [x_{i},x_{n+1}]=\oplus_{i=1}^{n}\mathbb{C}\cdot e_{H_{i}}=\mathbb{C}\cdot \mathcal{A}_{\parallel Y}^{c}$
and $E\otimes_{\mathbb{C}}\mathbb{C}_{\lambda}=E$,
the above equation shows that the following diagram commutes:
\[
	\begin{tikzcd}
	\mathfrak{LM}^{\mathrm{df}}_{\lambda}(E)\arrow[r,"\simeq"]\arrow[d,"H\cdot"]&
	E\otimes_{\mathbb{C}}\mathbb{C}\cdot \mathcal{A}_{\parallel Y}^{c}\arrow[d,"c_{\lambda}(A)_{H}\cdot"]\\
	\mathfrak{LM}^{\mathrm{df}}_{\lambda}(E)\arrow[r,"\cong"]&
	E\otimes_{\mathbb{C}}\mathbb{C}\cdot \mathcal{A}_{\parallel Y}^{c}
	\end{tikzcd}
\]
for $H\in \mathcal{A}_{\parallel Y}^{c}$.
Here the left vertical arrows are the actions of $H$ and $c_{\lambda}(A)_{H}$ 
on $\mathfrak{LM}^{\mathrm{df}}_{\lambda}(E)$ and 
 $E\otimes_{\mathbb{C}}\mathbb{C}\cdot \mathcal{A}_{\parallel Y}^{c}$ respectively.

Next let us take $H\in\mathcal{A}_{\parallel Y}$ and
$[x_{i_0},x_{n+1}]\otimes v\in\mathfrak{LM}^{\mathrm{df}}_{\lambda}(E)$.
Put $X_0:=H\cap H_{i_0}$ and write
\[
    \mathcal{F}_{X_0}=\{H_{i_0},H_{i_1},\ldots,H_{i_k}\},
    \qquad k\geq0.
\]
Since $H_{i_0}$ is transverse to $Y$, the flat $X_0$ has codimension
two and $X_0+Y=H$. If a hyperplane $H'\in\mathcal{A}_{\parallel Y}$
contains $X_0$, then it contains $X_0+Y=H$, hence $H'=H$.
Thus $\{H,H_{i_0},\ldots,H_{i_k}\}$ is the full family of hyperplanes
in $\mathcal{A}$ containing $X_0$. The holonomy relation gives
\[
    [H,x_{i_0}]=[x_{i_0},x_{i_1}+\cdots+x_{i_k}],
\]
where the right-hand side is zero if $k=0$.

The base action fixes $x_{n+1}$. Hence, by the Jacobi identity,
\begin{align*}
    &H\cdot([x_{i_0},x_{n+1}]\otimes v)\\
    &=[[H,x_{i_0}],x_{n+1}]\otimes v
          +[x_{i_0},x_{n+1}]\otimes H\cdot v\\
    &=\sum_{j=1}^{k}
        \left([[x_{i_0},x_{n+1}],x_{i_j}]
              -[[x_{i_j},x_{n+1}],x_{i_0}]\right)\otimes v
          +[x_{i_0},x_{n+1}]\otimes H\cdot v\\
    &=\sum_{j=1}^{k}
        \left([x_{i_0},x_{n+1}]\otimes H_{i_j}\cdot v
              -[x_{i_j},x_{n+1}]\otimes H_{i_0}\cdot v\right)
          +[x_{i_0},x_{n+1}]\otimes H\cdot v.
\end{align*}
In the last equality, we used the balanced tensor product over
$U(\mathbb{L}_{n+1})$ and the fact that every $[x_i,x_{n+1}]$
annihilates $E\otimes_{\mathbb{C}}\mathbb{C}_{\lambda}$.

In the definition of $c_{\lambda}(A)_H$, only the collision
class containing $H_{i_0}$ can contribute to its action on
$v\otimes e_{H_{i_0}}$. That class is $\mathcal{F}_{X_0}$ if
$k\geq1$; for $k=0$ there is no collision contribution. Consequently,
\[
    c_{\lambda}(A)_H(v\otimes e_{H_{i_0}})
      =\left(A_H+\sum_{j=1}^{k}A_{H_{i_j}}\right)v\otimes e_{H_{i_0}}
         -\sum_{j=1}^{k}A_{H_{i_0}}v\otimes e_{H_{i_j}},
\]
which agrees with the above action under our identification.
This shows that the following diagram commutes:
\[
	\begin{tikzcd}
	\mathfrak{LM}^{\mathrm{df}}_{\lambda}(E)\arrow[r,"\cong"]\arrow[d,"H\cdot"]&
	E\otimes_{\mathbb{C}}\mathbb{C}\cdot \mathcal{A}_{\parallel Y}^{c}\arrow[d,"c_{\lambda}(A)_{H}\cdot"]\\
	\mathfrak{LM}^{\mathrm{df}}_{\lambda}(E)\arrow[r,"\cong"]&
	E\otimes_{\mathbb{C}}\mathbb{C}\cdot \mathcal{A}_{\parallel Y}^{c}
	\end{tikzcd}
\]
even for $H\in \mathcal{A}_{\parallel Y}$.
This completes the proof of the proposition.
\end{proof}
Now we are ready to show the commutativity of the diagram in the theorem.
\begin{proof}[Proof of Theorem \ref{thm:commPfmod}]
Let us take a $U(\mathfrak{h}(M(\mathcal{A})))$-module $E$ and 
the corresponding object in $\mathrm{Pf}(\mathrm{log}(\mathcal{A}))$ by $(E, \nabla_A)$
with the coefficent $1$-form
$\Omega_{A}=\sum_{H\in \mathcal{A}}A_{H}\frac{d f_{H}}{f_{H}}$.
Then under the identifications as in the proof of the previous proposition, 
$\mathfrak{LM}_{\lambda}^{\mathrm{df}}(E)=\bigoplus_{i=1}^{n}\mathbb{C}\cdot [x_{i},x_{n+1}]\otimes_{\mathbb{C}}(E\otimes_{\mathbb{C}}\mathbb{C}_{\lambda}),
	\bigoplus_{i=1}^{n}\mathbb{C}\cdot [x_{i},x_{n+1}]=\mathbb{C}\cdot \mathcal{A}_{\parallel Y}^{c},
	E\otimes_{\mathbb{C}}\mathbb{C}_{\lambda}=E,
$
the submodules 
\[
	\bigoplus_{i=1}^{n}\mathbb{C}\cdot [x_{i},x_{n+1}]\otimes (E\otimes_{\mathbb{C}}\mathbb{C}_{\lambda})^{x_{i}},\quad
	\text{ and }\quad
	\mathfrak{LM}_{\lambda}^{\mathrm{df}}(E)^{\mathbb{L}(\mathcal{A}_{\parallel Y}^{c})}
\]
correspond to the subspaces $K$ and $L$ of $E\otimes_{\mathbb{C}}\mathbb{C}\cdot \mathcal{A}_{\parallel Y}^{c}$ respectively
under the equivalence in Proposition \ref{prop:equivPfmod}.
Therefore the exactness of $\mathfrak{LM}_{\lambda}^{\mathrm{df}}(E)$ and $c_{\lambda}$
gives us the desired commutative diagram in the theorem.
\end{proof}	
\section{Middle convolution for local systems}\label{sec:RH}
The middle convolution functor is originally introduced by Katz \cite{Katz} as an endofunctor on the category of local systems on the Riemann sphere with a finite number of punctures.
In \cite{AdaHir}, a natural generalization of the middle convolution functor for local systems over $M(\mathcal{A})$ was given.
In this section, 
we recall the definition of the middle convolution functor for local systems on $M(\mathcal{A})$ and
show that there is a natural compatibility between the middle convolution 
for local systems over $M(\mathcal{A})$ and the middle convolution $\mathfrak{mc}_{\lambda}$ 
for modules over the holonomy Lie algebra of $M(\mathcal{A})$.
\subsection{Middle convolution for local systems on complements of affine hyperplane arrangements}\label{sec:MCforLoc}
Let us recall the defintion of the middle convolution functor for local systems
by following \cite{Katz} and its subseqent work \cite{AdaHir}. 

Let $\mathcal{A}$ be an affine hyperplane arrangement in $\C^{l}$ and $Y$ a line in $\C^{l}$ passing through the origin.
Let $\mathrm{Loc}(M(\mathcal{A}),\mathbb{C})$ denote the 
the category of finite dimensional $\mathbb{C}$-local systems over $M(\mathcal{A})$, i.e., the category of
locally constant sheaves of  finite dimensional $\mathbb{C}$-vector spaces over $M(\mathcal{A})$.

We now assume that $\mathcal{A}$ is $Y$-closed. Then Theorem \ref{thm:goodline} ensures that the projection 
$\pi_{Y}\colon M(\mathcal{A})\rightarrow M(p\mathcal{A}_{\parallel Y})$
is a locally trivial fibration with fiber $\mathbb{C}\setminus \{n\text{ points}\}$, where $n$ is the number of hyperplanes in $\mathcal{A}_{\parallel Y}^{c}$.
Let us take the pullback of this fibration $\pi_{Y}$ by itself,
\[
	\begin{tikzcd}
		M(\mathcal{A})\times_{\pi_{Y}}M(\mathcal{A})\arrow[r,"\mathrm{pr}_{2}"]\arrow[d,"\mathrm{pr}_{1}"]&
		M(\mathcal{A})\arrow[d,"\pi_{Y}"]\\
		M(\mathcal{A})\arrow[r,"\pi_{Y}"]&M(p\mathcal{A}_{\parallel Y})
	\end{tikzcd}.
\]
Then, since $\pi_{Y}$ is a locally trivial, 
projections $\mathrm{pr}_{i}$, $i=1,2$, are locally trivial fibrations as well.
Moreover, 
after removing the diagonal $\Delta_{M(\mathcal{A})}$ from $M(\mathcal{A})\times_{\pi_{Y}}M(\mathcal{A})$, i.e., 
for 
\[
	M(\mathcal{A})\times_{\pi_{Y}}M(\mathcal{A})\setminus \Delta_{M(\mathcal{A})}=\{(x,y)\in M(\mathcal{A})\times M(\mathcal{A})\,|\,\pi_{Y}(x)=\pi_{Y}(y), x\neq y\},
\]
restrictions of $\mathrm{pr}_{i}$, $i=1,2$, 
\[
	\mathrm{pr}_{i}\colon M(\mathcal{A})\times_{\pi_{Y}}M(\mathcal{A})\setminus \Delta_{M(\mathcal{A})}\longrightarrow M(\mathcal{A})
\]
are again locally trivial fibrations with fiber $\mathbb{C}\setminus \{n+1\text{ points}\}$,
see Proposition 2.2 in \cite{AdaHir}.
Let us fix an identification $\mathbb{C}^{l}=Y\oplus (\mathbb{C}^{l}/Y)$
as vector spaces,
and 
consider the projection $p^{Y}\colon \mathbb{C}^{l}\rightarrow Y=\mathbb{C}$
along this decomposition.
Then we also have another projecton,
\[
	p^{Y}\colon M(\mathcal{A})\times_{\pi_{Y}}M(\mathcal{A})\setminus \Delta_{M(\mathcal{A})}\longrightarrow \mathbb{C}^{\times};\quad
	(x,y)\mapsto p^{Y}(x)-p^{Y}(y).
\]

Let us take a nontrivial multiplicative character $\chi\colon \mathbb{Z}\rightarrow \mathbb{C}^{\times}$ and 
consider the associated rank $1$ local system $\mathcal{K}_{\chi}$
on $\mathbb{C}^{\times}$.
Then for a local system $\mathcal{L}\in \mathrm{Loc}(M(\mathcal{A}),\mathbb{C})$ on $M(\mathcal{A})$,
we can define the external tensor product of them through the projection maps,
$\mathrm{pr}_{1}\colon M(\mathcal{A})\times_{\pi_{Y}}M(\mathcal{A})\backslash 
\Delta\rightarrow M(\mathcal{A})$ and $p^{Y}\colon M(\mathcal{A})\times_{\pi_{Y}}M(\mathcal{A})\setminus \Delta_{M(\mathcal{A})}\longrightarrow \mathbb{C}^{\times}$.
Namely,
\[
	\mathcal{L}\boxtimes \chi:=
	\mathrm{pr}_{1}^{*}\mathcal{L}\otimes (p^{Y})^{*}\mathcal{K}_{\chi}.
\]
Then, since $\mathrm{pr}_{2}\colon M(\mathcal{A})\times_{\pi_{Y}}M(\mathcal{A})\backslash 
\Delta\rightarrow M(\mathcal{A})$
is a locally trivial fibration, 
the right derived functors 
$R^{1}\mathrm{pr}_{2\,*}, R^{1}\mathrm{pr}_{2\,!}$
preserve the category of local systems on $M(\mathcal{A})$
(cf.\,Theorem 1.9.5 in \cite{Ach}).
Thus we can consequently define the local system by
\[
	\mathrm{MC}_{\chi}(\mathcal{L}):=\mathrm{Im}(R^{1}\mathrm{pr}_{2\,!}(\mathcal{L}\boxtimes \chi)
	\rightarrow R^{1}\mathrm{pr}_{2\,*}(\mathcal{L}\boxtimes \chi))
	\in \mathrm{Loc}(M(\mathcal{A}),\mathbb{C}).
\] 
\begin{df}[Middle convolution]\label{df:MC}
	Take a nontrivial multiplicative character $\chi\colon \mathbb{Z}\rightarrow \mathbb{C}^{\times}$.
	Then the {\em middle convolution functor} with respect to $\chi$ is defined by
	\[
		\mathrm{MC}_{\chi}\colon \mathrm{Loc}(M(\mathcal{A}),\mathbb{C})
		\rightarrow \mathrm{Loc}(M(\mathcal{A}),\mathbb{C});\quad
		\mathcal{L}\mapsto \mathrm{MC}_{\chi}(\mathcal{L}).			
	\]
\end{df}

If $\mathcal{A}$ is not $Y$-closed, 
we take the $Y$-closure $\overline{\mathcal{A}}^{Y}$ of $\mathcal{A}$ and 
the inclusion $\iota_{Y}\colon M(\overline{\mathcal{A}}^{Y})\hookrightarrow M(\mathcal{A})$
induces the inverse image functor
\[
	\iota_{Y}^{*}\colon \mathrm{Loc}(M(\mathcal{A}),\mathbb{C})\longrightarrow \mathrm{Loc}(M(\overline{\mathcal{A}}^{Y}),\mathbb{C}).
\]
Then we can also consider 
the middle convolution functor $\mathrm{MC}_{\chi}$ for local systems on $M(\mathcal{A})$ as the composition
\[
\mathrm{Loc}(M(\mathcal{A}),\mathbb{C})\xrightarrow{\iota_{Y}^*}\mathrm{Loc}(M(\overline{\mathcal{A}}^{Y}),\mathbb{C})\xrightarrow{\mathrm{MC}_{\chi}}\mathrm{Loc}(M(\overline{\mathcal{A}}^{Y}),\mathbb{C}).
\]
\subsection{De Rham functor for logarithmic Pfaffian systems}
Let us denote the category of holomorphic flat connections on holomorphic vector bundles on $M(\mathcal{A})$
by $\mathrm{Conn}(M(\mathcal{A}))$.
Namely, it consists of pairs $(\mathcal{E},\nabla)$ of holomorphic vector bundles $\mathcal{E}$ on $M(\mathcal{A})$
and flat connections $\nabla\colon \mathcal{E}\rightarrow \mathcal{E}\otimes \varOmega_{M(\mathcal{A})}^{1}$,
with usual bundle morphisms preserving the connections as morphisms.

Then the {\em de Rham functor} is defined by 
\[
	\mathrm{DR}\colon \mathrm{Conn}(M(\mathcal{A}))
	\longrightarrow \mathrm{Loc}(M(\mathcal{A}),\mathbb{C});\quad 
	(\mathcal{E},\nabla)
	\longmapsto
	\mathrm{Ker}(\nabla),
\]
which gives an equivalence of categories (see Theorem 2.17 in \cite{DelEDR}).

The open embedding $j_{\mathcal{A}}\colon M(\mathcal{A})\hookrightarrow \mathbb{C}^{l}$
defines the pullback functor
\[
	j_{\mathcal{A}}^{\dagger}\colon \mathrm{Pf}(\mathrm{log}(\mathcal{A}))
	\longrightarrow 
	\mathrm{Conn}(M(\mathcal{A}));\quad 
	(E,\nabla_{A})
	\longmapsto
	(j_{\mathcal{A}}^{-1}(\mathcal{O}_{\mathbb{C}^{l}}\otimes E),\nabla_{A}).
\]
Here $j_{\mathcal{A}}^{-1}(\mathcal{O}_{\mathbb{C}^{l}}\otimes E)$
denotes the pullback bundle of 
$\mathcal{O}_{\mathbb{C}^{l}}\otimes E$ by $j_{\mathcal{A}}$.

Then by composing these functors,
we define the {\em de Rham functor for logarithmic Pfaffian systems with constant coefficients} by
\[
	\mathrm{DR}_{\mathrm{Pf}}:=\mathrm{DR}\circ j^{\dagger}_{\mathcal{A}}\colon \mathrm{Pf}(\mathrm{log}(\mathcal{A}))
	\longrightarrow 
	\mathrm{Loc}(M(\mathcal{A}),\mathbb{C}).
\]

Let us consider the inclusion $\iota_{Y}\colon M(\overline{\mathcal{A}}^{Y})\hookrightarrow M(\mathcal{A})$
and 
the corresponding pullback functor of connections (see Section 2.1 in \cite{DelEDR}),
\[
	\iota_{Y}^{-1}\colon \mathrm{Conn}(M(\mathcal{A}),\mathbb{C})\longrightarrow \mathrm{Conn}(M(\overline{\mathcal{A}}^{Y}),\mathbb{C}).
\]
Then since 
$j_{\overline{\mathcal{A}}^{Y}}=j_{\mathcal{A}}\circ \iota_{Y}$, 
we can check that the following diagram commutes:
\[\begin{tikzcd}
\mathrm{Pf}(\mathrm{log}(\mathcal{A}))\arrow[r,"j_{\mathcal{A}}^{\dagger}"]\arrow[d,"\iota_{Y}^{-1}"]&
\mathrm{Conn}(M(\mathcal{A}),\mathbb{C})\arrow[d,"\iota_{Y}^{-1}"]\\
\mathrm{Pf}(\mathrm{log}(\overline{\mathcal{A}}^{Y}))\arrow[r,"j_{\overline{\mathcal{A}}^{Y}}^{\dagger}"]&
\mathrm{Conn}(M(\overline{\mathcal{A}}^{Y}),\mathbb{C})
\end{tikzcd}.
\]
Here the left vertical arrow is the embedding functor $\iota_{Y}^{-1}$ defined in the previous section.
Therefore, the following diagram commutes as well:
\[\begin{tikzcd}
\mathrm{Pf}(\mathrm{log}(\mathcal{A}))\arrow[r,"j_{\mathcal{A}}^{\dagger}"]\arrow[d,"\iota_{Y}^{-1}"]&
\mathrm{Conn}(M(\mathcal{A}),\mathbb{C})\arrow[d,"\iota_{Y}^{-1}"]\arrow[r,"\mathrm{DR}"]&
\mathrm{Loc}(M(\mathcal{A}),\mathbb{C})\arrow[d,"\iota_{Y}^{*}"]\\
\mathrm{Pf}(\mathrm{log}(\overline{\mathcal{A}}^{Y}))\arrow[r,"j_{\overline{\mathcal{A}}^{Y}}^{\dagger}"]&
\mathrm{Conn}(M(\overline{\mathcal{A}}^{Y}),\mathbb{C})\arrow[r,"\mathrm{DR}"]&
\mathrm{Loc}(M(\overline{\mathcal{A}}^{Y}),\mathbb{C})
\end{tikzcd},
\]
namely, the de Rham functor $\mathrm{DR}_{\mathrm{Pf}}$ commutes with the functors $\iota_{Y}^{-1}$ and $\iota_{Y}^{*}$.
\subsection{Riemann-Hilbert correspondence between $\mathrm{MC}_{\chi}$ and $\mathfrak{mc}_{\lambda}$}
Combining the de Rham functor $\mathrm{DR}_{\mathrm{Pf}}\colon \mathrm{Pf}(\mathrm{log}(\mathcal{A})) \to \mathrm{Loc}(M(\mathcal{A}),\mathbb{C})$ 
with the equivalence of categories $\Xi_{\mathcal{A}}^{-1}\colon U(\mathfrak{h}(M(\mathcal{A})))\text{-}\mathbf{mod}\xrightarrow{\sim} \mathrm{Pf}(\mathrm{log}(\mathcal{A}))$,
we obtain the functor
\[
	\mathfrak{DR}_{\mathcal{A}} := \mathrm{DR}_{\mathrm{Pf}} \circ \Xi_{\mathcal{A}}^{-1}\colon U(\mathfrak{h}(M(\mathcal{A})))\text{-}\mathbf{mod} \longrightarrow \mathrm{Loc}(M(\mathcal{A}),\mathbb{C}).
\]
Then for generic objects in $U(\mathfrak{h}(M(\mathcal{A})))\text{-}\mathbf{mod}$, the functor $\mathfrak{DR}_{\mathcal{A}}$ relates 
the middle convolution $\mathfrak{mc}_{\lambda}$ to the middle convolution $\mathrm{MC}_{\chi}$.
\begin{thm}\label{thm:RHmc}
	Let $\mathcal{A}$ be an affine hyperplane arrangement in $\C^{l}$ and $Y$ a line in $\C^{l}$ passing through the origin.
	Let us take $\lambda\in K\setminus\{0\}$ and let $\chi\colon \mathbb{Z}\to \mathbb{C}^{\times}$ be the character defined by $\chi(1)=\exp(2\pi i\lambda)$. 
	Let $V$ be a $U(\mathfrak{h}(M(\mathcal{A})))$-module satisfying the following conditions:
	\begin{itemize}
		\item For any $H\in \mathcal{A}_{\parallel Y}^{c}$, the eigenvalues of $H$ on $V$ are not in $\mathbb{Z}\setminus\{0\}$.
		\item The eigenvalues of $\left(\sum_{H\in \mathcal{A}_{\parallel Y}^{c}}H\right)+\lambda$ on $V$ are not in $\mathbb{Z}\setminus\{0\}$. 
	\end{itemize}
Then there exists an isomorphism of local systems on $M(\mathcal{A})$,
\[
	\mathfrak{DR}_{\overline{\mathcal{A}}^{Y}}\circ \mathfrak{mc}_{\lambda}\circ \mathrm{Res}_{\mathfrak{h}(M(\overline{\mathcal{A}}^{Y}))}^{\mathfrak{h}(M(\mathcal{A}))}(V)\simeq \mathrm{MC}_{\chi}\circ \iota_{Y}^{*}\circ\mathfrak{DR}_{\mathcal{A}}(V).
\]
\end{thm}
\begin{proof}
Under the assumption on $V$, $\iota_{Y}^{-1}\circ \Xi^{-1}_{\mathcal{A}}(V)\in \mathrm{Pf}(\mathrm{log}(\overline{\mathcal{A}}^{Y}))$
satisfies the assumption in Theorem 0.3 in \cite{AdaHir}, which implies that there exists an isomorphism of local systems on $M(\overline{\mathcal{A}}^{Y})$,
\begin{equation}\label{eq:MCchi}
	\mathrm{DR}_{\mathrm{Pf}}\circ \mathrm{mc}_{\lambda}\circ \iota_{Y}^{-1}\circ \Xi^{-1}_{\mathcal{A}}(V)\simeq \mathrm{MC}_{\chi}\circ \mathrm{DR}_{\mathrm{Pf}}\circ \iota_{Y}^{-1}\circ \Xi^{-1}_{\mathcal{A}}(V).
\end{equation}
As we saw in Section 4.3, in the left hand side, we have 
\[
	\mathrm{DR}_{\mathrm{Pf}}\circ \mathrm{mc}_{\lambda}\circ \iota_{Y}^{-1}\circ \Xi^{-1}_{\mathcal{A}}(V)\simeq \mathrm{DR}_{\mathrm{Pf}}\circ \mathrm{mc}_{\lambda}\circ \Xi^{-1}_{\mathcal{A}}(V).
\]
Thus by Theorem \ref{thm:commPfmod}, we moreover have 
\begin{align*}
\mathrm{DR}_{\mathrm{Pf}}\circ \mathrm{mc}_{\lambda}\circ \iota_{Y}^{-1}\circ \Xi^{-1}_{\mathcal{A}}(V)&\simeq \mathrm{DR}_{\mathrm{Pf}}\circ \mathrm{mc}_{\lambda}\circ \Xi^{-1}_{\mathcal{A}}(V)\\
&\simeq \mathrm{DR}_{\mathrm{Pf}}\circ \Xi_{\overline{\mathcal{A}}^{Y}}^{-1}\circ \Xi_{\overline{\mathcal{A}}^{Y}}\circ  \mathrm{mc}_{\lambda}\circ \Xi^{-1}_{\mathcal{A}}(V)\\
&\simeq \mathfrak{DR}_{\overline{\mathcal{A}}^{Y}}\circ \mathfrak{mc}_{\lambda}\circ \mathrm{Res}_{\mathfrak{h}(M(\overline{\mathcal{A}}^{Y}))}^{\mathfrak{h}(M(\mathcal{A}))}(V).
\end{align*}
On the other hand, in the right hand side of $\eqref{eq:MCchi}$, we have
\begin{align*}
\mathrm{MC}_{\chi}\circ \mathrm{DR}_{\mathrm{Pf}}\circ \iota_{Y}^{-1}\circ \Xi^{-1}_{\mathcal{A}}(V)
&\simeq \mathrm{MC}_{\chi}\circ \iota_{Y}^{*}\circ \mathrm{DR}_{\mathrm{Pf}}\circ \Xi^{-1}_{\mathcal{A}}(V)\\
&\simeq \mathrm{MC}_{\chi}\circ \iota_{Y}^{*}\circ \mathfrak{DR}_{\mathcal{A}}(V),
\end{align*}
since we have already seen that the de Rham functor $\mathrm{DR}_{\mathrm{Pf}}$ commutes with the functors $\iota_{Y}^{-1}$ and $\iota_{Y}^{*}$.
This completes the proof of the theorem.
\end{proof}	
\bibliography{hre} 

\begin{thebibliography}{10}

\bibitem{Ach}
Pramod~N. Achar.
\newblock {\em Perverse sheaves and applications to representation theory},
  volume 258 of {\em Mathematical Surveys and Monographs}.
\newblock American Mathematical Society, Providence, RI, 2021.

\bibitem{Ada25}
Shunya Adachi.
\newblock Middle laplace transform and middle convolution for linear pfaffian
  systems with irregular singularities.
\newblock {\em arXiv preprint}, 2502.01263, 2025.

\bibitem{AdaHir}
Shunya Adachi and Kazuki Hiroe.
\newblock Middle convolution for affine hyperplane arrangements and {Aomoto}
  realizability.
\newblock {\em arXiv preprint}, 2601.00544v4, 2026.

\bibitem{Ari10}
D.~Arinkin.
\newblock Rigid irregular connections on {$\mathbb{P}^1$}.
\newblock {\em Compos. Math.}, 146(5):1323--1338, 2010.

\bibitem{Ch73}
Kuo-tsai Chen.
\newblock Iterated integrals of differential forms and loop space homology.
\newblock {\em Ann. of Math. (2)}, 97:217--246, 1973.

\bibitem{CCX03}
Daniel~C. Cohen, Frederick~R. Cohen, and Miguel Xicot\'encatl.
\newblock Lie algebras associated to fiber-type arrangements.
\newblock {\em Int. Math. Res. Not.}, (29):1591--1621, 2003.

\bibitem{CB03}
William Crawley-Boevey.
\newblock On matrices in prescribed conjugacy classes with no common invariant
  subspace and sum zero.
\newblock {\em Duke Math. J.}, 118(2):339--352, 2003.

\bibitem{DelEDR}
Pierre Deligne.
\newblock {\em \'Equations diff\'erentielles \`a{} points singuliers
  r\'eguliers}, volume Vol. 163 of {\em Lecture Notes in Mathematics}.
\newblock Springer-Verlag, Berlin-New York, 1970.

\bibitem{DR00}
Michael Dettweiler and Stefan Reiter.
\newblock An algorithm of {K}atz and its application to the inverse {G}alois
  problem.
\newblock volume~30, pages 761--798. 2000.
\newblock Algorithmic methods in Galois theory.

\bibitem{DR07}
Michael Dettweiler and Stefan Reiter.
\newblock Middle convolution of {F}uchsian systems and the construction of
  rigid differential systems.
\newblock {\em J. Algebra}, 318(1):1--24, 2007.

\bibitem{Har1}
Yoshishige Haraoka.
\newblock Middle convolution for completely integrable systems with logarithmic
  singularities along hyperplane arrangements.
\newblock In {\em Arrangements of hyperplanes---{S}apporo 2009}, volume~62 of
  {\em Adv. Stud. Pure Math.}, pages 109--136. Math. Soc. Japan, Tokyo, 2012.

\bibitem{Harbook}
Yoshishige Haraoka.
\newblock {\em Linear differential equations in the complex domain---from
  classical theory to forefront}, volume 2271 of {\em Lecture Notes in
  Mathematics}.
\newblock Springer, Cham, [2020] \copyright 2020.

\bibitem{Katz}
Nicholas~M. Katz.
\newblock {\em Rigid local systems}, volume 139 of {\em Annals of Mathematics
  Studies}.
\newblock Princeton University Press, Princeton, NJ, 1996.

\bibitem{Kaw10}
Hiroshi Kawakami.
\newblock Generalized {O}kubo systems and the middle convolution.
\newblock {\em Int. Math. Res. Not. IMRN}, (17):3394--3421, 2010.

\bibitem{Koh83}
Toshitake Kohno.
\newblock On the holonomy {L}ie algebra and the nilpotent completion of the
  fundamental group of the complement of hypersurfaces.
\newblock {\em Nagoya Math. J.}, 92:21--37, 1983.

\bibitem{Lo94}
D.~D. Long.
\newblock Constructing representations of braid groups.
\newblock {\em Comm. Anal. Geom.}, 2(2):217--238, 1994.

\bibitem{Mag37}
Wilhelm Magnus.
\newblock \"uber {B}eziehungen zwischen h\"oheren {K}ommutatoren.
\newblock {\em J. Reine Angew. Math.}, 177:105--115, 1937.

\bibitem{OrSol80}
Peter Orlik and Louis Solomon.
\newblock Combinatorics and topology of complements of hyperplanes.
\newblock {\em Invent. Math.}, 56(2):167--189, 1980.

\bibitem{Oshi3}
Toshio Oshima.
\newblock Stable hyperplane arrangements.
\newblock {\em arXiv preprint}, 2510.11099, 2025.

\bibitem{Sou19}
Arthur Souli\'e.
\newblock The {L}ong-{M}oody construction and polynomial functors.
\newblock {\em Ann. Inst. Fourier (Grenoble)}, 69(4):1799--1856, 2019.

\bibitem{Sul77}
Dennis Sullivan.
\newblock Infinitesimal computations in topology.
\newblock {\em Inst. Hautes \'Etudes Sci. Publ. Math.}, (47):269--331, 1977.

\bibitem{Tak11}
Kouichi Takemura.
\newblock Introduction to middle convolution for differential equations with
  irregular singularities.
\newblock In {\em New trends in quantum integrable systems}, pages 393--420.
  World Sci. Publ., Hackensack, NJ, 2011.

\bibitem{Tera1}
Hiroaki Terao.
\newblock Modular elements of lattices and topological fibration.
\newblock {\em Adv. in Math.}, 62(2):135--154, 1986.

\bibitem{Wei94}
Charles~A. Weibel.
\newblock {\em An introduction to homological algebra}, volume~38 of {\em
  Cambridge Studies in Advanced Mathematics}.
\newblock Cambridge University Press, Cambridge, 1994.

\bibitem{Yam11}
Daisuke Yamakawa.
\newblock Middle convolution and {H}arnad duality.
\newblock {\em Math. Ann.}, 349(1):215--262, 2011.

\end{thebibliography}
\bibliographystyle{plain}
\end{document}